\documentclass[10pt]{amsart}
\usepackage[a4paper]{geometry}

\usepackage{graphicx} 
\usepackage{amsmath}
\usepackage{amsthm}
\usepackage{amssymb}
\usepackage{mathtools}
\usepackage{tikz-cd}
\usepackage{tikz}
\usepackage{tabularx}
\usepackage{footnote}
\usepackage[linktocpage=true]{hyperref}
\usepackage{pdflscape}
\usepackage{caption}
\usepackage{multirow} 

\usepackage{longtable}
\usepackage{etoolbox}
\usepackage[dvipsnames]{xcolor}

\usetikzlibrary{shapes,positioning}
\usetikzlibrary{arrows,chains,matrix,positioning,scopes}

 \usepackage[UKenglish]{babel}
\usepackage{graphicx}	
\usepackage{amsthm}
\usepackage{amsmath}
\usepackage{amssymb}
\usepackage{mathtools}
\usepackage{enumerate}
\usepackage{mathrsfs}
\usepackage{thmtools, thm-restate}
\usepackage[table]{xcolor}
\usepackage{tabularx, booktabs}
\usepackage{tikz}
\usepackage{pgfplots}

\usepackage{tikz-cd} 
\usetikzlibrary {arrows.meta,bending,positioning,patterns,decorations}
\usepackage{pbox}
\usepackage{tkz-graph}
\usetikzlibrary{snakes,fit,positioning,calc,shapes}
\usepackage{graphicx}
\usepackage{cleveref,tikz-cd,pbox,wrapfig}
\definecolor{amethyst}{rgb}{0.6, 0.4, 0.8}
\definecolor{atomictangerine}{rgb}{1.0, 0.6, 0.4}
\definecolor{deeppeach}{rgb}{1.0, 0.8, 0.64}
\definecolor{eggshell}{rgb}{0.94, 0.92, 0.84}
\definecolor{lightapricot}{rgb}{0.99, 0.84, 0.69}
\definecolor{lemonchiffon}{rgb}{1.0, 0.98, 0.8}
\definecolor{roundabout}{rgb}{1.0, 0.91, 0.75}
\definecolor{atomictangerine}{rgb}{1.0, 0.6, 0.4}
\definecolor{ruby}{rgb}{0.88, 0.07, 0.37}
\definecolor{sapphire}{rgb}{0.03, 0.15, 0.4}

\def\rootsep{0.03}               
\def\clustersep{0.06}            
\def\cnamescale{0.4}             
\def\cdepthscale{0.4}            
\def\cltopskip{1pt}              
\def\clbottomskip{1pt}           

\def\rootscale{0.5}   \def\rootcolor{gray}
\def\rootscaleA{0.7}  \def\rootcolorA{yellow}
\def\rootscaleB{0.5}  \def\rootcolorB{green}
\def\rootscaleC{0.4}  \def\rootcolorC{sapphire}
\def\rootscaleD{0.45}  \def\rootcolorD{ruby}
\tikzset{
  clA/.style = {very thick,black},
  clB/.style = {thick,purple}
}

\def\graphdslabelscale{0.6}

\def\GraphScale{0.6}

\tikzset{
  root/.style = {circle,scale=\rootscale,fill=\rootcolor},
    rc/.style 2 args = {right=#1*1.5*\clustersep of {#2.east|-first},root}, rr/.style = {right=\rootsep of {#1.east|-first},root},
  roott/.style = {circle,inner sep=-2pt,minimum size=5pt,black,font=\ttfamily\footnotesize},
    rct/.style 2 args = {right=#1*1.5*\clustersep of {#2.east|-first},roott}, rrt/.style = {right=\rootsep of {#1.east|-first},roott},
  rootA/.style = {circle,scale=\rootscaleA,ball color=\rootcolorA},
    rcA/.style 2 args = {right=#1*1.5*\clustersep of {#2.east|-first},rootA}, rrA/.style = {right=\rootsep of {#1.east|-first},rootA},
  rootB/.style = {circle,scale=\rootscaleB,ball color=\rootcolorB},
    rcB/.style 2 args = {right=#1*1.5*\clustersep of {#2.east|-first},rootB}, rrB/.style = {right=\rootsep of {#1.east|-first},rootB},
  rootC/.style = {diamond,scale=\rootscaleC,ball color=\rootcolorC},
    rcC/.style 2 args = {right=#1*1.5*\clustersep of {#2.east|-first},rootC}, rrC/.style = {right=\rootsep of {#1.east|-first},rootC},
  rootD/.style = {circle,scale=\rootscaleD,ball color=\rootcolorD},
    rcD/.style 2 args = {right=#1*1.5*\clustersep of {#2.east|-first},rootD}, rrD/.style = {right=\rootsep of {#1.east|-first},rootD},
  cluster/.style = {draw=black!90,thick,rounded corners,inner sep=22*\clustersep,outer xsep=22*\clustersep,fit=#1},
  clabel/.style  = {anchor=west,scale=\cdepthscale,black,inner sep=0,outer xsep=1,outer ysep=0},
  clabelL/.style = {above right=-\clustersep of #1t.north east,clabel},
  clabelD/.style = {below right=-\clustersep of #1t.south east,clabel},
  clouter/.style = {inner sep=0,outer sep=0,fit=#1}
}

\def\Cluster #1 = #2;{\node[cluster=#2] (#1) {};}
\def\ClusterL #1[#2] = #3;{
  \node[cluster=#3] (#1t) {}; \node[clabelL=#1] (#1l) {$#2$}; \node[clouter=(#1t)(#1l)] (#1) {};}
\def\ClusterD #1[#2] = #3;{
  \node[cluster=#3] (#1t) {}; \node[clabelD=#1] (#1d) {$#2$}; \node[clouter=(#1t)(#1d)] (#1) {};}
\def\ClusterLD #1[#2][#3] = #4;{
  \node[cluster=#4] (#1t) {}; \node[clabelL=#1] (#1l) {$#2$}; 
  \node[clabelD=#1] (#1d) {$#3$}; \node[clouter=(#1t)(#1l)(#1d)] (#1) {};}
\def\ClusterLDName #1[#2][#3][#4] = #5;{
  \node[cluster=#5] (#1t) {}; \node[clabelL=#1] (#1l) {$#2$}; 
  \node[clabelD=#1] (#1d) {$#3$}; 
  \node[scale=\cnamescale,above=\clustersep/3 of #1t,inner sep=0, outer sep=0] (#1n) {$#4$}; 
  \node[clouter=(#1l)(#1d)(#1t)] (#1) {};}

\newcommand{\Root}[4][]{
  \ifx\relax#2\relax\node[rr#1=#3] (#4) {};\else\node[rc#1={#2}{#3}] (#4) {};\fi}
\newcommand{\RootT}[5][]{
  \ifx\relax#2\relax\node[rrt#1=#3] (#4) {#5};\else\node[rct#1={#2}{#3}] (#4) {#5};\fi}

\def\frob(#1)(#2){\path[draw,thick,shorten <=-22*\clustersep,shorten >=-22*\clustersep](#1.east)--(#2.west|-#1){};}

\usepackage{pbox}
\def\pb#1{\pbox[c]{\textwidth}{\hfil #1\hfil}}

\long\def\clusterpicture#1\endclusterpicture{\pb{\vbox to \cltopskip{\vfill}\\%
  \begin{tikzpicture}\node[coordinate] (first) {};#1\end{tikzpicture}\\[-11pt]\vbox to \clbottomskip{\vfill}}}   
\long\def\clusterpictureopt#1#2\endclusterpicture{\pb{\vbox to \cltopskip{\vfill}\\%
  \begin{tikzpicture}[#1]\node[coordinate] (first) {};#2\end{tikzpicture}\\[-11pt]\vbox to \clbottomskip{\vfill}}}

\def\pb#1{\pbox[c]{\textwidth}{\hfil #1\hfil}}

\def\GraphVertices{\SetVertexNormal[Shape=circle, FillColor=blue!50, LineColor=blue!50, LineWidth=0.8pt]
  \tikzset{VertexStyle/.append style = {inner sep=0.5pt,minimum size=0.3em,font = \tiny\bfseries}}}

\def\BlueEdges{  \SetUpEdge[lw=0.8pt,color=blue!70]
   \tikzset{EdgeStyle/.append style = {shorten <=0.5pt,shorten >=0.5pt}}}

\def\LoopW(#1){
  \path[draw,-,thick,color=blue!70] (#1) edge[out=155,in=90] ($(#1)-(1.3,0)$);
  \path[draw,-,thick,color=blue!70] (#1) edge[out=210,in=270] ($(#1)-(1.3,0)$);
}
\def\LoopE(#1){
  \path[draw,-,thick,color=blue!70] (#1) edge[out=25,in=90] ($(#1)+(1.3,0)$);
  \path[draw,-,thick,color=blue!70] (#1) edge[out=-25,in=270] ($(#1)+(1.3,0)$);
}
\def\LoopS(#1){
  \path[draw,-,thick,color=blue!70] (#1) edge[out=115,in=180] ($(#1)+(0,1.2)$);
  \path[draw,-,thick,color=blue!70] (#1) edge[out=65,in=0] ($(#1)+(0,1.2)$);
}
\def\LoopN(#1){
  \path[draw,-,thick,color=blue!70] (#1) edge[out=-115,in=180] ($(#1)-(0,1.2)$);
  \path[draw,-,thick,color=blue!70] (#1) edge[out=-65,in=0] ($(#1)-(0,1.2)$);
}
\def\EdgeW(#1){
  \path[draw,-,thick,color=blue!70] (#1+) edge[out=180,in=90] ($(#1+)-(1.3,0.3)$);
  \path[draw,-,thick,color=blue!70] (#1-) edge[out=180,in=270] ($(#1+)-(1.3,0.3)$);
}
\def\EdgeE(#1){
  \path[draw,-,thick,color=blue!70] (#1+) edge[out=0,in=90] ($(#1-)+(1.3,0.3)$);
  \path[draw,-,thick,color=blue!70] (#1-) edge[out=0,in=270] ($(#1-)+(1.3,0.3)$);
}
\def\EdgeS(#1){
  \path[draw,-,thick,color=blue!70] (#1+) edge[out=90,in=0] ($(#1-)+(0.3,1.3)$);
  \path[draw,-,thick,color=blue!70] (#1-) edge[out=90,in=180] ($(#1-)+(0.3,1.3)$);
}
\def\EdgeN(#1){
  \path[draw,-,thick,color=blue!70] (#1+) edge[out=270,in=0] ($(#1+)-(0.3,1.3)$);
  \path[draw,-,thick,color=blue!70] (#1-) edge[out=270,in=180] ($(#1+)-(0.3,1.3)$);
}
\def\GCircle(#1,#2)(#3,#4){
  \path(#1,#2) node[coordinate] (1) {};
  \path(#3,#4) node[coordinate] (2) {};
  \path[draw,-,thick,color=blue!70] (1) edge[out=90,in=90] (2);
  \path[draw,-,thick,color=blue!70] (2) edge[out=270,in=270] (1);
}

\def\EdgeSign(#1)(#2)#3(#4)#5{
  \node at ($(#1)!#3!(#2) + (#4)$) [color=black, scale=\graphdslabelscale] {$\scriptstyle #5$};
}

\def\GraphEdgeSignDist{0.55}

\def\GraphEdgeSignS(#1)(#2)#3#4{\EdgeSign(#1)(#2)#3(0,-\GraphEdgeSignDist){#4}}

\def\VSwap#1#2#3#4{\path[draw](#1) edge[<->,#3,shorten >=#4pt,shorten <=#4pt] (#2){};}
\def\VArr#1#2#3#4{\path[draw](#1) edge[->,#3,shorten >=#4pt,shorten <=#4pt] (#2){};}

\def\ESwapOfs#1#2#3#4#5#6#7#8{\VSwap{$(#1)!0.5!(#2) + (#6)$}{$(#3)!0.5!(#4) + (#7)$}{#5}{#8}}

\def\EArrOfs#1#2#3#4#5#6#7#8{\VArr{$(#1)!0.5!(#2) + (#6)$}{$(#3)!0.5!(#4) + (#7)$}{#5}{#8}}

\def\tgrGB{\raise-7pt\hbox{\begin{tikzpicture}[scale=\GraphScale]
  \GraphVertices
  \Vertex[x=1.50,y=0.000,L=1]{1};
  \coordinate (2) at (0.000,0.000);
  \BlueEdges
  \LoopW(1)
\GraphEdgeSignS(1)(2){0.5}{n}\end{tikzpicture}}}

\def\tgrGBex{\raise-7pt\hbox{\begin{tikzpicture}[scale=\GraphScale]
  \GraphVertices
  \Vertex[x=1.50,y=0.000,L=1]{1};
  \coordinate (2) at (0.000,0.000);
  \BlueEdges
  \LoopW(1)
\GraphEdgeSignS(1)(2){0.5}{1}\end{tikzpicture}}}

\def\tgrGC{\raise-7pt\hbox{\begin{tikzpicture}[scale=\GraphScale]
  \GraphVertices
  \Vertex[x=1.50,y=0.000,L=1]{1};
  \coordinate (2) at (0.000,0.000);
  \BlueEdges
  \LoopW(1)
\GraphEdgeSignS(1)(2){0.5}{n}\ESwapOfs1212{}{0,-0.25}{0,0.25}{0.5}\end{tikzpicture}}}

\def\tgrGD{\raise-7pt\hbox{\begin{tikzpicture}[scale=\GraphScale]
  \GraphVertices
  \Vertex[x=1.50,y=0.000,L=\relax]{1};
  \coordinate (2) at (3.00,0.000);
  \coordinate (3) at (0.000,0.000);
  \BlueEdges
  \LoopE(1)
  \LoopW(1)
\GraphEdgeSignS(1)(3){0.5}{n}\GraphEdgeSignS(1)(2){0.5}{n}\end{tikzpicture}}}

\def\tgrGE{\raise-7pt\hbox{\begin{tikzpicture}[scale=\GraphScale]
  \GraphVertices
  \Vertex[x=1.50,y=0.000,L=\relax]{1};
  \coordinate (2) at (3.00,0.000);
  \coordinate (3) at (0.000,0.000);
  \BlueEdges
  \LoopE(1)
  \LoopW(1)
\GraphEdgeSignS(1)(3){0.5}{n}\GraphEdgeSignS(1)(2){0.5}{n}\ESwapOfs1212{}{0,-0.25}{0,0.25}{0.5}\end{tikzpicture}}}

\def\tgrGF{\raise-7pt\hbox{\begin{tikzpicture}[scale=\GraphScale]
  \GraphVertices
  \Vertex[x=1.50,y=0.000,L=\relax]{1};
  \coordinate (2) at (3.00,0.000);
  \coordinate (3) at (0.000,0.000);
  \BlueEdges
  \LoopE(1)
  \LoopW(1)
\GraphEdgeSignS(1)(3){0.5}{n}\GraphEdgeSignS(1)(2){0.5}{n}\ESwapOfs1313{}{0,-0.25}{0,0.25}{0.5}\ESwapOfs1212{}{0,-0.25}{0,0.25}{0.5}\end{tikzpicture}}}

\def\tgrGG{\raise-7pt\hbox{\begin{tikzpicture}[scale=\GraphScale]
  \GraphVertices
  \Vertex[x=1.50,y=0.000,L=\relax]{1};
  \coordinate (2) at (3.00,0.000);
  \coordinate (3) at (0.000,0.000);
  \BlueEdges
  \LoopE(1)
  \LoopW(1)
\GraphEdgeSignS(1)(3){0.5}{n}\GraphEdgeSignS(1)(2){0.5}{n}\ESwapOfs1312{in=160,out=20}{0.2,0.3}{-0.2,0.3}{0.5}\end{tikzpicture}}}

\def\tgrGH{\raise-7pt\hbox{\begin{tikzpicture}[scale=\GraphScale]
  \GraphVertices
  \Vertex[x=1.50,y=0.000,L=\relax]{1};
  \coordinate (2) at (3.00,0.000);
  \coordinate (3) at (0.000,0.000);
  \BlueEdges
  \LoopE(1)
  \LoopW(1)
\GraphEdgeSignS(1)(3){0.5}{n}\GraphEdgeSignS(1)(2){0.5}{n}\EArrOfs1312{in=150,out=30}{0.1,0.29}{0,0.35}{0.5}\EArrOfs1213{in=-60,out=-60}{0,0.2}{0.3,-0.25}{0.5}\end{tikzpicture}}}

\def\tgrGA{\raise-3pt\hbox{\begin{tikzpicture}[scale=\GraphScale]
  \GraphVertices
  \Vertex[x=0.000,y=0.000,L=2]{1};
  \BlueEdges
\end{tikzpicture}}}

\makeatletter
\tikzset{join/.code=\tikzset{after node path={%
\ifx\tikzchainprevious\pgfutil@empty\else(\tikzchainprevious)%
edge[every join]#1(\tikzchaincurrent)\fi}}}
\makeatother

\tikzset{>=stealth',every on chain/.append style={join},
         every join/.style={->}}

\def\pLBpgA{{\scalebox{1.5}{
\clusterpicture           
  \Root[D] {1} {first} {r1};
  \Root[D] {} {r1} {r2};
  \Root[D] {} {r2} {r3};
  \Root[D] {} {r3} {r4};
    \Root[D] {} {r4} {r5};
    \Root[D] {} {r5} {r6};
  \ClusterD c4[] = (r1)(r2)(r3)(r4)(r5)(r6);
\endclusterpicture}}}

\def\pLBpgB{{\scalebox{1.5}{
\clusterpicture           
  \Root[D] {1} {first} {r1};
  \Root[D] {} {r1} {r2};
  \Root[D] {} {r2} {r3};
  \Root[D] {} {r3} {r4};
    \Root[D] {} {r4} {r5};
    \Root[D] {1} {r5} {r6};
  \ClusterLD c1[][] = (r1)(r2)(r3)(r4)(r5);
  \ClusterD c4[] = (c1)(r6);
\endclusterpicture}}}

\def\pLBemA{{
\scalebox{1.5}{
\clusterpicture           
  \Root[D] {1} {first} {r1};
  \Root[D] {} {r1} {r2};
  \Root[D] {1} {r2} {r3};
  \Root[D] {2} {r3} {r4};
    \Root[D] {} {r4} {r5};
    \Root[D] {} {r5} {r6};
  \ClusterLD c1[][] = (r1)(r2);
  \ClusterLD c2[][] = (c1)(r3);
  \ClusterLD c3[][] = (r4)(r5)(r6);
  \ClusterD c4[] = (c2)(c3);
\endclusterpicture}
}}

\def\pLBemB{{
\scalebox{1.5}{
\clusterpicture           
  \Root[D] {1} {first} {r1};
  \Root[D] {1} {r1} {r2};
  \Root[D] {} {r2} {r3};
  \Root[D] {1.5} {r3} {r4};
    \Root[D] {} {r4} {r5};
    \Root[D] {} {r5} {r6};
  \ClusterLD c1[][] = (r2)(r3);
  \ClusterLD c3[][] = (r4)(r5)(r6);
  \ClusterD c4[] = (r1)(c1)(c3);
\endclusterpicture}
}}

\def\pLBemC{{
\scalebox{1.5}{
\clusterpicture           
  \Root[D] {1} {first} {r1};
  \Root[D] {} {r1} {r2};
  \Root[D] {} {r2} {r3};
  \Root[D] {1} {r3} {r4};
    \Root[D] {1} {r4} {r5};
    \Root[D] {} {r5} {r6};
  \ClusterLD c1[][] = (r5)(r6);
  \ClusterLD c3[][] = (r4)(c1);R
  \ClusterD c4[] = (r1)(c1)(c3);
\endclusterpicture}
}}

\def\pLBemD{{
\scalebox{1.5}{
\clusterpicture           
  \Root[D] {1} {first} {r1};
  \Root[D] {} {r1} {r2};
  \Root[D] {1} {r2} {r3};
  \Root[D] {1} {r3} {r4};
    \Root[D] {} {r4} {r5};
    \Root[D] {} {r5} {r6};
  \ClusterLD c1[][] = (r4)(r5)(r6);
  \ClusterLD c2[][] = (r3)(c1);
  \ClusterD c4[] = (r1)(r2)(c2);
\endclusterpicture}
}}

\def\pLBemE{{
\scalebox{1.5}{
\clusterpicture           
  \Root[D] {1} {first} {r1};
  \Root[D] {} {r1} {r2};
  \Root[D] {2} {r2} {r3};
  \Root[D] {1} {r3} {r4};
    \Root[D] {} {r4} {r5};
    \Root[D] {} {r5} {r6};
  \ClusterLD c1[][] = (r4)(r5)(r6);
  \ClusterLD c2[][] = (r3)(c1);
  \ClusterLD c3[][] = (r1)(r2);
  \ClusterD c4[] = (c3)(c2);
\endclusterpicture}
}}

\def\pLBemFa{{
\scalebox{1.5}{
\clusterpicture           
  \Root[D] {1} {first} {r1};
  \Root[D] {1} {r1} {r2};
  \Root[D] {} {r2} {r3};
  \Root[D] {1} {r3} {r4};
    \Root[D] {1} {r4} {r5};
    \Root[D] {} {r5} {r6};
  \ClusterLD c1[][] = (r5)(r6);
  \ClusterLD c2[][] = (r4)(c1);
  \ClusterLD c3[][] = (r2)(r3)(c2);
  \ClusterD c4[] = (r1)(c3);
\endclusterpicture}
}}

\def\pLBemG{{
\scalebox{1.5}{
\clusterpicture           
  \Root[D] {1} {first} {r1};
  \Root[D] {2} {r1} {r2};
  \Root[D] {} {r2} {r3};
  \Root[D] {1.5} {r3} {r4};
    \Root[D] {} {r4} {r5};
    \Root[D] {} {r5} {r6};
  \ClusterLD c1[][] = (r2)(r3);
  \ClusterLD c2[][] = (r4)(r5)(r6);
  \ClusterLD c3[][] = (c1)(c2);
  \ClusterD c4[] = (r1)(c3);
\endclusterpicture}
}}

\def\pLBemH{{
\scalebox{1.5}{
\clusterpicture           
  \Root[D] {1} {first} {r1};
  \Root[D] {1} {r1} {r2};
  \Root[D] {1} {r2} {r3};
  \Root[D] {1} {r3} {r4};
    \Root[D] {} {r4} {r5};
    \Root[D] {} {r5} {r6};
  \ClusterLD c1[][] = (r4)(r5)(r6);
  \ClusterLD c2[][] = (r3)(c1);
  \ClusterLD c3[][] = (r2)(c2);
  \ClusterD c4[] = (r1)(c3);
\endclusterpicture}
}}

\def\pLBeeB{{
\scalebox{1.5}{
\clusterpicture           
  \Root[D] {} {first} {r1};
  \Root[D] {} {r1} {r2};
  \Root[D] {} {r2} {r3};
  \Root[D] {2} {r3} {r4};
    \Root[D] {} {r4} {r5};
    \Root[D] {} {r5} {r6};
  \ClusterLD c2[][] = (r1)(r2)(r3);
  \ClusterLD c3[][] = (r4)(r5)(r6);
  \ClusterD c4[] = (c2)(c3);
\endclusterpicture}
}}

\def\pLBeeCb{{
\scalebox{1.5}{
\clusterpicture           
  \Root[D] {} {first} {r1};
  \Root[D] {} {r1} {r2};
  \Root[D] {} {r2} {r3};
  \Root[D] {1} {r3} {r4};
    \Root[D] {} {r4} {r5};
    \Root[D] {} {r5} {r6};
  \ClusterLD c3[][] = (r4)(r5)(r6);
  \ClusterD c4[] = (r1)(r2)(r3)(c3);
\endclusterpicture}
}}

\def\pLBeeDb{{
\scalebox{1.5}{
\clusterpicture           
  \Root[D] {1} {first} {r1};
  \Root[D] {1} {r1} {r2};
  \Root[D] {} {r2} {r3};
  \Root[D] {1} {r3} {r4};
    \Root[D] {} {r4} {r5};
    \Root[D] {} {r5} {r6};
  \ClusterLD c1[][] = (r4)(r5)(r6);
  \ClusterLD c2[][] = (r2)(r3)(c1);
  \ClusterD c4[] = (r1)(c2);
\endclusterpicture}
}}

\def\pLBonA{{
\scalebox{1.5}{
\clusterpicture           
  \Root[D] {1} {first} {r1};
  \Root[D] {} {r1} {r2};
  \Root[D] {1} {r2} {r3};
  \Root[D] {} {r3} {r4};
    \Root[D] {} {r4} {r5};
    \Root[D] {} {r5} {r6};
  \ClusterLD c1[][] = (r1)(r2);
  \ClusterLD c2[][] = (c1)(r3)(r4)(r5)(r6);
\endclusterpicture}
}}

\def\pLBonB{{
\scalebox{1.5}{
\clusterpicture           
  \Root[D] {} {first} {r1};
  \Root[D] {} {r1} {r2};
  \Root[D] {1} {r2} {r3};
  \Root[D] {} {r3} {r4};
    \Root[D] {} {r4} {r5};
    \Root[D] {} {r5} {r6};
  \ClusterLD c1[][] = (r3)(r4)(r5)(r6);
  \ClusterLD c2[][] = (r1)(r2)(c1);
\endclusterpicture}
}}

\def\pLBonC{{
\scalebox{1.5}{
\clusterpicture           
  \Root[D] {} {first} {r1};
  \Root[D] {} {r1} {r2};
  \Root[D] {2} {r2} {r3};
  \Root[D] {} {r3} {r4};
    \Root[D] {} {r4} {r5};
    \Root[D] {} {r5} {r6};
  \ClusterLD c1[][] = (r1)(r2); 
  \ClusterLD c2[][] = (r3)(r4)(r5)(r6);
  \ClusterLD c3[][] = (c1)(c2);
\endclusterpicture}
}}

\def\pLBonD{{
\scalebox{1.5}{
\clusterpicture           
  \Root[D] {} {first} {r1};
  \Root[D] {1} {r1} {r2};
  \Root[D] {} {r2} {r3};
  \Root[D] {} {r3} {r4};
    \Root[D] {1} {r4} {r5};
    \Root[D] {} {r5} {r6};
  \ClusterLD c1[][] = (r5)(r6); 
  \ClusterLD c2[][] = (r2)(r3)(r4)(c1);
  \ClusterLD c3[][] = (r1)(c2);
\endclusterpicture}
}}

\def\pLBonE{{
\scalebox{1.5}{
\clusterpicture           
  \Root[D] {} {first} {r1};
  \Root[D] {1} {r1} {r2};
  \Root[D] {1} {r2} {r3};
  \Root[D] {} {r3} {r4};
    \Root[D] {} {r4} {r5};
    \Root[D] {} {r5} {r6};
  \ClusterLD c1[][] = (r3)(r4)(r5)(r6); 
  \ClusterLD c2[][] = (r2)(c1);
  \ClusterLD c3[][] = (r1)(c2);
\endclusterpicture}
}}

\def\pLBpmpmA{{
\scalebox{1.5}{
\clusterpicture           
  \Root[D] {1} {first} {r1};
  \Root[D] {} {r1} {r2};
  \Root[D] {1} {r2} {r3};
  \Root[D] {3} {r3} {r4};
    \Root[D] {} {r4} {r5};
    \Root[D] {1} {r5} {r6};
  \ClusterLD c1[][] = (r1)(r2);
  \ClusterLD c2[][] = (c1)(r3);
  \ClusterLD c3[][] = (r4)(r5);
  \ClusterLD c4[][] = (c3)(r6);
  \ClusterLD c5[][] = (c2)(c4);
\endclusterpicture}
}}

\def\pLBpmpmC{{
\scalebox{1.5}{
\clusterpicture           
  \Root[D] {1} {first} {r1};
  \Root[D] {} {r1} {r2};
  \Root[D] {1} {r2} {r3};
  \Root[D] {1.5} {r3} {r4};
    \Root[D] {} {r4} {r5};
    \Root[D] {1} {r5} {r6};
  \ClusterLD c1[][] = (r1)(r2);
  \ClusterLD c2[][] = (c1)(r3);
  \ClusterLD c3[][] = (r4)(r5);
  \ClusterLD c5[][] = (c2)(c3)(r6);
\endclusterpicture}
}}

\def\pLBpmpmD{{
\scalebox{1.5}{
\clusterpicture           
  \Root[D] {1} {first} {r1};
  \Root[D] {} {r1} {r2};
  \Root[D] {1} {r2} {r3};
  \Root[D] {1} {r3} {r4};
    \Root[D] {1} {r4} {r5};
    \Root[D] {} {r5} {r6};
  \ClusterLD c1[][] = (r1)(r2);
  \ClusterLD c2[][] = (c1)(r3);
  \ClusterLD c3[][] = (c2)(r4);
  \ClusterLD c5[][] = (c3)(r5)(r6);
\endclusterpicture}
}}

\def\pLBpmpmE{{
\scalebox{1.5}{
\clusterpicture           
  \Root[D] {1} {first} {r1};
  \Root[D] {} {r1} {r2};
  \Root[D] {1} {r2} {r3};
  \Root[D] {1} {r3} {r4};
    \Root[D] {1} {r4} {r5};
    \Root[D] {1} {r5} {r6};
  \ClusterLD c1[][] = (r1)(r2);
  \ClusterLD c2[][] = (c1)(r3);
  \ClusterLD c3[][] = (c2)(r4);
  \ClusterLD c4[][] = (c3)(r5);
  \ClusterLD c5[][] = (c4)(r6);
\endclusterpicture}
}}

\def\pLBpmpmF{{
\scalebox{1.5}{
\clusterpicture           
  \Root[D] {1} {first} {r1};
  \Root[D] {} {r1} {r2};
  \Root[D] {1} {r2} {r3};
  \Root[D] {2} {r3} {r4};
    \Root[D] {} {r4} {r5};
    \Root[D] {2} {r5} {r6};
  \ClusterLD c1[][] = (r1)(r2);
  \ClusterLD c2[][] = (c1)(r3);
  \ClusterLD c3[][] = (r4)(r5);
  \ClusterLD c4[][] = (c2)(c3)(c3);
  \ClusterLD c5[][] = (c4)(r6);
\endclusterpicture}
}}

\def\pLBpmpmG{{
\scalebox{1.5}{
\clusterpicture           
  \Root[D] {1} {first} {r1};
  \Root[D] {} {r1} {r2};
  \Root[D] {1} {r2} {r3};
  \Root[D] {1} {r3} {r4};
    \Root[D] {2} {r4} {r5};
    \Root[D] {} {r5} {r6};
  \ClusterLD c1[][] = (r1)(r2);
  \ClusterLD c2[][] = (c1)(r3);
  \ClusterLD c3[][] = (c2)(r4);
  \ClusterLD c5[][] = (r5)(r6);
  \ClusterLD c6[][] = (c3)(c5);
\endclusterpicture}
}}

\def\pLBtwonAa{{
\scalebox{1.5}{
\clusterpicture           
  \Root[D] {} {first} {r1};
  \Root[D] {} {r1} {r2};
  \Root[D] {2} {r2} {r3};
  \Root[D] {} {r3} {r4};
    \Root[D] {1} {r4} {r5};
    \Root[D] {} {r5} {r6};
  \ClusterLD c1[][] = (r1)(r2); 
  \ClusterLD c2[][] = (r3)(r4);
  \ClusterLD c3[][] = (c1)(c2)(r5)(r6);
\endclusterpicture}
}}

\def\pLBtwonBa{{
\scalebox{1.5}{
\clusterpicture           
  \Root[D] {} {first} {r1};
  \Root[D] {} {r1} {r2};
  \Root[D] {1} {r2} {r3};
  \Root[D] {} {r3} {r4};
    \Root[D] {2} {r4} {r5};
    \Root[D] {} {r5} {r6};
  \ClusterLD c1[][] = (r3)(r4); 
  \ClusterLD c2[][] = (r1)(r2)(c1);
  \ClusterLD c3[][] = (c1)(c2)(r5)(r6);
\endclusterpicture}
}}

\def\pLBtwonBb{{
\scalebox{1.5}{
\clusterpicture           
  \Root[D] {} {first} {r1};
  \Root[D] {} {r1} {r2};
  \Root[D] {1} {r2} {r3};
  \Root[D] {} {r3} {r4};
    \Root[D] {2} {r4} {r5};
    \Root[D] {} {r5} {r6};
  \ClusterLD c1[][] = (r3)(r4); 
  \ClusterLD c2[][] = (r1)(r2)(c1);
  \ClusterLD c3[][] = (c1)(c2)(r5)(r6);
\endclusterpicture}
}}

\def\pLBtwonCa{{
\scalebox{1.5}{
\clusterpicture           
  \Root[D] {} {first} {r1};
  \Root[D] {} {r1} {r2};
  \Root[D] {2} {r2} {r3};
  \Root[D] {} {r3} {r4};
    \Root[D] {1} {r4} {r5};
    \Root[D] {} {r5} {r6};
  \ClusterLD c1[][] = (r1)(r2); 
  \ClusterLD c2[][] = (r5)(r6);
  \ClusterLD c3[][] = (r3)(r4)(c2);
   \ClusterLD c4[][] = (c1)(c3);
\endclusterpicture}
}}

\def\pLBtwonDa{{
\scalebox{1.5}{
\clusterpicture           
  \Root[D] {} {first} {r1};
  \Root[D] {1} {r1} {r2};
  \Root[D] {1} {r2} {r3};
  \Root[D] {} {r3} {r4};
    \Root[D] {2} {r4} {r5};
    \Root[D] {} {r5} {r6};
  \ClusterLD c1[][] = (r3)(r4); 
  \ClusterLD c2[][] = (r5)(r6);
  \ClusterLD c3[][] = (r2)(c1)(c2);
   \ClusterLD c4[][] = (r1)(c3);
\endclusterpicture}
}}

\def\pLBtwonE{{
\scalebox{1.5}{
\clusterpicture           
  \Root[D] {} {first} {r1};
  \Root[D] {1} {r1} {r2};
  \Root[D] {1} {r2} {r3};
  \Root[D] {} {r3} {r4};
    \Root[D] {1} {r4} {r5};
    \Root[D] {} {r5} {r6};
  \ClusterLD c1[][] = (r5)(r6); 
  \ClusterLD c2[][] = (r3)(r4)(c1);
  \ClusterLD c3[][] = (r2)(c2);
   \ClusterLD c4[][] = (r1)(c3);
\endclusterpicture}
}}

\def\pLBthreenAa{{
\scalebox{1.5}{
\clusterpicture           
  \Root[D] {1} {first} {r1};
  \Root[D] {} {r1} {r2};
  \Root[D] {2} {r2} {r3};
  \Root[D] {} {r3} {r4};
    \Root[D] {2} {r4} {r5};
    \Root[D] {} {r5} {r6};
  \ClusterLD c1[][] = (r1)(r2); 
  \ClusterLD c2[][] = (r3)(r4);
  \ClusterLD c3[][] = (r5)(r6);
    \ClusterLD c4[][] = (c1)(c2)(c3);
\endclusterpicture}
}}

\def\pLBthreenBa{{
\scalebox{1.5}{
\clusterpicture           
  \Root[D] {} {first} {r1};
  \Root[D] {} {r1} {r2};
  \Root[D] {2} {r2} {r3};
  \Root[D] {} {r3} {r4};
    \Root[D] {1.5} {r4} {r5};
    \Root[D] {} {r5} {r6};
  \ClusterLD c1[][] = (r3)(r4); 
  \ClusterLD c2[][] = (r5)(r6);
  \ClusterLD c3[][] = (c1)(c2);
  \ClusterLD c4[][] = (r1)(r2)(c3);
\endclusterpicture}
}}

\def\pLBthreenCa{{
\scalebox{1.5}{
\clusterpicture           
  \Root[D] {} {first} {r1};
  \Root[D] {} {r1} {r2};
  \Root[D] {3} {r2} {r3};
  \Root[D] {} {r3} {r4};
    \Root[D] {2.5} {r4} {r5};
    \Root[D] {} {r5} {r6};
  \ClusterLD c1[][] = (r3)(r4); 
  \ClusterLD c2[][] = (r5)(r6);
  \ClusterLD c3[][] = (c1)(c2);
  \ClusterLD c4[][] = (r1)(r2);
  \ClusterLD c4[][] = (c4)(c3);
\endclusterpicture}
}}

\def\pLBthreenDa{{
\scalebox{1.5}{
\clusterpicture           
  \Root[D] {} {first} {r1};
  \Root[D] {1} {r1} {r2};
  \Root[D] {2} {r2} {r3};
  \Root[D] {} {r3} {r4};
    \Root[D] {1.5} {r4} {r5};
    \Root[D] {} {r5} {r6};
  \ClusterLD c1[][] = (r3)(r4); 
  \ClusterLD c2[][] = (r5)(r6);
  \ClusterLD c3[][] = (c1)(c2);
  \ClusterLD c3[][] = (r2)(c3);  
   \ClusterLD c4[][] = (r1)(c3);
\endclusterpicture}
}}

\newtheorem{theorem}{Theorem}[section]
\newtheorem{corollary}[theorem]{Corollary}
\newtheorem{lemma}[theorem]{Lemma}
\newtheorem{proposition}[theorem]{Proposition}

\theoremstyle{definition}
\newtheorem{defn}[theorem]{Definition}
\newtheorem{remark}[theorem]{Remark}
\newtheorem{example}[theorem]{Example}
\newtheorem{notation}[theorem]{Notation}

\def\Z{\mathbb{Z}}
\def\Q{\mathbb{Q}}
\def\C{\mathbb{C}}
\def\D{\mathrm{D}}

\def\fs{\mathfrak{s}}
\def\ft{\mathfrak{t}}

\DeclareMathOperator{\Gal}{Gal}
\DeclareMathOperator{\Aut}{Aut}

\DeclareMathOperator{\PGL}{PGL}
\DeclareMathOperator{\disc}{disc}
\DeclareMathOperator{\Jac}{Jac}

\def\cR{{\mathcal{R}}}

\def\sstar{{\mathrlap{\star}}}
\def\aast{{\mathrlap{\diamond}}}

\newcommand{\matr}[9]{\def\matrMore##1##2##3##4##5##6##7{\left(\begin{smallmatrix}#1&#2&#3&#4\\#5&#6&#7&#8\\#9&##1&##2&##3\\##4&##5&##6&##7\end{smallmatrix}\right)}\matrMore}

\def\LBpgA{{\scalebox{1.5}{
\clusterpicture           
  \Root[D] {1} {first} {r1};
  \Root[D] {} {r1} {r2};
  \Root[D] {} {r2} {r3};
  \Root[D] {} {r3} {r4};
    \Root[D] {} {r4} {r5};
    \Root[D] {} {r5} {r6};
  \ClusterD c4[\delta] = (r1)(r2)(r3)(r4)(r5)(r6);
\endclusterpicture}}}

\def\LBpgB{{\scalebox{1.5}{
\clusterpicture           
  \Root[D] {1} {first} {r1};
  \Root[D] {} {r1} {r2};
  \Root[D] {} {r2} {r3};
  \Root[D] {} {r3} {r4};
    \Root[D] {} {r4} {r5};
    \Root[D] {5} {r5} {r6};
  \ClusterLD c1[][s+\delta] = (r1)(r2)(r3)(r4)(r5);
  \ClusterD c4[0] = (c1)(r6);
\endclusterpicture}}}

\def\LBemA{{
\scalebox{1.5}{
\clusterpicture           
  \Root[D] {1} {first} {r1};
  \Root[D] {} {r1} {r2};
  \Root[D] {3} {r2} {r3};
  \Root[D] {3} {r3} {r4};
    \Root[D] {} {r4} {r5};
    \Root[D] {} {r5} {r6};
  \ClusterLD c1[][\frac{l}{2}] = (r1)(r2);
  \ClusterLD c2[][2a\!+\!\delta] = (c1)(r3);
  \ClusterLD c3[][2b\!+\!\epsilon] = (r4)(r5)(r6);
  \ClusterD c4[0] = (c2)(c3);
\endclusterpicture}
}}

\def\LBemB{{
\scalebox{1.5}{
\clusterpicture           
  \Root[D] {1} {first} {r1};
  \Root[D] {1} {r1} {r2};
  \Root[D] {} {r2} {r3};
  \Root[D] {4} {r3} {r4};
    \Root[D] {} {r4} {r5};
    \Root[D] {} {r5} {r6};
  \ClusterLD c1[][\frac{l}{2}] = (r2)(r3);
  \ClusterLD c3[][2t\!+\!\gamma] = (r4)(r5)(r6);
  \ClusterD c4[0] = (r1)(c1)(c3);
\endclusterpicture}
}}

\def\LBemC{{
\scalebox{1.5}{
\clusterpicture           
  \Root[D] {1} {first} {r1};
  \Root[D] {} {r1} {r2};
  \Root[D] {} {r2} {r3};
  \Root[D] {1} {r3} {r4};
    \Root[D] {1} {r4} {r5};
    \Root[D] {} {r5} {r6};
  \ClusterLD c1[][\frac{l}{2}] = (r5)(r6);
  \ClusterLD c3[][2t\!+\!\gamma] = (r4)(c1);R
  \ClusterD c4[\eta] = (r1)(c1)(c3);
\endclusterpicture}
}}

\def\LBemD{{
\scalebox{1.5}{
\clusterpicture           
  \Root[D] {1} {first} {r1};
  \Root[D] {} {r1} {r2};
  \Root[D] {1} {r2} {r3};
  \Root[D] {1} {r3} {r4};
    \Root[D] {} {r4} {r5};
    \Root[D] {} {r5} {r6};
  \ClusterLD c1[][2t\!+\!\gamma] = (r4)(r5)(r6);
  \ClusterLD c2[][\frac{l}{2}] = (r3)(c1);
  \ClusterD c4[\frac{l}{2}] = (r1)(r2)(c2);
\endclusterpicture}
}}

\def\LBemE{{
\scalebox{1.5}{
\clusterpicture           
  \Root[D] {1} {first} {r1};
  \Root[D] {} {r1} {r2};
  \Root[D] {6} {r2} {r3};
  \Root[D] {1} {r3} {r4};
    \Root[D] {} {r4} {r5};
    \Root[D] {} {r5} {r6};
  \ClusterLD c1[][2t\!+\!\gamma] = (r4)(r5)(r6);
  \ClusterLD c2[][j] = (r3)(c1);
  \ClusterLD c3[][\frac{l}{2}\!-\!j] = (r1)(r2);
  \ClusterD c4[0] = (c3)(c2);
\endclusterpicture}
}}

\def\LBemFb{{
\scalebox{1.5}{
\clusterpicture           
  \Root[D] {1} {first} {r1};
  \Root[D] {1} {r1} {r2};
  \Root[D] {} {r2} {r3};
  \Root[D] {1} {r3} {r4};
    \Root[D] {1} {r4} {r5};
    \Root[D] {} {r5} {r6};
  \ClusterLD c1[][\frac{l}{2}] = (r5)(r6);
  \ClusterLD c2[][2t\!+\!\gamma] = (r4)(c1);
  \ClusterLD c3[][2s\!+\!\eta] = (r2)(r3)(c2);
  \ClusterD c4[0] = (r1)(c3);
\endclusterpicture}
}}

\def\LBemG{{
\scalebox{1.5}{
\clusterpicture           
  \Root[D] {1} {first} {r1};
  \Root[D] {2} {r1} {r2};
  \Root[D] {} {r2} {r3};
  \Root[D] {3} {r3} {r4};
    \Root[D] {} {r4} {r5};
    \Root[D] {} {r5} {r6};
  \ClusterLD c1[][\frac{l}{2}] = (r2)(r3);
  \ClusterLD c2[][2t\!+\!\gamma] = (r4)(r5)(r6);
  \ClusterLD c3[][s] = (c1)(c2);
  \ClusterD c4[0] = (r1)(c3);
\endclusterpicture}
}}

\def\LBemH{{
\scalebox{1.5}{
\clusterpicture           
  \Root[D] {1} {first} {r1};
  \Root[D] {1} {r1} {r2};
  \Root[D] {1} {r2} {r3};
  \Root[D] {1} {r3} {r4};
    \Root[D] {} {r4} {r5};
    \Root[D] {} {r5} {r6};
  \ClusterLD c1[][2t\!+\!\gamma] = (r4)(r5)(r6);
  \ClusterLD c2[][\frac{l}{2}] = (r3)(c1);
  \ClusterLD c3[][s] = (r2)(c2);
  \ClusterD c4[0] = (r1)(c3);
\endclusterpicture}
}}

\def\LBeeA{{
\scalebox{1.5}{
\clusterpicture           
  \Root[D] {1} {first} {r1};
  \Root[D] {} {r1} {r2};
  \Root[D] {} {r2} {r3};
  \Root[D] {5} {r3} {r4};
    \Root[D] {} {r4} {r5};
    \Root[D] {} {r5} {r6};
  \ClusterLD c2[][t+\delta] = (r1)(r2)(r3);
  \ClusterLD c3[][t+\delta] = (r4)(r5)(r6);
  \ClusterD c4[\frac{1}{2}] = (c2)(c3);
\endclusterpicture}
}}

\def\LBeeB{{
\scalebox{1.5}{
\clusterpicture           
  \Root[D] {} {first} {r1};
  \Root[D] {} {r1} {r2};
  \Root[D] {} {r2} {r3};
  \Root[D] {5} {r3} {r4};
    \Root[D] {} {r4} {r5};
    \Root[D] {} {r5} {r6};
  \ClusterLD c2[][2a\!+\!\delta] = (r1)(r2)(r3);
  \ClusterLD c3[][2b\!+\!\epsilon] = (r4)(r5)(r6);
  \ClusterD c4[0] = (c2)(c3);
\endclusterpicture}
}}

\def\LBeeCb{{
\scalebox{1.5}{
\clusterpicture           
  \Root[D] {} {first} {r1};
  \Root[D] {} {r1} {r2};
  \Root[D] {} {r2} {r3};
  \Root[D] {1} {r3} {r4};
    \Root[D] {} {r4} {r5};
    \Root[D] {} {r5} {r6};
  \ClusterLD c3[][2t\!+\gamma] = (r4)(r5)(r6);
  \ClusterD c4[\eta] = (r1)(r2)(r3)(c3);
\endclusterpicture}
}}

\def\LBeeDb{{
\scalebox{1.5}{
\clusterpicture           
  \Root[D] {1} {first} {r1};
  \Root[D] {1} {r1} {r2};
  \Root[D] {} {r2} {r3};
  \Root[D] {1} {r3} {r4};
    \Root[D] {} {r4} {r5};
    \Root[D] {} {r5} {r6};
  \ClusterLD c1[][2t\!+\!\gamma] = (r4)(r5)(r6);
  \ClusterLD c2[][2s\!+\!\eta] = (r2)(r3)(c1);
  \ClusterD c4[0] = (r1)(c2);
\endclusterpicture}
}}

\def\LBonA{{
\scalebox{1.5}{
\clusterpicture           
  \Root[D] {1} {first} {r1};
  \Root[D] {} {r1} {r2};
  \Root[D] {5} {r2} {r3};
  \Root[D] {} {r3} {r4};
    \Root[D] {} {r4} {r5};
    \Root[D] {} {r5} {r6};
  \ClusterLD c1[][\frac{l}{2}+\delta] = (r1)(r2);
  \ClusterLD c2[][-\delta] = (c1)(r3)(r4)(r5)(r6);
\endclusterpicture}
}}

\def\LBonB{{
\scalebox{1.5}{
\clusterpicture           
  \Root[D] {} {first} {r1};
  \Root[D] {} {r1} {r2};
  \Root[D] {1} {r2} {r3};
  \Root[D] {} {r3} {r4};
    \Root[D] {} {r4} {r5};
    \Root[D] {} {r5} {r6};
  \ClusterLD c1[][\frac{l}{2}+\delta] = (r3)(r4)(r5)(r6);
  \ClusterLD c2[][\frac{l}{2}] = (r1)(r2)(c1);
\endclusterpicture}
}}

\def\LBonC{{
\scalebox{1.5}{
\clusterpicture           
  \Root[D] {} {first} {r1};
  \Root[D] {} {r1} {r2};
  \Root[D] {5} {r2} {r3};
  \Root[D] {} {r3} {r4};
    \Root[D] {} {r4} {r5};
    \Root[D] {} {r5} {r6};
  \ClusterLD c1[][\frac{l}{2}-j] = (r1)(r2); 
  \ClusterLD c2[][j + \delta] = (r3)(r4)(r5)(r6);
  \ClusterLD c3[][0] = (c1)(c2);
\endclusterpicture}
}}

\def\LBonD{{
\scalebox{1.5}{
\clusterpicture           
  \Root[D] {} {first} {r1};
  \Root[D] {1} {r1} {r2};
  \Root[D] {} {r2} {r3};
  \Root[D] {} {r3} {r4};
    \Root[D] {1} {r4} {r5};
    \Root[D] {} {r5} {r6};
  \ClusterLD c1[][\frac{l}{2}+\delta] = (r5)(r6); 
  \ClusterLD c2[][s-\delta] = (r2)(r3)(r4)(c1);
  \ClusterLD c3[][0] = (r1)(c2);
\endclusterpicture}
}}

\def\LBonE{{
\scalebox{1.5}{
\clusterpicture           
  \Root[D] {} {first} {r1};
  \Root[D] {1} {r1} {r2};
  \Root[D] {1} {r2} {r3};
  \Root[D] {} {r3} {r4};
    \Root[D] {} {r4} {r5};
    \Root[D] {} {r5} {r6};
  \ClusterLD c1[][\frac{l}{2}+\delta] = (r3)(r4)(r5)(r6); 
  \ClusterLD c2[][s] = (r2)(c1);
  \ClusterLD c3[][0] = (r1)(c2);
\endclusterpicture}
}}

\def\LBpmpmA{{
\scalebox{1.5}{
\clusterpicture           
  \Root[D] {1} {first} {r1};
  \Root[D] {} {r1} {r2};
  \Root[D] {2} {r2} {r3};
  \Root[D] {7} {r3} {r4};
    \Root[D] {} {r4} {r5};
    \Root[D] {2} {r5} {r6};
  \ClusterLD c1[][\frac{l}{2}] = (r1)(r2);
  \ClusterLD c2[][2a+\delta] = (c1)(r3);
  \ClusterLD c3[][\frac{m}{2}] = (r4)(r5);
  \ClusterLD c4[][2b+\epsilon] = (c3)(r6);
  \ClusterLD c5[][0] = (c2)(c4);
\endclusterpicture}
}}

\def\LBpmpmB{{
\scalebox{1.5}{
\clusterpicture           
  \Root[D] {1} {first} {r1};
  \Root[D] {} {r1} {r2};
  \Root[D] {2} {r2} {r3};
  \Root[D] {6} {r3} {r4};
    \Root[D] {} {r4} {r5};
    \Root[D] {2} {r5} {r6};
  \ClusterLD c1[][\frac{l}{4}] = (r1)(r2);
  \ClusterLD c2[][t+\delta] = (c1)(r3);
  \ClusterLD c3[][\frac{l}{4}] = (r4)(r5);
  \ClusterLD c4[][t+\delta] = (c3)(r6);
  \ClusterLD c5[][\frac{1}{2}] = (c2)(c4);
\endclusterpicture}
}}

\def\LBpmpmBb{{
\scalebox{1.5}{
\clusterpicture           
  \Root[D] {1} {first} {r1};
  \Root[D] {} {r1} {r2};
  \Root[D] {2} {r2} {r3};
  \Root[D] {6} {r3} {r4};
    \Root[D] {} {r4} {r5};
    \Root[D] {2} {r5} {r6};
  \ClusterLD c1[][\frac{l}{4}] = (r1)(r2);
  \ClusterLD c2[][t+\delta] = (c1)(r3);
  \ClusterLD c3[][\frac{l}{4}] = (r4)(r5);
  \ClusterLD c4[][t+\delta] = (c3)(r6);
  \ClusterLD c5[][-k + \frac{1}{2}] = (c2)(c4);
\endclusterpicture}
}}

\def\LBpmpmC{{
\scalebox{1.5}{
\clusterpicture           
  \Root[D] {1} {first} {r1};
  \Root[D] {} {r1} {r2};
  \Root[D] {2} {r2} {r3};
  \Root[D] {3} {r3} {r4};
    \Root[D] {} {r4} {r5};
    \Root[D] {2} {r5} {r6};
  \ClusterLD c1[][\frac{l}{2}] = (r1)(r2);
  \ClusterLD c2[][2t + \gamma] = (c1)(r3);
  \ClusterLD c3[][\frac{m}{2}] = (r4)(r5);
  \ClusterLD c5[][0] = (c2)(c3)(r6);
\endclusterpicture}
}}

\def\LBpmpmD{{
\scalebox{1.5}{
\clusterpicture           
  \Root[D] {1} {first} {r1};
  \Root[D] {} {r1} {r2};
  \Root[D] {2} {r2} {r3};
  \Root[D] {2} {r3} {r4};
    \Root[D] {2} {r4} {r5};
    \Root[D] {} {r5} {r6};
  \ClusterLD c1[][\frac{l}{2}] = (r1)(r2);
  \ClusterLD c2[][2t+\gamma] = (c1)(r3);
  \ClusterLD c3[][\frac{m}{2}] = (c2)(r4);
  \ClusterLD c5[][\frac{m}{2}] = (c3)(r5)(r6);
\endclusterpicture}
}}

\def\LBpmpmE{{
\scalebox{1.5}{
\clusterpicture           
  \Root[D] {1} {first} {r1};
  \Root[D] {} {r1} {r2};
  \Root[D] {2} {r2} {r3};
  \Root[D] {2} {r3} {r4};
    \Root[D] {2} {r4} {r5};
    \Root[D] {2} {r5} {r6};
  \ClusterLD c1[][\frac{l}{2}] = (r1)(r2);
  \ClusterLD c2[][2t+\gamma] = (c1)(r3);
  \ClusterLD c3[][\frac{m}{2}] = (c2)(r4);
  \ClusterLD c4[][s] = (c3)(r5);
  \ClusterLD c5[][0] = (c4)(r6);
\endclusterpicture}
}}

\def\LBpmpmF{{
\scalebox{1.5}{
\clusterpicture           
  \Root[D] {1} {first} {r1};
  \Root[D] {} {r1} {r2};
  \Root[D] {2} {r2} {r3};
  \Root[D] {3} {r3} {r4};
    \Root[D] {} {r4} {r5};
    \Root[D] {4} {r5} {r6};
  \ClusterLD c1[][\frac{l}{2}] = (r1)(r2);
  \ClusterLD c2[][2t + \gamma] = (c1)(r3);
  \ClusterLD c3[][\frac{m}{2}] = (r4)(r5);
  \ClusterLD c4[][s] = (c2)(c3)(c3);
  \ClusterLD c5[][0] = (c4)(r6);
\endclusterpicture}
}}

\def\LBpmpmG{{
\scalebox{1.5}{
\clusterpicture           
  \Root[D] {1} {first} {r1};
  \Root[D] {} {r1} {r2};
  \Root[D] {2} {r2} {r3};
  \Root[D] {2} {r3} {r4};
    \Root[D] {4} {r4} {r5};
    \Root[D] {} {r5} {r6};
  \ClusterLD c1[][\frac{l}{2}] = (r1)(r2);
  \ClusterLD c2[][2t+\gamma] = (c1)(r3);
  \ClusterLD c3[][j] = (c2)(r4);
  \ClusterLD c5[][\frac{m}{2}-j] = (r5)(r6);
  \ClusterLD c6[][0] = (c3)(c5);
\endclusterpicture}
}}

\def\LBtwonAa{{
\scalebox{1.5}{
\clusterpicture           
  \Root[D] {} {first} {r1};
  \Root[D] {} {r1} {r2};
  \Root[D] {3} {r2} {r3};
  \Root[D] {} {r3} {r4};
    \Root[D] {2} {r4} {r5};
    \Root[D] {} {r5} {r6};
  \ClusterLD c1[][\frac{l}{2}] = (r1)(r2); 
  \ClusterLD c2[][\frac{m}{2}] = (r3)(r4);
  \ClusterLD c3[][0] = (c1)(c2)(r5)(r6);
\endclusterpicture}
}}

\def\LBtwonAb{{
\scalebox{1.5}{
\clusterpicture           
  \Root[D] {} {first} {r1};
  \Root[D] {} {r1} {r2};
  \Root[D] {3} {r2} {r3};
  \Root[D] {} {r3} {r4};
    \Root[D] {2} {r4} {r5};
    \Root[D] {} {r5} {r6};
  \ClusterLD c1[][\frac{l}{4}] = (r1)(r2); 
  \ClusterLD c2[][\frac{l}{4}] = (r3)(r4);
  \ClusterLD c3[][\frac{1}{2}] = (c1)(c2)(r5)(r6);
\endclusterpicture}
}}

\def\LBtwonBa{{
\scalebox{1.5}{
\clusterpicture           
  \Root[D] {} {first} {r1};
  \Root[D] {} {r1} {r2};
  \Root[D] {1} {r2} {r3};
  \Root[D] {} {r3} {r4};
    \Root[D] {4} {r4} {r5};
    \Root[D] {} {r5} {r6};
  \ClusterLD c1[][\frac{l}{2}] = (r3)(r4); 
  \ClusterLD c2[][\frac{m}{2}] = (r1)(r2)(c1);
  \ClusterLD c3[][0] = (c1)(c2)(r5)(r6);
\endclusterpicture}
}}

\def\LBtwonBb{{
\scalebox{1.5}{
\clusterpicture           
  \Root[D] {} {first} {r1};
  \Root[D] {} {r1} {r2};
  \Root[D] {1} {r2} {r3};
  \Root[D] {} {r3} {r4};
    \Root[D] {4} {r4} {r5};
    \Root[D] {} {r5} {r6};
  \ClusterLD c1[][\frac{l}{2}] = (r3)(r4); 
  \ClusterLD c2[][\frac{m}{2}] = (r1)(r2)(c1);
  \ClusterLD c3[][\frac{1}{2}] = (c1)(c2)(r5)(r6);
\endclusterpicture}
}}

\def\LBtwonCa{{
\scalebox{1.5}{
\clusterpicture           
  \Root[D] {} {first} {r1};
  \Root[D] {} {r1} {r2};
  \Root[D] {6} {r2} {r3};
  \Root[D] {} {r3} {r4};
    \Root[D] {1} {r4} {r5};
    \Root[D] {} {r5} {r6};
  \ClusterLD c1[][\frac{m}{2} - j] = (r1)(r2); 
  \ClusterLD c2[][\frac{l}{2}] = (r5)(r6);
  \ClusterLD c3[][j] = (r3)(r4)(c2);
   \ClusterLD c4[][0] = (c1)(c3);
\endclusterpicture}
}}

\def\LBtwonCb{{
\scalebox{1.5}{
\clusterpicture           
  \Root[D] {} {first} {r1};
  \Root[D] {} {r1} {r2};
  \Root[D] {9} {r2} {r3};
  \Root[D] {} {r3} {r4};
    \Root[D] {1} {r4} {r5};
    \Root[D] {} {r5} {r6};
  \ClusterLD c1[][\frac{m}{2} - j-\frac{1}{2}] = (r1)(r2); 
  \ClusterLD c2[][\frac{l}{2}] = (r5)(r6);
  \ClusterLD c3[][j+\frac{1}{2}] = (r3)(r4)(c2);
   \ClusterLD c4[][0] = (c1)(c3);
\endclusterpicture}
}}

\def\LBtwonDa{{
\scalebox{1.5}{
\clusterpicture           
  \Root[D] {} {first} {r1};
  \Root[D] {1} {r1} {r2};
  \Root[D] {1} {r2} {r3};
  \Root[D] {} {r3} {r4};
    \Root[D] {3} {r4} {r5};
    \Root[D] {} {r5} {r6};
  \ClusterLD c1[][\frac{l}{2}] = (r3)(r4); 
  \ClusterLD c2[][\frac{m}{2}] = (r5)(r6);
  \ClusterLD c3[][s] = (r2)(c1)(c2);
   \ClusterLD c4[][0] = (r1)(c3);
\endclusterpicture}
}}

\def\LBtwonDb{{
\scalebox{1.5}{
\clusterpicture           
  \Root[D] {} {first} {r1};
  \Root[D] {1} {r1} {r2};
  \Root[D] {1} {r2} {r3};
  \Root[D] {} {r3} {r4};
    \Root[D] {3} {r4} {r5};
    \Root[D] {} {r5} {r6};
  \ClusterLD c1[][\frac{l}{4}] = (r3)(r4); 
  \ClusterLD c2[][\frac{l}{4}] = (r5)(r6);
  \ClusterLD c3[][s+\frac{1}{2}] = (r2)(c1)(c2);
   \ClusterLD c4[][0] = (r1)(c3);
\endclusterpicture}
}}

\def\LBtwonE{{
\scalebox{1.5}{
\clusterpicture           
  \Root[D] {} {first} {r1};
  \Root[D] {1} {r1} {r2};
  \Root[D] {1} {r2} {r3};
  \Root[D] {} {r3} {r4};
    \Root[D] {1} {r4} {r5};
    \Root[D] {} {r5} {r6};
  \ClusterLD c1[][\frac{l}{2}] = (r5)(r6); 
  \ClusterLD c2[][\frac{m}{2}] = (r3)(r4)(c1);
  \ClusterLD c3[][s] = (r2)(c2);
   \ClusterLD c4[][0] = (r1)(c3);
\endclusterpicture}
}}

\def\LBthreenAa{{
\scalebox{1.5}{
\clusterpicture           
  \Root[D] {1} {first} {r1};
  \Root[D] {} {r1} {r2};
  \Root[D] {3} {r2} {r3};
  \Root[D] {} {r3} {r4};
    \Root[D] {3} {r4} {r5};
    \Root[D] {} {r5} {r6};
  \ClusterLD c1[][\frac{l}{2}] = (r1)(r2); 
  \ClusterLD c2[][\frac{m}{2}] = (r3)(r4);
  \ClusterLD c3[][\frac{n}{2}] = (r5)(r6);
    \ClusterLD c4[][0] = (c1)(c2)(c3);
\endclusterpicture}
}}

\def\LBthreenAb{{
\scalebox{1.5}{
\clusterpicture           
  \Root[D] {1} {first} {r1};
  \Root[D] {} {r1} {r2};
  \Root[D] {3} {r2} {r3};
  \Root[D] {} {r3} {r4};
    \Root[D] {3} {r4} {r5};
    \Root[D] {} {r5} {r6};
  \ClusterLD c1[][\frac{m}{4}] = (r1)(r2); 
  \ClusterLD c2[][\frac{m}{4}] = (r3)(r4);
  \ClusterLD c3[][\frac{l}{2}] = (r5)(r6);
    \ClusterLD c4[][\frac{1}{2}] = (c1)(c2)(c3);
\endclusterpicture}
}}

\def\LBthreenAc{{
\scalebox{1.5}{
\clusterpicture           
  \Root[D] {1} {first} {r1};
  \Root[D] {} {r1} {r2};
  \Root[D] {3} {r2} {r3};
  \Root[D] {} {r3} {r4};
    \Root[D] {3} {r4} {r5};
    \Root[D] {} {r5} {r6};
  \ClusterLD c1[][\frac{l}{6}] = (r1)(r2); 
  \ClusterLD c2[][\frac{l}{6}] = (r3)(r4);
  \ClusterLD c3[][\frac{l}{6}] = (r5)(r6);
    \ClusterLD c4[][\frac{1}{3}\textup{ or }\frac{2}{3}] = (c1)(c2)(c3);
\endclusterpicture}
}}

\def\LBthreenBa{{
\scalebox{1.5}{
\clusterpicture           
  \Root[D] {} {first} {r1};
  \Root[D] {} {r1} {r2};
  \Root[D] {2} {r2} {r3};
  \Root[D] {} {r3} {r4};
    \Root[D] {2.5} {r4} {r5};
    \Root[D] {} {r5} {r6};
  \ClusterLD c1[][\frac{m}{2}] = (r3)(r4); 
  \ClusterLD c2[][\frac{n}{2}] = (r5)(r6);
  \ClusterLD c3[][\frac{l}{2}] = (c1)(c2);
  \ClusterLD c4[][-\frac{l}{2}] = (r1)(r2)(c3);
\endclusterpicture}
}}

\def\LBthreenBb{{
\scalebox{1.5}{
\clusterpicture           
  \Root[D] {} {first} {r1};
  \Root[D] {} {r1} {r2};
  \Root[D] {2} {r2} {r3};
  \Root[D] {} {r3} {r4};
    \Root[D] {2.5} {r4} {r5};
    \Root[D] {} {r5} {r6};
  \ClusterLD c1[][\frac{m}{4}] = (r3)(r4); 
  \ClusterLD c2[][\frac{m}{4}] = (r5)(r6);
  \ClusterLD c3[][\frac{l}{2}] = (c1)(c2);
  \ClusterLD c4[][-\frac{l}{2}-\frac{1}{2}] = (r1)(r2)(c3);
\endclusterpicture}
}}

\def\LBthreenCa{{
\scalebox{1.5}{
\clusterpicture           
  \Root[D] {} {first} {r1};
  \Root[D] {} {r1} {r2};
  \Root[D] {7} {r2} {r3};
  \Root[D] {} {r3} {r4};
    \Root[D] {2.5} {r4} {r5};
    \Root[D] {} {r5} {r6};
  \ClusterLD c1[][\frac{m}{2}] = (r3)(r4); 
  \ClusterLD c2[][\frac{n}{2}] = (r5)(r6);
  \ClusterLD c3[][j] = (c1)(c2);
  \ClusterLD c4[][\frac{l}{2}-j] = (r1)(r2);
  \ClusterLD c4[][0] = (c4)(c3);
\endclusterpicture}
}}

\def\LBthreenCb{{
\scalebox{1.5}{
\clusterpicture           
  \Root[D] {} {first} {r1};
  \Root[D] {} {r1} {r2};
  \Root[D] {9} {r2} {r3};
  \Root[D] {} {r3} {r4};
    \Root[D] {2.5} {r4} {r5};
    \Root[D] {} {r5} {r6};
  \ClusterLD c1[][\frac{m}{4}] = (r3)(r4); 
  \ClusterLD c2[][\frac{m}{4}] = (r5)(r6);
  \ClusterLD c3[][j+\frac{1}{2}] = (c1)(c2);
  \ClusterLD c4[][\frac{l}{2}-j-\frac{1}{2}] = (r1)(r2);
  \ClusterLD c4[][0] = (c4)(c3);
\endclusterpicture}
}}

\def\LBthreenDa{{
\scalebox{1.5}{
\clusterpicture           
  \Root[D] {} {first} {r1};
  \Root[D] {1} {r1} {r2};
  \Root[D] {2} {r2} {r3};
  \Root[D] {} {r3} {r4};
    \Root[D] {2.5} {r4} {r5};
    \Root[D] {} {r5} {r6};
  \ClusterLD c1[][\frac{m}{2}] = (r3)(r4); 
  \ClusterLD c2[][\frac{n}{2}] = (r5)(r6);
  \ClusterLD c3[][\frac{l}{2}] = (c1)(c2);
  \ClusterLD c3[][s] = (r2)(c3);  
   \ClusterLD c4[][0] = (r1)(c3);
\endclusterpicture}
}}

\def\LBthreenDb{{
\scalebox{1.5}{
\clusterpicture           
  \Root[D] {} {first} {r1};
  \Root[D] {1} {r1} {r2};
  \Root[D] {2} {r2} {r3};
  \Root[D] {} {r3} {r4};
    \Root[D] {2.5} {r4} {r5};
    \Root[D] {} {r5} {r6};
  \ClusterLD c1[][\frac{m}{4}] = (r3)(r4); 
  \ClusterLD c2[][\frac{m}{4}] = (r5)(r6);
  \ClusterLD c3[][\frac{l}{2}] = (c1)(c2);
  \ClusterLD c3[][s] = (r2)(c3);  
   \ClusterLD c4[][0] = (r1)(c3);
\endclusterpicture}
}}

\makesavenoteenv{table}
\makesavenoteenv{tabularx}

\title{Reduction types of genus $2$ curves}

\DeclareRobustCommand{\potgoodpicture}{\begin{tikzpicture}
\draw (0,0) -- (2,0);

\node[scale=0.6] at (1.5,0.1) {\footnotesize{$g_2$}};
\end{tikzpicture}}

\DeclareRobustCommand{\potonenodepicture}{\begin{tikzpicture}
  \draw (0, 0) to[out=0, in=180] (0.75,0) to[out=0, in=-90] (1.125,0.15) to[out=90, in=0] (1, 0.25) to[out=180, in=90] (0.875,0.15) to[out=-90,in=180] (1.25,0) to[out=0,in=180] (2,0);

  \node[scale=0.6] at (1.5,0.1) {\footnotesize{$g_1$}};
\end{tikzpicture}}

\DeclareRobustCommand{\pottwonodespicture}{\begin{tikzpicture}
  \draw (0, 0) to[out=0, in=180] (0.45,0) to[out=0, in=-90] (0.825,0.15) to[out=90, in=0] (0.7, 0.25) to[out=180, in=90] (0.575,0.15) to[out=-90,in=180] (0.95,0) to[out=0,in=180] (1,0) 
  
  to[out=0, in=180] (1.05,0) to[out=0, in=-90] (1.425,0.15) to[out=90, in=0] (1.3, 0.25) to[out=180, in=90] (1.175,0.15) to[out=-90,in=180] (1.55,0) to[out=0,in=180] (2,0);
\end{tikzpicture}}

\DeclareRobustCommand{\potthreenodespicture}{\begin{tikzpicture}

\draw 
    (0,0.15) to[out=0, in=180] (0.66666,-0.15) to[out=0, in=180] (1.333333,0.15) to[out=0, in=180] (2,-0.15);

\draw 
    (0,-0.15) to[out=0, in=180] (0.66666,0.15) to[out=0, in=180] (1.333333,-0.15) to[out=0, in=180] (2,0.15);

\end{tikzpicture}}

\DeclareRobustCommand{\potgoodtimestwopicture}{\begin{tikzpicture}

\draw (0, 0.15) to[out=0, in=180] (0.9,0.15) to[out=0,in=180] (1.1,0.25); 

\draw (0.9,0.25) to[out=0,in=180] (1.1,0.15) to[out=0,in=180] (2,0.15);

\node[scale=0.6] at (1.5,0.25) {\footnotesize{$g_1$}};

\node[scale=0.6] at (0.5,0.25) {\footnotesize{$g_1$}};

\end{tikzpicture}}

\DeclareRobustCommand{\potgoodtimespotmultpicture}{\begin{tikzpicture}

\draw (0, 0.15) to[out=0, in=180] (0.35,0.15) to[out=0, in=-90] (0.725,0.3) to[out=90, in=0] (0.6, 0.4) to[out=180, in=90] (0.475,0.3) to[out=-90,in=180] (0.85,0.15) to[out=0,in=180] (0.9,0.15) to[out=0,in=180] (1.1,0.25); 

\draw (0.9,0.25) to[out=0,in=180] (1.1,0.15) to[out=0,in=180] (2,0.15);

\node[scale=0.6] at (1.5,0.25) {\footnotesize{$g_1$}};

\end{tikzpicture}
}

\DeclareRobustCommand{\potmulttimestwopicture}{\begin{tikzpicture}

\draw (0, 0.15) to[out=0, in=180] (0.35,0.15) to[out=0, in=-90] (0.725,0.3) to[out=90, in=0] (0.6, 0.4) to[out=180, in=90] (0.475,0.3) to[out=-90,in=180] (0.85,0.15) to[out=0,in=180] (0.9,0.15) to[out=0,in=180] (1.1,0.25); 

\draw (0.9,0.25) to[out=0,in=180] (1.1,0.15) to[out=0,in=180] (1.15,0.15) to[out=0, in=-90] (1.525,0.3) to[out=90, in=0] (1.4, 0.4) to[out=180, in=90] (1.275,0.3) to[out=-90,in=180] (1.65,0.15) to[out=0,in=180] (2,0.15);

\end{tikzpicture}
}

\begin{document}
\author{Edwina Aylward}
\author{Lilybelle Cowland Kellock}
\author{Vladimir Dokchitser}
\author{Elvira Lupoian}
\address{Department of Mathematics, University College London, London, WC1H 0AY, United Kingdom}
\email{edwina.aylward.23@ucl.ac.uk}
\email{v.dokchitser@ucl.ac.uk}
\email{elupoian@gmail.com}
\address{Department of Mathematics, University of Manchester, Manchester, M13 9PL, United Kingdom}
\email{lilybelle.cowlandkellock@manchester.ac.uk}
\date{}

\subjclass[2020]{Primary: 11G20, Secondary: 14D10, 14G20, 14H45}


\begin{abstract}
Tate produced a table for elliptic curves over local fields that beautifully summarises their arithmetic invariants in terms of the Kodaira type. We present an analogous set of tables for curves of genus 2. The invariants addressed include: reduction type of the minimal regular model, reduction type of the minimal regular model with normal crossings, Néron component group of the Jacobian, conductor exponent, valuation of the discriminant of a minimal Weierstrass model, and cluster pictures. 
\end{abstract}

\maketitle

\vspace{-0.5cm}

\tableofcontents


\addtocontents{toc}{\protect\setcounter{tocdepth}{1}}

\section{Introduction}

The aim of the present paper is to classify the main arithmetic invariants of genus 2 curves over local fields in terms of their reduction type. Our presentation is modelled on Tate's classical table for elliptic curves in \cite{tate}, which was reproduced by Silverman in his books \cite{sil1,sil2}.
The Kodaira symbols used in the elliptic curves classification are replaced by 106 families of ``reduction types'', which are essentially the labels introduced by Namikawa and Ueno \cite{namikawaueno}.
We have split these in seven tables --- one for each (expected) potential stable type, which are the genus 2 analogues of ``potentially good'' and ``potentially multiplicative'' reduction. For the convenience of the reader, the data in these tables is also available on the accompanying web-page: \url{https://elviralupoian.github.io/genus2reductiontypes/}.

As for elliptic curves, a key feature is that this classification is essentially independent of the underlying local field or DVR, except for some cases of residue characteristic $\leq 5$. Reduction types can be thought of as being purely combinatorial objects. The associated curve invariants are given in terms of the types and the parameters of the families (analogously to how an elliptic curve with Kodaira type I${}_n$ has N{\'e}ron component group $\Z/n\Z$ for every $n\ge 1$), rather than through an algorithm that relies on the equation of the curve.

Our setting is as follows.
Let $R$ be a complete discrete valuation ring with valuation $v$, fraction field $K$ and an algebraically closed residue field $k$ of characteristic $p\ge 0$. Throughout this work, $C/K$ denotes a smooth, projective and geometrically connected curve. The terminology used in the following theorem is introduced in \S \ref{ss:notation} and thereafter. 

\begin{theorem}\label{thm:main}
Tables \ref{table:potentiallygood}--\ref{table:potentiallymultiplicativex2} list all possible reduction types of the minimal regular model of genus $2$ curves over $K$, and how the reduction type determines
\begin{enumerate}
\item Reduction type of the minimal regular normal crossings model;
\item N{\'e}ron component group of the Jacobian;
\item Tame part of the conductor exponent;
\item Whether the curve has tame reduction,
 in terms of the value of $p$;
 \item Birch--Swinnerton-Dyer ``fudge factor'' $v(\omega_{\text{min}}/\omega^0)$ provided that $p \neq 2$.
\end{enumerate}
When $p\neq 2$, for curves of genus 2 with tame reduction the tables\footnote{In the exceptional case of Namikawa--Ueno labels II${}_{t-m}^*$ and 2I${}_m$-$t$, which have identical special fibres, we also require the potential stable type as an input; this also applies to II${}_{t-0}^*$ and 2I${}_0$-$t$} also list their:
\begin{enumerate}
\item[(6)] Potential stable type;
\item[(7)] Valuation of the discriminant of a minimal Weierstrass model; 
\item[(8)] Complete list of cluster pictures of curves giving rise to this reduction type. Each cluster picture gives rise to exactly one reduction type. Moreover, for a given curve each cluster picture in this list is realised by some Weierstrass equation. 
\end{enumerate}
\end{theorem}


Our main goal has been to present the local invariants of genus 2 curves in a user-friendly format. In doing so, we have had to address a number of issues that appear to be gaps in the literature or that have appeared without justification. Let us stress that \Cref{thm:main} gives invariants in terms of the minimal regular model reduction type; for a given curve there are well-known algorithms that compute these from a Weierstrass equation, such as Liu's \cite{liu-models}. Indeed, we were motivated by theoretical applications that rely on a finite case-by-case analysis of reduction types, such as the original proof of Ogg's formula \cite{Oggformula}, the parity conjecture for elliptic curves and for semistable abelian surfaces \cite{squarity,paras}, bounding torsion subgroups over number fields \cite{flexor-oesterle} and the study of canonical heights of points on Jacobians \cite{heights}. We give a proof to a conjecture of Lorenzini as an illustration (see Theorem \ref{Z/2Z/4}). 


The main theoretical results of the paper of the are as follows:
\begin{itemize}
    \item The 1-to-1 correspondence between reduction types of minimal regular models (essentially Namikawa--Ueno labels) and minimal models with normal crossings (Theorem \ref{MRNCtoMRMintro}). This also appears in \cite{timreduction}, but without justification.
    \item The potential stable type for each Namikawa--Ueno label.
    \item The correspondence between reduction types and cluster pictures (\Cref{thm:CP2} and \Cref{thm:CP3}).
    \item An approach for computing N\'eron component groups of general curves based on their reduction type (Theorem \ref{thm:introcomponentgroups}). The resulting tables for curves of genus 2 are already given in \cite{liu-models}, based on the same technique (although without a full proof). 
    \item The value of the minimal Weierstrass discrimiant and (a choice of) minimal Weierstrass model for each reduction type. A different minimal discriminant $\Delta_u$ 
    was previously tabulated by Ueno \cite[\S5]{ueno}, which, as pointed out to us by Qing Liu, can also be used to compute $v(\Delta_{\text{min}})$ by the formula $v(\Delta_{\min}) = v(\Delta_u) -10v(\omega_{\text{min}} / \omega_0)$.
\end{itemize}
The other invariants that we tabulate have been straightforward applications of results in the literature:
$v(\omega_{\min}/\omega^0)$ (computed from the tables in \cite[\S5]{ueno} as $d_x - \delta_x$ as observed in \cite[Remarque 6]{Liu}), 
tame conductor exponents (from the special fibre of a regular normal crossings model, see \Cref{tameconductorformula})
and residue characteristic for wild reduction 
(based on Saito's criterion).

For want of space, we have not included the pictures of special fibres in the tables. Those for minimal regular models can be found in \cite{namikawaueno}, and for minimal regular models with normal crossings at the end of \cite{timreduction}. Both sets of pictures are included in the online version of the tables.
When $p\neq 2$, for curves with tame reduction, Bisatt \cite{bisatt} has classified the inertia actions on the $\ell$-adic Galois representation (for $\ell\neq p$) in terms of the underlying cluster pictures, which, using our tables, can be related to the reduction type. 

Beyond the sheer number of different reduction types, genus 2 curves present some new features and challenges when compared to elliptic curves. In the remainder of the introduction, we highlight and briefly discuss three of these: (i) what we mean by a {\em reduction type}, (ii) our approach to computing N{\'e}ron component groups, and (iii) cluster pictures, which serve as a replacement for the valuations of the $j$-invariant and discriminant. We end the introduction with some examples and applications.

\subsection{Reduction types}\label{s:introreduction}

Formally, by the {\em reduction type} of a 
regular model $\mathcal{C}/R$ we mean the following data (see Definition \ref{def:reductiontype}):
\begin{itemize}
\item The list of components of the special fibre $\mathcal{C}_k$ with their multiplicity and geometric genus;
\item The intersection pairing on components; and
\item The list of singular points of $\mathcal{C}_k$, whether they are normal crossings singularities, and the number of times that each component passes through each singular point.
\end{itemize}
The {\em MRM} (respectively, {\em MRNC})
{\em reduction type} is the reduction type of the minimal regular model of~$C$ (respectively, minimal regular model with normal crossings).

For semistable and MRNC models, this data is equivalent to the dual graph of the special fibre (with genera and multiplicities). Our definition of the MRM reduction type is fairly crude for general curves, as one could keep better track of the types of singularities that appear. However, it recovers the Kodaira types for elliptic curves and the description of special fibres given by the Namikawa--Ueno classification for genus 2 curves (Lemma \ref{lem:MRMtoNU}).

As there appears to be a shift from using MRM to MRNC reduction types (e.g. \cite{timreduction,Faraggi-Nowell}) we have decided to include both in our tables. The special fibre of the minimal regular model and of the minimal regular model with normal crossings do not in general agree. However, there is a 1-to-1 correspondence between them.

\begin{theorem}[see Theorem \ref{thm:SNCvsMRM} and Corollary \ref{cor:genus2MRNCtoMRM}] \label{MRNCtoMRMintro}
The MRNC reduction type of a curve determines its MRM reduction type. For curves of genus 2 this is a 1-to-1 correspondence.
\end{theorem}

Schrettner \cite{Jakab} has recently proposed a refinement for the definition of MRM reduction type, and showed that these are in 1-to-1 correspondence with MRNC reduction types in general.

We caution the reader that, in spite of Tate's table for elliptic curves and Theorem \ref{thm:main}, the MRM reduction type can miss some crucial arithmetic data, such as the potential stable type and the potential toric rank of the Jacobian. For elliptic curves this is visible in residue characteristic $p\!=\!2$: an elliptic curve with Kodaira type I${}_n^*$ can have either potentially good or potentially multiplicative reduction. For genus 2 curves this is an issue in {\em{every}} residue characteristic: there is an MRM family of reduction types (Namikawa–Ueno labels II${}_{t-m}^*$ and 2I${}_m$-$t$) that can have either stable  stable type IV (two genus 0 curves meeting at three points) or type VII (two genus 0 curves with a node each intersecting at one point); and for $m\!=\!0$ they are either type II (elliptic curve with a node) or type V (two elliptic curves meeting at a point), which do not even have the same potential toric rank. However, as Theorem \ref{thm:main} shows, these are the only problematic MRM reduction types, at least for tame curves and $p\neq 2$. 

\subsection{N{\'e}ron component groups}

The N{\'e}ron component group of the Jacobian of a curve can be explicitly computed from the intersection pairing matrix on the special fibre of a regular model of the curve due to Raynaud's work \cite[Thm. 1.1]{boschliu} and is therefore a natural invariant of the reduction type. Unfortunately, manipulating these types of matrices is troublesome for whole families of reduction types, as the dimension of the matrix grows with the underlying parameters, for example the Kodaira type I${}_n$ results in an $n\times n$ matrix. We propose the following alternative approach, which keeps to $2g\times 2g$ matrices, where $g$ is the genus.

\begin{theorem}[see \Cref{compgroupthm}]\label{thm:introcomponentgroups}
Let $C/K$ be a curve, and suppose that $C'/\C[t]$ is a curve with the same MRM reduction type at $t\!=\!0$ as $C/K$. Then the N{\'e}ron component group of $C/K$ is isomorphic to the torsion subgroup of the co-invariants of the endomorphism of $H_1(C',\Z)$ induced by the monodromy operator at $t\!=\!0$.
\end{theorem}

As the N{\'e}ron component group only depends on the intersection pairing matrix, the auxiliary curve~$C'$ in fact only needs to have the same intersection pairing rather than MRM reduction type. Such a curve always exists by a result of Winters \cite{winters}.  
However, for our classification in genus 2 we need to work with families of reduction types, rather than individual ones. The corresponding monodromy matrices, in terms of the family parameters, were already computed by Namikawa and Ueno in \cite{namikawaueno}, and turn out to be suitable for our purposes.

\subsection{Cluster pictures as an analogue of $v(j)$ and $v(\Delta)$}\label{subsec:intro_cpics}

Classically, Kodaira types of elliptic curves can be obtained from the valuation of the discriminant and $j$-invariant of the curve when $p\neq 2, 3$ (see \cite[p.46]{tate}).  
Unfortunately, there is no direct analogue of such invariants in the genus $2$ case
\cite[Thm.~1.2]{lilybelle}. 

As all genus $2$ curves are hyperelliptic, following \cite{M2D2} we instead use cluster pictures, which requires us to restrict to residue characteristic $p \not=2$. 

For a hyperelliptic curve $y^2\!=\!f(x)$ its cluster picture is simply the data consisting of the matrix of $v$-adic distances between the roots of $f(x)$ in $K^{\text{sep}}$ and the valuation of its leading coefficient (note that including the valuation of the leading coefficient in the data of a cluster picture strays from the usual convention in \cite{bisatt, M2D2, Faraggi-Nowell}).
In the context of tame curves and $p\neq 2$ cluster pictures are arguably better than MRM types: the cluster picture recovers both the MRM reduction type and the $\ell$-adic Galois representation (Theorem \ref{thm:CP1} and \cite[Thm.~1.20]{M2D2}). 

Moreover, just like the valuation of the $j$-invariant and the discriminant for elliptic curves, cluster pictures change in a very simple and predictable manner in field extensions. In particular, we get for free the following consequence of Theorem \ref{thm:main}. For an explicit illustration, see Example \ref{ex:introram}.

\begin{corollary}\label{cor:introram}
Suppose $p\neq 2$, and $C/K$ is a tame curve of genus 2. Then the MRM reduction type of~$C$ over a finite extension $L/K$ is determined by its MRM reduction type over $K$, its potential stable type, and the degree of $L/K$.
\end{corollary}

The obvious weakness of cluster pictures is that they are model dependent, although this is, of course, also true of the valuation of the discriminant in the elliptic curve case.
In the genus 2 case, we introduce a purely combinatorial equivalence relation on cluster pictures, with the property that a curve admits a $K$-model with a specific cluster picture if and only if it admits $K$-models for all the cluster pictures in the equivalence class. It turns out that there are 106 (families of) equivalence classes, and they correspond precisely to MRM reduction types together with the potential stable type. (They are the ``lists'' referred to in Theorem \ref{thm:main}(8).) 

For elliptic curves, it is very easy to pass between a model with discriminant of valuation $m$ and one with valuation $m\!+\!12n$. In the genus 2 case, the passage between the different cluster pictures in an equivalence class is more delicate, but also fully explicit. We have summarised it in Definition \ref{def:equiv-cpics} and Table \ref{ss:changeofmodel}, which, in particular, lets one find minimal Weierstrass equations; see also Example \ref{ex:changemodel}.

Cluster pictures have the further added benefit that one can easily read off the potential stable type (Theorem \ref{thm:CP1}): this is summarised in Table \ref{ss:changeofmodel}. The table is also a handy tool for finding the MRM reduction type from a cluster picture, as the potential stable type indicates which of the Tables \ref{table:potentiallygood}--\ref{table:potentiallymultiplicativex2} to use.

\subsection{Examples and applications}\label{subsec:examples}
We illustrate some applications our tables to the study of the local arithmetic of genus $2$ curves. For the reader's convenience we review the definition of cluster pictures (see \cite{M2D2}, especially Definitions 1.1, 1.5).

Let $C$ be a genus $2$ curve  over $K$ with Weierstrass equation $y^2 = f(x)$.  Write $\mathcal{R}$ for the set of roots of $f(x)$ in $K^{\text{sep}}$ and $c_f$ for its leading coefficient, that is, 
$$
    f(x) = c_f \prod_{r \in \cR} (x - r).
$$

A \textit{cluster} is a non-empty subset $\mathfrak{s}\subseteq\mathcal{R}$ cut out by a disc, that is $\mathfrak{s}=D\cap\mathcal{R}$ for some disc $D = \{x \in K^{\text{sep}} \mid v(x- z) \geq d\}$ with $z \in K^{\text{sep}}$ and $d \in \mathbb{Q}$. 
(We have here extended the valuation $v$ to $K^{\text{sep}}$, normalised so that $v(\pi) = 1$ for a uniformiser $\pi$ of $K$.)

A cluster $\mathfrak{s}$ with $|\mathfrak{s}| > 1$ is called \emph{proper}; for such a cluster, its \emph{(absolute) depth} $d_{\mathfrak{s}}$ is essentially the radius of its defining disc,
$$
    d_\mathfrak{s} = \textup{{min}}_{r,r'\in\mathfrak{s}} v(r-r').
$$
If $\mathfrak{s}\neq \mathcal{R}$, then its \emph{relative
depth} is $\delta_\mathfrak{s}=d_\mathfrak{s}-d_{P(\mathfrak{s})}$, where $P(\mathfrak{s})$ is the smallest cluster with $\mathfrak{s}\subsetneq P(\mathfrak{s})$. The \textit{cluster picture} of $y^2 = f(x)$ is the valuation $v_c:= v(c_f)$ of the leading coefficient and the collection of all clusters, with their depths, associated to $f$.

Following \cite{M2D2}, we draw cluster pictures by drawing the roots $r\in\mathcal{R}$ as red dots {\smash{\raise4pt\hbox{\clusterpicture 
\Root[D] {}{first}{r1} \endclusterpicture}}}, drawing ovals around the dots to represent clusters of size $>1$, and labelling the clusters $\fs \subsetneq \cR$ with their relative depth $\delta_{\mathfrak{s}}$ and $\cR$ with its absolute depth $d_{\cR}$. 

We refer to specific MRM reduction types by the labels in the Namikawa--Ueno classification, as these describe the special fibres of the curves.

\begin{example}[Computation with cluster pictures]
Cluster pictures are easy to work with for numerical examples. Consider the curve 
$$
C:y^2=17\cdot(x-1)(x-17^2)(x-17^4)(x-17^7)(x^2-17^{19})
$$ 
over $\mathbb{Q}_{17}^{\textup{unr}}$. Its cluster picture has five proper clusters:
$$
 \{ {17}^{\frac{19}{2}},-{17}^{\frac{19}{2}} \}
 \subset
 \{ {17}^{\frac{19}{2}},-{17}^{\frac{19}{2}}, 17^7 \}
 \subset
 \{ {17}^{\frac{19}{2}},-{17}^{\frac{19}{2}}, 17^7, 17^4 \}
 \subset
 \{ {17}^{\frac{19}{2}},-{17}^{\frac{19}{2}}, 17^7, 17^4, 17^2 \}
 \subset {\text{all roots}},
$$
with depths $\frac{19}{2}, 7, 4, 2$ and 0, respectively. It thus has cluster picture (the indices are the relative depths)
\begin{center}
     \raisebox{8pt}{\scalebox{1.5}{
\clusterpicture           
  \Root[D] {1} {first} {r1};
  \Root[D] {} {r1} {r2};
  \Root[D] {2} {r2} {r3};
  \Root[D] {2} {r3} {r4};
  \Root[D] {2} {r4} {r5};
  \Root[D] {1} {r5} {r6};
  \ClusterD c1[\frac{5}{2}] = (r1)(r2);
   \ClusterD c2[3] = (c1)(r3);
    \ClusterD c3[2] = (c2)(r4);
     \ClusterD c4[2] = (c3)(r5);
  \ClusterD c5[0] = (c4)(r6);
\endclusterpicture}}  \raisebox{8pt}{,\qquad $v_c = 1$.} 
\end{center}
Table \ref{ss:changeofmodel} shows that the curve has potential stable type VII, so its detailed information can be found in Table \ref{table:potentiallymultiplicativex2}. 
Note that as $p\!=\!17\neq 2,3,5$ the curve is automatically tame; see Lemma \ref{lem:cluster_pic_tame} for a general criterion.
In the notation of the latter table, we obtain $\frac{l}{2}\!=\!\frac{5}{2}$, $2t\!+\!\gamma\!=\!3$ (so $t\!=\!1, \gamma\!=\!1$), $\frac{m}{2}\!=\!2, s\!=\!2$. Checking the row ``$v_c\not\equiv s$'' the entry for $\gamma\!=\! 1$ (the boxed parameter) shows that the curve has MRM reduction type I$_5$-I$_4^*$-$1$, its minimal discriminant has valuation 27, its Jacobian has Néron component group $\Z/5 \Z \times \Z/ 2 \Z \times \Z / 2 \Z$, conductor exponent $3$, etc. \end{example}

\begin{example}[Realising reduction types]
Cluster pictures make it very easy to write down explicit genus~2 curves for any given reduction type, away from residue characteristic 2 and, in the wild cases, 3 or~5. 
For example, to construct a curve with MRM reduction type I$_l^*$-I$_q$-$t$, Table \ref{table:potentiallymultiplicativex2} tells us that we can take the curve $y^2\!=\!f(x)$ whose cluster picture is

\begin{center}
    \raisebox{8pt}{\scalebox{1.5}{
\clusterpicture           
  \Root[D] {1} {first} {r1};
  \Root[D] {} {r1} {r2};
  \Root[D] {2} {r2} {r3};
  \Root[D] {3} {r3} {r4};
  \Root[D] {} {r4} {r5};
  \Root[D] {2} {r5} {r6};
  \ClusterD c1[\frac{l}{2}] = (r1)(r2);
   \ClusterD c2[2t+1] = (c1)(r3);
    \ClusterD c3[\frac{q}{2}] = (r4)(r5);
  \ClusterD c5[0] = (c3)(c2)(r6);
\endclusterpicture}, $\qquad v_c \equiv 0 \mod 2$,}
\end{center}
so long as $p\neq 2$. We can take 
$$
C:y^2=((x-1)^2-\pi^{l+4t+2})(x-1-\pi^{2t+1})(x^2-\pi^q)(x-2),
$$ 
where $\pi$ is a uniformiser. (For a general recipe of producing polynomials for cluster pictures, see \cite[Construction 2.5]{bisatt}).
Constructing curves with a given reduction type in arbitrary residue characteristic is a difficult problem in general; for instance, to the best of our knowledge it is not yet known if all MRM reduction types of genus 2 curves can be realised in residue characteristic 2.
\end{example}

\begin{example}[BSD data]\label{ex:bsd}
Suppose we would like to investigate the Birch--Swinnerton-Dyer formula for the leading term at $s\!=\!1$ of the $L$-series of the Jacobian of $C:y^2\! =\! p\cdot x (x^5\! -\! p^3)$ over $\Q$, for some prime $p>5$. We can examine the local contribution from the prime $p$ using Theorem \ref{thm:main} and Table \ref{table:potentiallygood}.

The cluster picture is ${\scalebox{1.5}{
\clusterpicture           
  \Root[D] {1} {first} {r1};
  \Root[D] {} {r1} {r2};
  \Root[D] {} {r2} {r3};
  \Root[D] {} {r3} {r4};
    \Root[D] {} {r4} {r5};
    \Root[D] {} {r5} {r6};
  \ClusterD c4[{3}/{5}] = (r1)(r2)(r3)(r4)(r5)(r6);
\endclusterpicture}}$
with $v_c=1$, so the curve has MRM reduction type VIII-3. The N{\'e}ron component group $\Phi$ is trivial, so the local Tamagawa number is 1. Note that our tables do not list the Frobenius action on $\Phi$ (they assume that the residue field is algebraically closed), which would be needed to obtain the Tamagawa number in general.

The other contribution comes from the choice of periods, which are usually computed with respect to the standard exterior form $\omega_C=\frac{dx}{y}\wedge\frac{xdx}{y}$. The term appearing in the BSD formula is $|\omega_C/\omega_p^0|_p$, where $\omega_p^0$ is the N{\'e}ron exterior form at $p$ (see \cite{tate-bsd}; the fraction $\omega/\omega'$ is shorthand for $\lambda\!\in\! K^\times$ such that $\omega\!=\!\lambda\omega'$). As $\omega^0_p$ remains unchanged is unramified extensions, we can use Table \ref{table:potentiallygood} to determine $v(\omega_C/\omega_p^0)$ as follows. The valuation of the discriminant of our model is 28, and that of the minimal Weierstrass model is 18, resulting in $v(\omega_C/\omega_p^0)=-\frac{1}{10}(28-18)+v(\omega_{\text{min}}/\omega_p^0)=-2$ (see Lemma~\ref{relativevaluation}). Thus $|\omega_C/\omega_p^0|_p = p^2$.

Let us note in passing the genus 2 version of Ogg's formula (see Theorem  \ref{oggformula}): $\#$components on special fibre $+$ conductor exponent $-11v(\omega_{\text{min}}/\omega_p^0)$ = $v(\Delta_{\text{min}})+1$, which in this example reads 4+4+11=18+1. The formula is sensitive to using the minimal regular model (as opposed to MRNC) and the discriminant of the minimal Weierstrass equation. 
\end{example}

\begin{example}[Finding a minimal Weierstrass model]\label{ex:changemodel}
Consider the curve in Example \ref{ex:bsd}. Table \ref{table:potentiallygood} tells us that the curve can be represented by Weierstrass equations with cluster pictures
$$
 {\scalebox{1.5}{
\clusterpicture           
  \Root[D] {1} {first} {r1};
  \Root[D] {} {r1} {r2};
  \Root[D] {} {r2} {r3};
  \Root[D] {} {r3} {r4};
    \Root[D] {} {r4} {r5};
    \Root[D] {} {r5} {r6};
  \ClusterD c4[\frac{3}{5} + n] = (r1)(r2)(r3)(r4)(r5)(r6);
\endclusterpicture}} \quad {\text{with }}2|v_c, \qquad
 {\text{and}} \qquad 
 {\scalebox{1.5}{
\clusterpicture           
  \Root[D] {1} {first} {r1};
  \Root[D] {} {r1} {r2};
  \Root[D] {} {r2} {r3};
  \Root[D] {} {r3} {r4};
    \Root[D] {} {r4} {r5};
    \Root[D] {5} {r5} {r6};
  \ClusterLD c1[][s+\frac{2}{5}] = (r1)(r2)(r3)(r4)(r5);
  \ClusterD c4[m] = (c1)(r6);
\endclusterpicture}} \quad {\text{with }}2\nmid v_c,
$$
for all integers $n,m, s\in \Z$ and $s\ge 0$. The one corresponding to a minimal Weierstrass equation is ${\scalebox{1.5}{
\clusterpicture           
  \Root[D] {1} {first} {r1};
  \Root[D] {} {r1} {r2};
  \Root[D] {} {r2} {r3};
  \Root[D] {} {r3} {r4};
    \Root[D] {} {r4} {r5};
    \Root[D] {5} {r5} {r6};
  \ClusterLD c1[][2 / 5] = (r1)(r2)(r3)(r4)(r5);
  \ClusterD c4[0] = (c1)(r6);
\endclusterpicture}}$ with $v_c = 1$, indicated by a $\star$ in the table. 

As indicated in Table \ref{ss:changeofmodel}, to go from the cluster picture of $C \colon y^2 = p \cdot x(x^5 - p^3)$ to the one corresponding to a minimal Weierstrass equation, we apply the M\"{o}bius transformation in Proposition \ref{prop:cpic-ABC}(5A) with cluster $\fs = \{ 0 \}$. The M\"{o}bius transformation $\phi(x) = \frac{p}{x + p}$ gives rise to a change of model via $X = \phi(x)$ and $Y = p^{-2}\cdot X^3y$ so that $C' : Y^2 = p\cdot(X-1)(X^5 - p^2(1-X)^5)$ is a minimal Weierstrass equation.

\end{example}

\begin{example}[Change of MRM type in extensions]\label{ex:introram}
Let us illustrate how one can easily trace the change of the MRM reduction type in ramified extensions using Tables \ref{table:potentiallygood}--\ref{table:potentiallymultiplicativex2}, as noted in Corollary \ref{cor:introram}. An analogous result can be obtained for every reduction type of genus 2 curves. 
We remind the reader that this is not a property of curves in general, as the MRM reduction type might not even determine the potential stable type. 

Suppose $p\neq 2,5$. We claim that a curve with MRM reduction type VIII-1 obtains the following MRM reduction type over an extension of degree $e$: 
$$
\begin{array}{|c|c|c|c|c|c|c|c|c|c|c|}
\hline
e {\textnormal{ mod }} 10&1&2&3&4&\>\>\>5\>\>\>\>&6&7&8&9&\>\>\>0\>\>\>\> \cr\hline
{\textnormal{MRM reduction type}}&{\textnormal{VIII-1}}&{\textnormal{IX-1}}&{\textnormal{VIII-2}}&{\textnormal{IX-2}}&{\textnormal{I}}_{0-0-0}^*&{\textnormal{IX-3}}&{\textnormal{VIII-3}}&{\textnormal{IX-4}}&{\textnormal{VIII-4}}&{\textnormal{I}}_{0-0-0}\cr\hline
\end{array}
$$
To see this, consider Theorem \ref{thm:main} and Table \ref{table:potentiallygood}.
The curve is tame and admits a Weierstrass model with cluster picture ${\scalebox{1.5}{
\clusterpicture           
  \Root[D] {1} {first} {r1};
  \Root[D] {} {r1} {r2};
  \Root[D] {} {r2} {r3};
  \Root[D] {} {r3} {r4};
    \Root[D] {} {r4} {r5};
    \Root[D] {} {r5} {r6};
  \ClusterD c4[4 / 5] = (r1)(r2)(r3)(r4)(r5)(r6);
\endclusterpicture}}$ and $v_c=1$. Over an extension of ramification degree $e$ the cluster picture of this model becomes ${\scalebox{1.5}{
\clusterpicture           
  \Root[D] {1} {first} {r1};
  \Root[D] {} {r1} {r2};
  \Root[D] {} {r2} {r3};
  \Root[D] {} {r3} {r4};
    \Root[D] {} {r4} {r5};
    \Root[D] {} {r5} {r6};
  \ClusterD c4[{4e}/{5}] = (r1)(r2)(r3)(r4)(r5)(r6);
\endclusterpicture}}$ and $v_c=e$. The result now follows by looking through the table, since the depth of the cluster containing all six roots is taken mod 1 and the valuation of the leading coefficient mod~2. 

\end{example}

The final example is a theoretical application of our work. In \cite[Cor.\ 3.25 p.113]{lorenzini}, Lorenzini classified the possible N{\'e}ron component groups of principally polarised abelian surfaces with purely additive reduction, with the sole remaining question of this work being whether $\mathbb{Z}/2\mathbb{Z}\oplus \mathbb{Z}/4\mathbb{Z}$ is realised as the prime-to-$p$ part of the component group of such a surface with potentially good reduction,  where $p$ is the residue characteristic.

\begin{theorem} \label{Z/2Z/4}
The group $\mathbb{Z}/2\mathbb{Z}\oplus \mathbb{Z}/4\mathbb{Z}$ does not appear as the prime-to-$p$ part of the N{\'e}ron component group of a principally polarised abelian surface with potentially good reduction over any complete DVR of residue characteristic $p$.
\end{theorem}

\begin{proof}
Since $\mathbb{Z}/2\mathbb{Z}\oplus \mathbb{Z}/4\mathbb{Z}$ cannot be the prime-to-$2$ part, we can assume that the residue characteristic of $K$ is $\neq 2$. 

A principally polarised abelian surface is either isomorphic to (i) a product of two elliptic curves, (ii) the Weil restriction of an elliptic curve from a quadratic extension, or (iii) the Jacobian of a curve of genus 2 (see \cite[Thm.~3.1]{GGR}). In case (i), one of the elliptic curves would have to have reduction type I${}_4$ or I${}_n^*$ ($n$ odd), and hence the surface would not have potentially good reduction. Case (ii) cannot have N{\'e}ron component group $\mathbb{Z}/2\mathbb{Z}\oplus \mathbb{Z}/4\mathbb{Z}$ as N{\'e}ron component groups are unchanged by Weil's restriction of scalars \cite[Prop.~3.19]{LorenziniTT} and this group does not occur for elliptic curves. 

Finally, in case (iii) we use Theorem \ref{thm:main} as follows. Checking through Tables \ref{table:potentiallygood}--\ref{table:potentiallymultiplicativex2} we see that to have this N{\'e}ron component group the underlying curve would have to have MRM type 
I${}_{l-m-0}$, 
II${}_{l-m}$,
I${}_{l-m-n}$,
III-I${}_l$-t, 
III${}^*$-I${}^*_l$-t, 
III-I${}^*_l$-t, 
III${}^*$-I${}_l$-t, 
I${}_l$-I${}_m$-$t$ or
I${}^*_l$-I${}_m$-$t$.
Moreover, such a curve has tame reduction whenever $p\neq 2$, and hence its potential stable type must be either III, IV, VI or VII. In particular, the Jacobian does not have potentially good reduction (types I and V).
\end{proof}

\subsection{Outline of the paper}

In \S\ref{s:reduction}, we formally define the reduction type of a model of a curve and we describe how the MRNC reduction type determines the MRM reduction type (\Cref{MRNCtoMRMintro}), and vice versa in the genus $2$ case. This covers the proof of part (1) of Theorem \ref{thm:main}.

In \S\ref{s:clusters}, we prove part (8) of Theorem \ref{thm:main}, that our tables give the complete list of cluster pictures that gives rise to each reduction type in the tame case. We also detail the explicit $K$-isomorphisms to go between models of curves with different cluster pictures (Proposition \ref{prop:cpic-ABC}), and show that, for a curve with a given reduction type, each cluster picture is realised by some Weierstrass equation of the curve (Theorem \ref{thm:CP2}). 

In \S\ref{s:componentgroups}, we prove Theorem \ref{thm:introcomponentgroups}, which relates the Néron component group of the Jacobian of a general curve $C/K$ to the homology of a curve $C'$ defined over $\C[t]$ with the same reduction type as $C/K$. We use this to calculate the  Néron component groups associated to the reduction type families in our tables, giving part (2) of Theorem \ref{thm:main}.

In \S\ref{s:discriminant} we review results from literature used to determine the tame part of the conductor exponent, BSD fudge factor $v (\omega_{\min}/ \omega^0)$ and the valuation of the minimal discriminant.  This proves parts (3), (5) and (7) of Theorem \ref{thm:main}.

In \S\ref{mainthmproofsection} we complete the proof of \Cref{thm:main} and in \S\ref{s:tables}, we present the tables listing all possible reduction types of the minimal regular model of genus $2$ curves over complete discrete valuation fields, along with their associated invariants.

\begin{remark}
There are packages in SageMath and Magma for calculating cluster pictures and reduction types of genus $2$ curves \cite{bestvanbommel,magma,timredlib, sage}, which we have used throughout this project to check calculations done by hand. The website \cite{timredlib} of T. Dokchitser also gives a pictorial library of MRNC reduction types of genus $2$ curves, which, since we adopt his labelling for MRNC reduction types, can be used alongside our tables to obtain an explicit description of the special fibre. 
\end{remark}

\subsection{Notation, terminology and conventions}\label{ss:notation}

We use the following notation. 
\begin{align*}
& K && \textup{a complete discrete valuation field with an algebraically closed residue field;} \\
& R && \textup{the valuation ring of }K;  \\
& \pi && \textup{a uniformiser of }K; \\
& k && \textup{the residue field of }K; \\
& p && \textup{the characteristic of }k; \\
& v && \textup{the normalised valuation extended to } K^{\textup{sep}}\textup{ so that } v(\pi)=1; \\
& G_K && \textup{the absolute Galois group of }K; \\
& C && \textup{a smooth, projective, geometrically connected curve of genus }g\textup{ over }K; \\
& \textup{Jac}(C) && \textup{the Jacobian variety of }C; \\
&y^2=f(x) && \textup{a Weierstrass equation for } C,\textup{when it is a hyperelliptic curve}; \\
& v_c && v(c_f),\textup{ where }c_f\textup{ is the leading coefficient of }f(x);\\
& \Delta_{\min} && \textup{the discriminant of a minimal Weierstrass equation for a hyperelliptic curve } C \textup{ (Definition \ref{def:mindisc})};\\
& \omega^0 && \textup{the N\'eron exterior form when }C \textup{ has genus 2}; \\
& \omega_{\text{min}} && \frac{dx}{y}\wedge\frac{xdx}{y} \textup{ for a minimal Weierstrass model } y^2=f(x) 
\textup{ of a genus 2 curve } C.
\end{align*}

We use the phrase \textit{cluster picture} to refer to the classically defined cluster picture in \cite{M2D2} along with $v_c$ (see \S\ref{subsec:examples}).

We say a curve $C/K$ has \textit{tame reduction} if it obtains semistable reduction over a tamely ramified extension of $K$. 

By a \textit{Weierstrass equation} or a \textit{Weierstrass model} we always mean one in the shortened form $y^2=f(x)$.

We use the abbreviation \textit{MRM} for the minimal regular model and \textit{MRNC} for the minimal regular normal crossings model. For the terms \textit{MRM reduction type} and \textit{MRNC reduction type} see \S\ref{s:introreduction} or Definition \ref{def:reductiontype}.
(See \cite[\S3]{saito} or \cite[\S\S1--2]{timreduction} for the definition and an overview of MRNC models.)

By the \emph{potential stable type} of a genus 2 curve we mean the type of the special fibre of the stable model of $C/F$ (in the sense of \cite[Thm.\ 1]{liu-stabletypes}, see also Table \ref{ss:changeofmodel}), where $F/K$ is a finite extension over which $C$ becomes semistable.

We refer to MRM reduction types of genus 2 curves by the labels given in the Namikawa--Ueno classification. Lemma \ref{lem:MRMtoNU} shows that the MRM type determines the description of the special fibre in the Namikawa--Ueno classification, and hence the label except for two cases. For the two pairs of labels that have the same special fibre see Remark \ref{rmk:MRM-NU}(2) below.

\begin{remark}\label{rmk:MRM-NU}
We will often pass between the description of special fibres of minimal regular models, MRNC models, cluster pictures, the Namikawa--Ueno labels and our own Tables \ref{table:potentiallygood}--\ref{table:potentiallymultiplicativex2}. For clarity, we give a brief summary that compares these notions for genus 2 curves: 
\begin{enumerate}
\item
MRM and MRNC reduction types are in 1-to-1 correspondence (Theorem \ref{thm:SNCvsMRM},  Corollary \ref{cor:genus2MRNCtoMRM}). There are 104 families of MRNC types defined in \cite{timreduction}.

\item Tables \ref{table:potentiallygood}--\ref{table:potentiallymultiplicativex2} contain 106 families. They represent MRM reduction types together with the potential stable type {\em provided} $p\neq 2$. In practice, they are exactly the MRM reduction types, except that two families have been duplicated: II${}_{t-m}^*$ and 2I${}_m$-$t$ have the same MRM type but different potential stable type, and similarly for II${}_{t-0}^*$ and 2I${}_0$-$t$. For $p=2$ this splitting has no meaning: the data in the tables for $p=2$ is identical for each pair. These 106 families match the 106 families of equivalence classes of cluster pictures, although we will not attempt to axiomatise the latter concept of family (\Cref{thm:CP2}).

\item The Namikawa--Ueno classification contains 105 families, after one replaces their use of ``$\alpha$'' in labels by $-1$ for the corresponding parameter. 
These are precisely the 106 families in (2), except that 2I${}_l$-$t$ ($t>0$) and 2I${}_l$-$0$ are treated as one family. From the point of view of (2) they must be kept separate, as they have different potential stable types.
\end{enumerate}
\end{remark}

\begin{remark}\label{fieldofdefnremark}
Several references on cluster pictures, including \cite{bisatt, M2D2}, are stated for curves over local fields, where one can also keep track of Frobenius actions. In this paper, we work over a complete discrete valuation field with algebraically closed residue field. The results we use only involve the valuation, the configuration of discs containing the roots, and the tame inertia action; they do not use finiteness of the residue field or Frobenius. Accordingly, the relevant cluster picture arguments carry over unchanged to our setting. Any statements we reference that involve Frobenius are used only after forgetting the Frobenius action.
\end{remark}

\begin{remark}
 Throughout this paper we take $K$ to be a complete discrete valuation field. As pointed out to us by Qing Liu, this assumption can in fact be relaxed and \Cref{thm:main} holds for general discrete valuation fields with algebraically closed residue field. Indeed, although our proofs use completeness, the invariants classified are unchanged under taking completions. For MRM and MRNC models see \cite[Prop.\ 9.3.28 and Lem.\ 9.3.35]{Liu_book}; for minimal discriminants see \cite[Lem.\ 1.29]{liu-2026}; for N{\'e}ron model (and hence N{\'e}ron component groups) see \cite[\S7.2 Thm.\ 1(ii)]{neronmodels}; conductor exponents are defined by taking the completion; for $v(\omega_{\text{min}}/\omega^0)$ this follows from the analogue of Ogg's formula for genus 2 curves due to Ueno, Saito and Liu \cite{Liu} (see Theorem~\ref{oggformula}); for cluster pictures see \cite[Thm.\ 1.1]{lilybelle2}.
\end{remark}

\vspace{0.5em}

\noindent {\bf{Acknowledgements.}}
This project began at the {\em{Workshop on Arithmetic and Algebra of Rational Points}}, we thank Julian Lyczak and Ross Paterson for organising it.  We would also like to thank Peter Jossen, Qing Liu, Tim Dokchitser, Jakab Schrettner, Raymond van Bommel, Adam Morgan, Padma Srinivasan and David Angdinata for helpful conversations and comments, and the Heilbronn Institute for funding the {\em{Tamagawa Numbers of Curves}} research meeting, where some of the work on this project took place. 

The first author was supported by the Engineering and Physical Sciences Research Council [EP/S021590/1], the EPSRC Centre for Doctoral Training in Geometry and Number Theory (The London School of Geometry and Number Theory), University College London.
The second author was supported by EPSRC Postdoctoral Pathway Fellowship [EP/W524347/1].  The fourth author was supported by the EPSRC Doctoral Prize fellowship EP/W524335/1 and the extension UKRI3030.

\section{Reduction types}
\label{s:reduction}

In this section we formalise what we mean by MRM and MRNC reduction types, show that they are in 1-to-1 correspondence for genus 2 curves, and that they are essentially equivalent to the Namikawa--Ueno labels.

Kodaira and N\'eron developed a classification of reduction types of elliptic curves over DVRs \cite{kodaira,neron}. For genus 2 curves, the standard classification has become that of Namikawa and Ueno \cite{namikawaueno}. There are a few subtleties, however, as they worked over $\C[t]$ and the precise objects that they classified rely on this. We take a slightly different approach: for us the reduction type should be described purely combinatorially in terms of the special fibre and apply over any DVR. In view of the general usage of the Namikawa--Ueno classification, we have chosen a definition that as closely as possible recovers it (Lemma~\ref{lem:MRMtoNU}).

\begin{defn}\label{def:reductiontype}
For a curve $C/K$ with a regular model $\mathcal{C}/R$, by the {\em reduction type} of $\mathcal{C}$ we mean the following data:
\begin{itemize}
\item A set $X$ that enumerates the components of the special fibre $\mathcal{C}_k$, and functions that record their multiplicities $m:X\to  \Z_{\ge 1}$ and geometric genera $g: X\to \Z_{\ge 0}$;
\item The intersection pairing $X\times X\to \Z$;
\item A set $S$ that enumerates the singular points of $\mathcal{C}_k$, together with a function that records the number of times each component passes through this point $f:S\to \Z[X]$, and another that records whether the singularities are normal crossings.
\end{itemize}
The {\em MRNC reduction type} of $C/K$ is the reduction type of the minimal MRNC model of $C/R$. The {\em MRM reduction type} is the reduction type of the minimal regular model.
\end{defn}

\begin{remark}
The MRNC reduction type of a curve determines its N\'eron component group (Raynaud \cite{raynaud}), the tame part of its conductor exponent (\Cref{tameconductorformula}) and whether the curve has tame reduction (Saito \cite[Thm.\ 3.11]{saito}, \cite[Thm.\ 10.4.47]{Liu_book}), i.e. the invariants in part (2)--(4) of Theorem \ref{thm:main}. 
\end{remark}

\begin{theorem}\label{thm:SNCvsMRM}
The reduction type of any regular model of a curve determines its MRM reduction type.
\end{theorem}

\begin{proof}
The minimal regular model is obtained from a regular model by repeatedly blowing down genus 0 curves on the special fibre with self-intersection $-1$ (\cite[Thm.~9.3.8, Prop.~9.3.19]{Liu_book}). The reduction type of the regular model then determines the MRM reduction type by Lemma \ref{lem:contraction} below.

\end{proof}

\begin{lemma}\label{lem:contraction}
Suppose $C$ is given by a regular model $\mathcal{C}$ whose special fibre contains a genus 0 curve $E$ with self-intersection $-1$. Let $\mathcal{C}'$ be the regular model for $C$ obtained by contracting $E$. Then the reduction type of $\mathcal{C'}$ is determined by the reduction type of $\mathcal{C}$.
\end{lemma}

\begin{proof}
The set of components of $\mathcal{C}'_k$ (together with genera and multiplicities) is exactly the same as that of $\mathcal{C}_k$ with $E$ removed. The intersection pairing is determined by \cite[\S9 Exc.~2.9(d)]{Liu_book}.

The set of singular points is the same, except for those that lay on $E$: these points $x_1,... x_n$ are all merged into one, possibly non-singular, point $x$, and the number of times a component $\Gamma$ of $\mathcal{C'}$ passes through $x$ is given by the sum of the times it passed through the $x_i$ on $\mathcal{C}$.

The intersection pairing between two components $\Gamma_1$ and $\Gamma_2$ that intersect $E$ is related to that of their images under the contraction $\Gamma_1', \Gamma_2'$ by the formula $\Gamma_1\cdot\Gamma_2=\Gamma_1'\cdot\Gamma_2' - \mu_x(\Gamma_1')\mu_x(\Gamma_2')$ \cite[\S9 Exc.~2.9(d)]{Liu_book}. Here $\mu_x(\Gamma')$ is the multiplicity of the point $x$ on $\Gamma'$; it is 1 if an only of $x$ is smooth on $\Gamma'$. In particular, the point $x$ is a transversal intersection on $\mathcal{C}_k'$ if and only if $E$ intersects exactly two other components $\Gamma_1, \Gamma_2$ with $E\cdot \Gamma_i\!=\!1$ (so that $\mu_x(\Gamma_1)=\mu_x(\Gamma_2)=1$) and that do not intersect at the same point of $E$.

The point $x$ is non-singular on $\mathcal{C}_k'$ if and only if $E$ intersects exactly one other component $\Gamma$ and $E\cdot \Gamma\!=\!1$, as follows from the formula $\Gamma\cdot E = \mu_x(\Gamma')$ \cite[\S9 Exc.~2.9(d)]{Liu_book}.

\end{proof}

We now turn to curves of genus 2. We begin by showing that the MRM reduction type, in as far as it is possible, determines the Namikawa--Ueno label of the curve.

\begin{lemma}\label{lem:MRMtoNU}
Let $C/K$ be a genus 2 curve whose special fibre is described by one of the labels in the Namikawa--Ueno classification. This description is uniquely recovered from the MRM reduction type of $C$ by the following process. Here we refer singularities that are not normal crossings as ``bad''.
\begin{enumerate}
\item[(1)] A bad singularity that appears on a single component is always a cusp, except for 
  \begin{itemize}
  \item a (2,5) cusp when there is a unique component, which has geometric genus 0, a unique singular point and passes once through it (type {\rm{VIII-1}});
  \item a point with intersection multiplicity 2 when there is a unique component, which has geometric genus 0, a unique singular point and passes twice through it (type {\rm{III-II}}${}_0$);
  \end{itemize}
\item[(2)] A bad singularity that appears on two components is always a point with intersection multiplicity 2, except for 
  \begin{itemize}
  \item a point with intersection multiplicity 3 when there are exactly two components and these have intersection number 3 (type {\rm{V}});
  \item a cusp intersecting another component when there are exactly two components, intersection number 2, and they pass once each through the unique singular point (type {\rm{VII}});
  \item a node intersecting another component when there are exactly two components, intersection number 2, and one of them passes twice through the unique singular point (type {\rm{IV-II}}${}_0$);
  \end{itemize}
\item[(3)] A bad singularity that appears on three components always corresponds to three pairwise transversal intersections, except if there are exactly three components and these intersect with multiplicities 1, 1 and 2, in which case one component passes transversally through a multiplicity 2 intersection (type {\rm{IX-2}});
\item[(4)] ``elliptic'' and ``regular'' labels correspond to (geometric) genus 1 and 2, respectively;
\item[(5)] A,B,D labels correspond to self-intersection $-1$,$-3$ and $-4$, respectively; 
label C corresponds to a multiplicity 1 component with self-intersection $-2$ that is either 
genus 1 (type {\rm{I}}${}_0$-{\rm{I}}${}_0^*$-$m$), or 
genus 0 that has two nodes (type {\rm{I}}${}_{n-p-0}$), or 
genus 0 carrying the cusp/node at the bad singularity in the configurations for type {\rm{VII}} and {\rm{IV-II}}${}_0$ described in (2) above.
\end{enumerate}
\end{lemma}

\begin{proof}
This is a simple case-by-case check through the list of Namikawa and Ueno in \cite{namikawaueno}.
\end{proof}

It is an important feature of MRM and MRNC reduction types that they are purely combinatorial objects and do not require specifying a curve or the underlying DVR. Since the MRNC reduction type determines the MRM type, to establish an explicit 1-to-1 correspondence it suffices to take one numerical example for each MRNC reduction type and check the resulting MRM type. In \cite{namikawaueno}, Namikawa and Ueno have provided an explicit example of a curve over $\C[t]$ for each of their labels, which we can use for this purpose.
(We will use this technique again in the context of cluster pictures.)

\begin{proposition}\label{prop:NUclusterpictures}
For each family of examples of genus 2 curves over $\C[t]$ (at $t\!=\!0$) in the Namikawa--Ueno classification, its
\begin{itemize}
\item MRM reduction type,
\item MRNC reduction type, and
\item cluster picture
\end{itemize}
match under the explicit correspondence given in Tables \ref{table:potentiallygood}--\ref{table:potentiallymultiplicativex2}.
This list of examples exhausts all possible MRNC reduction types.
\end{proposition}

\begin{proof}
This is done by an explicit verification.
The original paper by Namikawa and Ueno \cite{namikawaueno} lists the special fibre, and hence the MRM reduction type, for each family.
We have explicitly computed the cluster picture in each case: this is elementary, as it only requires one to determine the roots of the underlying polynomials.
The MRNC reduction type is obtained from the cluster picture by the recipe given in \cite{Faraggi-Nowell} Thms. 1.6, 1.12; the list for genus 2 curves was already worked out in the Appendix of \cite{SarahThesis}.
This resulting list exhausts all possible MRNC reduction types, as given in \cite[Table G2]{timreduction}. 
(Both \cite{SarahThesis} and \cite{timreduction} associate a Namikawa--Ueno label to each MRNC reduction type, but neither gives a justification.) 
\end{proof}

\begin{corollary}\label{cor:genus2MRNCtoMRM}
For curves $C/K$ of genus 2, the MRM reduction type determines the MRNC reduction type. The explicit correspondence is given in Tables \ref{table:potentiallygood}--\ref{table:potentiallymultiplicativex2}. 
\end{corollary}

\begin{proof}
This follows from Theorem \ref{thm:SNCvsMRM} and Proposition \ref{prop:NUclusterpictures}.
\end{proof}

\begin{remark}
It follows from the above corollary and the classification of MRNC types at the end of \cite{timreduction} that the Namikawa--Ueno classification covers all possible MRM types of genus 2 curves in all characteristics and over all complete DVRs, at least with our crude definition of MRM reduction type. (We do not claim that each reduction type occurs in every residue characteristic, as this does not appear to be currently known for MRNC types.)
\end{remark}

\section{Cluster pictures}
\label{s:clusters}

For a hyperelliptic curve $C\colon y^2=f(x)$, the cluster picture, introduced by Dokchitser, Dokchitser, Maistret and Morgan in \cite{M2D2}, records the $v$-adic configuration of the differences of roots of $f(x)$. Cluster pictures were used in \cite{M2D2} to compute the MRM reduction type of semistable hyperelliptic curves and many other local invariants. Faraggi and Nowell later proved that the  cluster picture also determines the MRNC reduction type for hyperelliptic curves with tame reduction \cite{Faraggi-Nowell}. We note that, although some of the cluster picture results cited in this section are stated for local fields, the results from \cite{bisatt,M2D2,Faraggi-Nowell} that we use apply in our setting by Remark~\ref{fieldofdefnremark}.

Throughout this section we assume that the residue characteristic $p$ of $K$ satisfies $p\ne 2$.

\begin{notation}
Recall from \S\ref{subsec:examples} that the \emph{cluster picture} associated to a Weierstrass equation $C \colon y^2 = f(x)$ of a hyperelliptic curve is the valuation $v_c$ of the leading coefficient of $f(x)$ together with the collection of all \emph{clusters} associated to the roots $\cR$ of $f$. 
We write $\Sigma_f$, or $(\Sigma_f, v_c)$, to refer to the cluster picture of $f$, and write $\fs \in \Sigma_f$ when $\fs \subset \cR$ is a cluster.  The absolute depth of a proper cluster $\fs$ is denoted $d_{\fs}$, and its relative depth, when $\fs\ne\cR$, is denoted $\delta_{\fs}=d_{\fs}-d_{P(\fs)}$.
\end{notation}

It will be useful to separate the combinatorial content of a cluster picture from any particular polynomial realising it. We therefore also use \emph{abstract cluster pictures}, introduced in \cite[Def.~13.1]{M2D2}, where the roots are replaced by an abstract finite set and the cluster picture is recorded as a collection of subsets, together with their depths and the valuation of the leading coefficient. This point of view makes clear that the results below depend only on the combinatorial data, rather than on the choice of field or on a particular Weierstrass equation. The formal definition is given in Definition~\ref{def:abstract_cpic}. The cluster pictures displayed in Tables~\ref{table:potentiallygood}--\ref{table:potentiallymultiplicativex2} are to be understood in this abstract sense.

The main results of this section are as follows.

\begin{theorem}\label{thm:CP1}
Suppose $K$ has residue characteristic $p \ne 2$, and let $C: y^2 = f(x)$ be a hyperelliptic curve over $K$. Let $\Sigma_f$ be the associated abstract cluster picture. 
\begin{enumerate}
    \item If $C / K$ has tame reduction, then the cluster picture $\Sigma_f$ determines the MRNC and MRM reduction types of $C / K$.
    \item If $C$ is a  genus $2$ curve, then the underlying cluster shape, obtained from $\Sigma_f$ by forgetting all depth labels and the valuation of the leading coefficient, determines the potential stable type of $C/K$. This assignment is given in Table~\ref{ss:changeofmodel}; for quintic models, we use the associated degree $6$ cluster shape obtained as in Definition~\ref{def:equiv-cpics}(3).
\end{enumerate}
\end{theorem}

Among abstract cluster pictures on sets of size $2g+1$ or $2g+2$, we single out those which arise as cluster pictures of genus $g$ hyperelliptic curves with tame reduction. This subset is independent of the choice of the ground field; see Remark~\ref{rem:K-indep}. We call its elements \emph{admissible genus $g$ cluster pictures} (Definition~\ref{def:cpic-from-curve}). In Definition~\ref{def:equiv-cpics}, we define an equivalence relation on the set of admissible genus $g$ cluster pictures. This equivalence relation accounts for the dependence of cluster pictures on the choice of Weierstrass equation, and captures how cluster pictures change when one replaces a Weierstrass
equation by a $K$-isomorphic one.

\begin{theorem}\label{thm:CP2}
    Definition \ref{def:equiv-cpics} defines an equivalence relation on the set of admissible genus $g$ cluster pictures. This equivalence relation has the following properties. 
    \begin{enumerate}
        \item Let $K$ have residue characteristic $p\ne 2$, and let $C/K$ be a genus $g$ hyperelliptic curve with tame reduction and cluster picture $(\Sigma_f,v_c)$. Then $C$ admits Weierstrass equations over $K$ realising every admissible genus $g$ cluster picture in the equivalence class of $(\Sigma_f,v_c)$.
        \item When $g = 2$, the equivalence classes are in one-to-one correspondence with pairs (MRM reduction type, potential stable type). The correspondence is defined by sending the class of a cluster picture $(\Sigma,v_c)$, arising from a Weierstrass equation $C\colon y^2=f(x)$ with tame reduction, to the MRM reduction type and potential stable type of $C / K$. This does not depend on the representative cluster picture, on the choice of curve realising it, or on the choice of $K$. 
    \end{enumerate}
\end{theorem}

Finally, we explain how the cluster-picture entries in Tables~\ref{table:potentiallygood}--\ref{table:potentiallymultiplicativex2} are to be read, and verify that they are complete. The tables display only degree $6$ cluster pictures. This is not a restriction: after applying an integral shift to all depths, any quintic cluster picture is equivalent to a degree $6$ cluster picture by Definition~\ref{def:equiv-cpics}(3). The tables are also read up to shifting all depths by an integer (which preserves relative depths), as in Definition~\ref{def:equiv-cpics}(2).

\begin{theorem}\label{thm:CP3}
The cluster picture entries in Tables~\ref{table:potentiallygood}--\ref{table:potentiallymultiplicativex2} have the following interpretation.
\begin{enumerate}

\item Suppose $K$ has residue characteristic $p\ne 2$. Let $C/K$ be a genus $2$ curve with tame reduction and Weierstrass equation $C\colon y^2=f(x)$. If necessary, first replace its cluster picture by the equivalent degree $6$ cluster picture using Definition~\ref{def:equiv-cpics}(3).
Then, up to an integral shift of all depths, this cluster picture
appears in Tables~\ref{table:potentiallygood}--\ref{table:potentiallymultiplicativex2}.

\item If a genus $2$ curve over a field $K$ with residue characteristic $p \not=2$ has one of the displayed cluster pictures, then its potential stable type is the stable type of the table containing that entry. If the curve has tame reduction, then its MRM reduction type is the type heading the column containing that entry.

\item After specialising the parameters in a column to allowed values, the cluster pictures displayed in that column are precisely the degree $6$ cluster pictures in a single equivalence class of admissible genus $2$ cluster pictures, after identifying pictures which differ by an integral shift of all depths.
\end{enumerate}

\end{theorem}

We use abstract cluster pictures, following \cite[Def.~13.1]{M2D2} (where we also include the leading coefficient). 

\begin{defn}\label{def:abstract_cpic}
Let $\cR$ be a finite set and let $\Sigma$ be a collection of non-empty subsets of $\cR$.
The elements of $\Sigma$ are called \emph{clusters}. Attach a depth $d_{\fs}\in\mathbb{Q}$ to every
cluster $\fs\in\Sigma$ with $|\fs|>1$ (such clusters are called \emph{proper}). Let $v_c \in \Z$. Then $(\Sigma, v_c)$, or $(\Sigma,\cR,d, v_c)$, is called an \emph{abstract
cluster picture} if:
\begin{enumerate}
    \item every singleton $\{x\}$, for $x\in \cR$, lies in $\Sigma$, and $\cR\in\Sigma$;
    \item any two clusters are either disjoint or one contains the other;
    \item if $\ft\subsetneq\fs$ and both are proper, then $d_{\ft}>d_{\fs}$.
\end{enumerate}
The \emph{parent} $P(\fs)$ of a cluster $\fs \not= \cR$ is the smallest cluster such that $\fs \subsetneq P(\fs)$, a maximal subcluster $\fs' \subsetneq \fs$ of $\fs$ is referred to as a \emph{child} of $\fs$. The \emph{relative depth} of a proper cluster $\fs\ne\cR$ is $\delta_{\fs}:=d_{\fs}-d_{P(\fs)}$. 

Two cluster pictures $(\Sigma,\cR,d, v_c)$ and $(\Sigma',\cR',d', v_c')$ are isomorphic if $v_c = v_c'$ and there is a bijection $\iota:\cR\to \cR'$ that induces a bijection $\Sigma\to\Sigma'$ and such that $d_{\fs}=d'_{\iota(\fs)}$ for every cluster $\fs\in\Sigma$ with $|\fs|>1$.
\end{defn}

Given a model $C\colon y^2=f(x)$, the associated abstract cluster picture is obtained from the cluster picture described in the beginning of \S\ref{subsec:examples} by forgetting the particular labels of the roots.

\subsection{Cluster pictures determine reduction types}

We first prove that the cluster picture determines the reduction types appearing
in \Cref{thm:CP1}.

\begin{proof}[Proof of \Cref{thm:CP1}]

\begin{enumerate}[(1)]
\item That the abstract cluster picture determines the MRNC reduction type follows from \cite[Thm.~1.2]{Faraggi-Nowell}. Faraggi and Nowell phrase their result in
terms of the \emph{minimal strict normal crossings} reduction type. However, their convention does not require the irreducible components of the special fibre to be smooth. In practice, what they actually compute is the MRNC reduction type in the sense of \Cref{def:reductiontype}.
Faraggi and Nowell also include the $G_K$-action on $\Sigma_f$ and signs $\epsilon_{\fs} \in \{\pm 1\}$ for clusters $\fs \in \Sigma_f$ (see \cite[Def.\ 3.9]{Faraggi-Nowell}) as part of the input data to determine the reduction type. For curves with tame reduction over $K$, the $G_K$-action consists of a cyclic inertia action, and this action is determined by the cluster picture depths, see \cite[Thm.\ 1.3]{bisatt}. The signs $\epsilon_{\fs}$ are computed from the depths in the cluster picture and $v(c_f)\bmod 2$. 
The MRM reduction type follows from \Cref{thm:SNCvsMRM}.

\item Now assume that $C / K$ has genus $2$. Choose a finite extension $F/K$ over which $C$ has semistable reduction.
The stable model is independent of the choice of such an extension
$F$, since stable models commute with base change over fields where $C$ is
semistable. Passing from $K$ to $F$ does not change the underlying cluster shape; it only multiplies all depths by $[F \colon K]$. Thus it suffices to prove the assertion after replacing $K$ by a finite extension over which $C$ is semistable. 

The stable model over $F$ is obtained from the semistable minimal regular model by contracting the genus $0$ components of self-intersection $-2$ in the special fibre \cite[Prop.\ ~9.4.8, Thm.\ 10.3.34]{Liu_book}.
The assignment of the underlying cluster shapes in Table~\ref{ss:changeofmodel} to potential stable types is obtained as follows. For each underlying cluster shape of size $6$, choose semistable depth data for that shape, apply \Cref{thm:CP1}(1) or \cite[Thm.~18.7,~18.8]{M2D2} to obtain the MRM reduction type, and compute the stable model. 
The depth parameters determine only the lengths of chains and loops in the semistable MRM reduction type. These lengths do not affect the stable model after the above contractions. Hence the resulting potential stable type depends only on the underlying cluster shape.
\end{enumerate}
\end{proof}


\subsection{Change of models}

\begin{defn}\label{def:cpic-from-curve} 
An abstract cluster picture $(\Sigma,\cR,d,v_c)$ is an \emph{admissible genus $g$ cluster picture} if there exists a hyperelliptic curve $C/K$ of genus $g$ with tame reduction and a Weierstrass equation $C \colon y^2=f(x)$ such that, writing  $\cR_f$ for the roots of $f$ and $c_f$ for its leading coefficient, $(\Sigma, \cR,d,v_c)\cong (\Sigma_f,\cR_f,d_f,v(c_f))$. In particular, $|\cR| = 2g + 1$ or $2g + 2$. Here $K$ is any complete discrete valuation field with algebraically closed residue field of characteristic $p \not= 2$. 
\end{defn}

\begin{remark}\label{rem:K-indep}
The property of being admissible is independent of the choice of $K$. Indeed, by \cite[Thm.~2.4]{bisatt}, this property is determined by the abstract cluster picture itself. When an abstract cluster picture is admissible, it can be realised by a hyperelliptic curve over any such $K$; see \cite[Thm.~1.7]{bisatt}.

Let $(\Sigma,\cR,d,v_c)$ be an admissible genus $g$ cluster picture. If $C\colon y^2=f(x)$ is a realisation as above, then $K(\cR_f)/K$ is tame, hence cyclic, and $\Gal(K(\cR_f)/K)$ acts by automorphisms on the associated abstract cluster picture. After choosing an isomorphism $(\Sigma,\cR,d,v_c)\cong(\Sigma_f,\cR_f,d_f,v(c_f))$,
we can consider this action on $(\Sigma,\cR,d,v_c)$. By Bisatt's description of the tame Galois action on clusters, the resulting conjugacy class of cyclic subgroups of $\Aut(\Sigma,\cR,d,v_c)$ is determined by the depth data; see \cite[Thm.~1.3]{bisatt}. In what follows, we write $G$ for a representative of this conjugacy class. All constructions below are understood up to isomorphism of abstract cluster pictures, and hence do not depend on this choice.
\end{remark}

\begin{defn}[{\cite[Def.~1.2]{bisatt}}]
Let $(\Sigma,\cR,d,v_c)$ be an admissible cluster picture, with associated cyclic action $G$ as in \Cref{rem:K-indep}. Let $\fs$ be a cluster. Write $G_{\fs}=\{\sigma\in G:\sigma(\fs)=\fs\}$ for the stabiliser of $\fs$. A child $\fs'$ of $\fs$ is called an \emph{orphan} of $\fs$ if $\fs'$ is the unique child of $\fs$ that is fixed by $G_{\fs}$.
\end{defn}

The following definition is an adaptation of the equivalence relation of \cite[Def.~14.1]{M2D2} to the tame genus $2$ setting. The moves below are motivated by elementary changes of Weierstrass equation: scaling $x$ and $y$, adding or removing a root at infinity, and applying an inversion on the $x$-line. Applying these transformations to a Weierstrass equation gives explicit formulae for the new depths and leading coefficient; the definition below records these formulae abstractly. For genus $2$ curves with tame reduction, these moves are sufficient to establish the one-to-one correspondence between equivalence classes and the possible
combinations of MRM reduction type and potential stable type. We do not claim that this equivalence relation is the correct one for such a statement in arbitrary genus.

\begin{defn}
\label{def:equiv-cpics}
Let $(\Sigma,\cR,d,v_c)$ and $(\Sigma',\cR',d',v_c')$ be admissible genus $g$ cluster pictures, where $(\Sigma, \cR, d, v_c)$ has associated cyclic action $G$ as in \Cref{rem:K-indep}.
We say that they are \emph{equivalent} if $(\Sigma',\cR',d',v_c')$ is isomorphic to a cluster picture obtained from $(\Sigma,\cR,d,v_c)$ by a finite sequence of the following moves.

\begin{enumerate}[(1)]
    \item \emph{Changing $v_c$ by an even integer:}  Set $\cR' = \cR$, $\Sigma' = \Sigma$, $d_{\ft}' = d_{\ft}$ for all proper clusters $\ft \in \Sigma$, and $v_c' = v_c + 2m$ for some $m \in \Z$.
    
    \item \label{item:shift_depths} \emph{Shifting depths by $n \in \Z$:} Set $\cR' = \cR$, $\Sigma' = \Sigma$, $d_{\ft}' = d_{\ft} + n$ for all proper clusters $\ft \in \Sigma$, and $v_c' = v_c - |\cR|n$.

    \item \label{item:add_root} \emph{Adding a root when $|\cR|$ is odd and $d_{\cR} > 0$:}
    Set $\cR'=\cR\cup\{r_0\}$ where $r_0 \not\in \cR$ and set
    $\Sigma'=\Sigma\cup\{\{r_0\},\cR'\}$. Set $d'_{\cR'} = 0$, $d'_{\ft} = d_{\ft}$ for all proper clusters $\ft \in \Sigma$, $v_c' = v_c$.

    \item \label{item:remove_root} \emph{Removing a $G$-fixed root $r \in \cR$ when $|\cR|$ is even, $d_{\cR} \in \Z$, $P(\{r\}) = \cR$, and $\cR \setminus \{r\} \in \Sigma$:} Set $\cR'=\cR\setminus\{r\}$. Set $\Sigma'=\Sigma\setminus\{\cR,\{r\}\}$. Set $d'_{\cR'} = d_{\cR \setminus \{r\}}$, $d'_{\ft} = d_{\ft}$ for all proper clusters $\ft \subseteq \cR'$, and $v_c' = v_c +(|\cR| + 1)d_{\cR}$. 

    \item \emph{Redistributing depth when $|\cR|$ is even:}
    Here, we use the convention that $d_{\ft}=+\infty$ if $\ft$ is a singleton, and that clusters of relative depth $0$ are omitted from the resulting cluster picture.
    In each case below, the move is allowed only if the displayed depths define an abstract cluster picture and the displayed value of $v_c'$ is an integer.

    For (A), (B) and (C), let $\fs$ be a $G$-stable cluster of $\Sigma$ with $P(\fs)=\cR$, and set $\fs^c=\cR\setminus\fs$. We also use the convention that $\delta_{\fs^c}=0$ if $\fs^c\notin\Sigma$ (i.e. $d_{\fs^c}=d_{\cR}$).
    \begin{enumerate}[(A)]
        \item \label{item:moveA} \emph{Redistributing integral depth:} 
         Let $j\in\Z$ satisfy $d_{\cR}\le j\le d_{\fs}$. Set $\cR' = \cR$, $\Sigma' = \Sigma \cup \{\fs^c \}$. Set $d_{\cR'}' = 0$, $d'_{\ft} = d_{\ft} - j$ for proper clusters $\ft \subseteq \fs$ and $d'_{\ft} = d_{\ft} + j - 2d_{\cR}$ for proper clusters $\ft \subseteq \fs^c$. Set $ v_c' =  v_c+|\fs|j+|\fs^c|d_{\cR}.$ 

        \item \label{item:moveB} \emph{Deleting $\fs$ from $\Sigma$ when $|\fs| > 1$ and $\fs$ has no orphan:}
        Set $\cR' = \cR$, $\Sigma' = (\Sigma\setminus \{\fs\}) \cup \{\fs^c\}$. Set $d_{\cR'}' = -d_{\fs}$, $d_{\ft}' = d_{\ft} - 2d_{\fs}$ for proper clusters $\ft \subseteq \fs$, and $d_{\ft}' = d_{\ft} - 2d_{\cR}$ for proper clusters $\ft \subseteq \fs^c$. Set $v_c' = v_c+|\fs|d_{\fs}+|\fs^c|d_{\cR}$.

        \item \emph{Deleting $\fs$ from $\Sigma$ when $\fs$ has an orphan $\fs'$, and if $\fs'$ is proper, $\fs'$ has no orphan:} Let $h\in\Z$ satisfy $h\geq d_{\fs}$, and set $ m=\min(h,d_{\fs'})$. Set $\cR'=\cR$, $\Sigma'=(\Sigma\setminus\{\fs\})\cup \{\cR'\setminus\fs',\fs^c\}$. Set $d'_{\cR'}=h-m$,  $d'_{\cR'\setminus\fs'}=h-d_{\fs}$. Set $d'_{\ft}=d_{\ft}+h-2m$ for proper clusters $\ft\subseteq\fs'$, $d'_{\ft}=d_{\ft}+h-2d_{\fs}$ for proper clusters $\ft\subseteq\fs\setminus\fs'$, and $d'_{\ft}=d_{\ft}+h-2d_{\cR}$ for proper clusters  $\ft\subseteq\fs^c$.
        Set  $v_c'=v_c+|\fs'|m+|\fs\setminus\fs'|d_{\fs}+|\fs^c|d_{\cR}$.

        \item \label{item:moveD} \emph{Inverting the top cluster when $\cR$ has no orphan:} Set $\cR'=\cR$, $\Sigma'=\Sigma$. Set $d'_{\cR'}=-d_{\cR}$, and for every proper cluster $\ft\subsetneq\cR$, set $d'_{\ft}=d_{\ft}-2d_{\cR}$. Set $v_c'=v_c+|\cR|d_{\cR}$.
    \end{enumerate}
\end{enumerate}
\end{defn}

One can check, using \cite[Thm.~2.4]{bisatt}, that under the hypotheses of each elementary move, the displayed data again define an admissible cluster picture. This also follows from the explicit changes of Weierstrass equation in \Cref{prop:cpic-ABC}.

\begin{lemma}\label{lem:equiv-is-equivalence}
    Definition \ref{def:equiv-cpics} defines an equivalence relation on the set of admissible genus $g$ cluster pictures.
\end{lemma}

\begin{proof}
Reflexivity follows from the empty sequence of moves, and transitivity follows by concatenating sequences. It remains to prove symmetry. It is enough to show that each elementary move is reversible, and that the reverse sequence consists of moves whose hypotheses are satisfied. 

\begin{enumerate}[(1)]
    \item \Cref{def:equiv-cpics}(1) is reversed by applying \Cref{def:equiv-cpics}(1) with $-m$.
    \item \Cref{def:equiv-cpics}(2) is reversed by applying \Cref{def:equiv-cpics}(2) with $-n$. 
    \item \Cref{def:equiv-cpics}(3) is reversed by \Cref{def:equiv-cpics}(4) applied to the newly added $G$-fixed root. 
    \item Suppose \Cref{def:equiv-cpics}(4) removes a $G$-fixed root $r$ from a top cluster $\cR$ of even size. Let $\cR_0=\cR\setminus\{r\}$. After the removal, first apply \Cref{def:equiv-cpics}(2) with $n = -d_{\cR}$. Then the top depth is $d_{\cR_0}-d_{\cR}>0$, so \Cref{def:equiv-cpics}(3) may be applied to add back the fixed root $r$. Finally apply \Cref{def:equiv-cpics}(2) again with $n = d_{\cR}$. This restores all depths. The resulting value of $v_c$ differs from the original one by $|\cR|d_{\cR}$, which is even since $|\cR|$ is even and $d_{\cR}\in\Z$; hence \Cref{def:equiv-cpics}(1) restores $v_c$.
    \item 
    Note first that if $\fs$ is a $G$-stable cluster with $P(\fs) = \cR$, then its complement $\fs^c = \cR \setminus \fs$ is also $G$-stable. 
    \begin{enumerate}[(A)]
    \item The inverse of \Cref{def:equiv-cpics}\textup{(5A)} is as follows.
    If $d_{\cR}\in\Z$, apply \Cref{def:equiv-cpics}(5A) to $\fs^c$ with parameter $j-d_{\cR}$, then apply Definition~\ref{def:equiv-cpics}(2) with $n = d_{\cR}$, and finally apply Definition~\ref{def:equiv-cpics}(1) if necessary.
    If $d_{\cR} \notin\Z$, then $G$ acts non-trivially on the children of $\cR$ by \cite[Thm.~1.3(iii)]{bisatt}, and so $\fs$ is the orphan of $\cR$ (it is $G$-fixed and hence unique by \cite[Rmk. A.8]{bisatt}). After the move, $\fs^c$ has no orphan. Apply \Cref{def:equiv-cpics}(5B) to $\fs^c$ and then \Cref{def:equiv-cpics}(2) with $n = j$.
    
    \item The inverse of \Cref{def:equiv-cpics}\textup{(5B)} is as follows. 
    If $d_{\cR}\in\Z$, apply \Cref{def:equiv-cpics}(5A) to $\fs^c$ with parameter $-d_{\cR}$, then apply \Cref{def:equiv-cpics}(2) with $n = d_{\cR}$ and use \Cref{def:equiv-cpics}(1) if necessary.
    If $d_{\cR}\notin\Z$, then  $\fs$ is the orphan of $\cR$ (as in (A)), and the inverse is \Cref{def:equiv-cpics}(5B) applied to $\fs^c$ (which has no orphan).
    
    \item The inverse of \Cref{def:equiv-cpics}\textup{(5C)} is as follows. Suppose the move is applied to $\fs$ with orphan $\fs'$ and parameter $h$. Let $ U=\cR\setminus \fs'$ be the cluster added by the move. In the transformed cluster picture, $\fs^c$ is the orphan of $U$. If $d_{\cR}\in\Z$, apply \Cref{def:equiv-cpics}(5C) to $U$, with orphan $\fs^c$, and with parameter $h-d_{\cR}$. Then apply \Cref{def:equiv-cpics}(2) with $n=d_{\cR}$. This restores all depths, and \Cref{def:equiv-cpics}(1) restores $v_c$ if necessary. Suppose now that $d_{\cR}\notin\Z$. Then $\fs$ is the orphan of $\cR$. In particular, in the transformed picture, $\fs^c$ has no orphan. Choose an integer $H$ with $H>h-d_{\cR}$. Apply \Cref{def:equiv-cpics}$(5C)$ to $U$, with orphan $\fs^c$, and with parameter $H$. Then apply \Cref{def:equiv-cpics}(2) with $n=h-H$. This restores all depths. The resulting value of $v_c$ differs from the original one by $|\cR|H$, which is even as $|\cR|$ is even; hence \Cref{def:equiv-cpics}(1) restores $v_c$.
    
    \item  \Cref{def:equiv-cpics}(5D) is self-inverse. 
    \end{enumerate}
\end{enumerate}
\end{proof}

We now show that, when $C\colon y^2=f(x)$ is a hyperelliptic curve over $K$ with tame reduction, the moves that can be applied to the associated admissible cluster picture, as in Definition~\ref{def:equiv-cpics}, are realised by $K$-rational changes of variables between Weierstrass equations. Thus every admissible cluster picture obtained from the one associated to $f$ by one of these moves is realised by a Weierstrass equation of a curve $K$-isomorphic to $C$.

Before realising the moves geometrically, we first recall \emph{centres} of clusters. 

\begin{defn}[{\cite[Def.\ 1.10]{M2D2}}]
Let $\Sigma_f$ be the cluster picture associated to a Weierstrass equation $C \colon y^2 = f(x)$ of a hyperelliptic curve $C  / K$. A \emph{centre} $z_{\fs}$ of a cluster $\fs \in \Sigma_f$ with $|\fs| > 1$ is any element $z_{\fs} \in K^{\text{sep}}$ such that $v(z_{\fs} - r) \geq d_{\fs}$ for all $r \in \fs$. If $\fs = \{r\}$ is a singleton, $z_{\fs} = r$. 
\end{defn}

When $|\fs| > 1$, $\fs$ is the intersection of $\cR$ with the closed disc $D = \{ x \in K^{\text{sep}} \mid v(x - z_{\fs}) \geq d_{\fs} \}$. In practice, a centre can often be found by taking an average of the
roots in $\fs$.

\begin{lemma}[{\cite[Lem.\ B.1]{M2D2}}]\label{lem:rational-centre}
    Let $\Sigma_f$ be the cluster picture arising from a Weierstrass equation $C \colon y^2 = f(x)$ of a hyperelliptic curve $C  / K$. 
    Assume $K(\cR) / K$ is tame, where $K(\cR)$ is the splitting field of $f$, and let $G = \Gal(K(\cR) / K)$. If $\fs$ is $G$-stable then there exists $z_{\fs} \in K$. 
\end{lemma}

The following lemma follows by direct computation. 

\begin{lemma}[M\"obius changes of Weierstrass equation]
\label{lem:mobius-change}
Let $C\colon y^2=f(x)$ be a hyperelliptic curve over $K$, with
$f(x)=c_f\prod_{r\in\cR}(x-r)$
and $|\cR|$ even.

\begin{enumerate}[(1)]
\item Suppose that $\phi(x)=\dfrac{ax+b}{cx+d}\in \PGL_2(K)$ has no pole on $\cR$. After the change of variables
$X=\phi(x)$, $Y=\mu (a-cX)^{|\cR|/2}y$, with $\mu\in K^\times$,
we obtain a $K$-isomorphic curve with Weierstrass equation
$$C'\colon Y^2=g(X), \qquad g(X)=\mu^2 c_f\prod_{r\in\cR}(cr+d)(X-\phi(r)).$$
The roots are $\phi(\cR)$, and the leading coefficient $c_g$ of $g$ satisfies
$v(c_g)=v(c_f)+2v(\mu)+\sum_{r\in\cR}v(cr+d)$.
\item Suppose that $r\in\cR\cap K$ and $\phi(x)=\dfrac{\pi^k}{x-r}$, with $k\in\Z$. Then $\phi$ sends $r$ to infinity. After the change of variables $X=\phi(x),  Y=\mu X^{|\cR|/2}y$ with  $\mu\in K^\times$,
we obtain a $K$-isomorphic curve with odd degree Weierstrass equation
$$C'\colon Y^2=g(X),\qquad g(X)=\mu^2 c_f\pi^k \prod_{r'\in\cR\setminus\{r\}}(r-r')(X-\phi(r')).$$
The roots are $\phi(\cR\setminus\{r\})$, and the leading coefficient $c_g$ of $g$ satisfies $v(c_g)=v(c_f)+2v(\mu)+k+ \sum_{r'\in\cR\setminus\{r\}}v(r-r').$
\end{enumerate}
\end{lemma}

In the following proposition, saying that an operation in Definition~\ref{def:equiv-cpics} is \emph{realised} by a change of variables means the following. Starting from a Weierstrass equation $C\colon y^2=f(x)$ whose associated admissible cluster picture satisfies the hypotheses of the operation, the displayed $K$-rational change of variables gives a $K$-isomorphic curve $C'$ with Weierstrass equation $C'\colon Y^2=g(X)$ whose associated admissible cluster picture is isomorphic to the one prescribed by that operation. Since the change of variables is defined over $K$, the isomorphism of cluster pictures is compatible with the relevant Galois actions.

Although Definition~\ref{def:equiv-cpics}(5) only allows a move when the displayed value of $v_c'$ is an integer, this integrality condition is automatic for the moves realised below. In each case, the formula for $v_c'$ is obtained from the valuation of the leading coefficient of the transformed Weierstrass equation. Since the transformed equation is defined over $K$, this valuation is an integer.

\begin{proposition}\label{prop:cpic-ABC}
Let $C\colon y^2=f(x)$ be a hyperelliptic curve with tame reduction over $K$.
Let $\cR$ be the set of roots of $f$, let $v_c=v(c_f)$, and set $G=\Gal(K(\cR)/K)$. 
The operations in Definition~\ref{def:equiv-cpics} that can be applied to the associated admissible cluster picture are realised by the following $K$-rational changes of variables between Weierstrass equations.

\begin{enumerate}[(1)]
\item \emph{Changing $v_c$ by an even integer:}
The operation in Definition~\ref{def:equiv-cpics}(1) is realised by
$X=x, Y=\pi^m y$, for some $m\in\Z$.

\item \emph{Shifting depths by $n\in\Z$:}
The operation in Definition~\ref{def:equiv-cpics}(2) is realised by $X=\pi^n x, Y=y$.

\item \emph{Adding a root when $|\cR|$ is odd and $d_{\cR}>0$:}
There exists $r_0\in K$ such that $v(r_0-r)=0$ for all $r\in\cR$, and set $\phi(x)=r_0+\frac{1}{x-r_0}$.
Then the operation in Definition~\ref{def:equiv-cpics}(3) is realised by $X=\phi(x)$, $Y=(X-r_0)^{(|\cR|+1)/2}y$.

\item \emph{Removing a root:}
If $r\in K$ is a root of $f(x)$ satisfying the hypotheses of Definition~\ref{def:equiv-cpics}(4), then the operation is realised by $X=\phi(x)$, $Y=X^{|\cR|/2}y$, where $\phi(x)=\dfrac{\pi^{2d_{\cR}}}{x-r}$.

\item \emph{Redistributing depth when $|\cR|$ is even:}
The operations in Definition~\ref{def:equiv-cpics}(5) are realised as follows. In cases (A), (B) and (C), let $\fs$ be a $G$-stable cluster with $P(\fs)=\cR$, and set  $\fs^c=\cR\setminus\fs$.

\begin{enumerate}[(A)]
\item \emph{Redistributing integral depth:}
Let $j\in\Z$ satisfy $d_{\cR}\le j\le d_{\fs}$, and choose a centre $z_{\fs}\in K$. There exists $\lambda\in K$ with $v(\lambda)=j$ such that $v(r-z_{\fs}+\lambda)=j$ for $r\in\fs$ and $v(r-z_{\fs}+\lambda)=d_{\cR}$ for $r\in\fs^c$.
For any such $\lambda$, the operation in Definition~\ref{def:equiv-cpics}(5A) is realised by $X=\phi(x)$, $Y=X^{|\cR|/2}y$, where $\phi(x)=\dfrac{\pi^j}{x-z_{\fs}+\lambda}$.

\item \emph{Deleting $\fs$ from $\Sigma_f$ when $|\fs|>1$ and $\fs$ has no orphan:}
Choose a centre $z_{\fs}\in K$. Then the operation in Definition~\ref{def:equiv-cpics}(5B) is realised by $X=\phi(x)$, $Y=X^{|\cR|/2}y$, where $\phi(x)=\dfrac{1}{x-z_{\fs}}$.

\item \emph{Deleting $\fs$ from $\Sigma_f$ when $\fs$ has an orphan $\fs'$, and if $\fs'$ is proper, $\fs'$ has no orphan:}
Let $h\in\Z$ satisfy $h\ge d_{\fs}$, and set $m=\min(h,d_{\fs'})$. Then there exists $z\in K$ such that $v(r-z)= \begin{cases} m, & r\in\fs',\\[-0.2em] d_{\fs}, & r\in\fs\setminus\fs',\\[-0.2em]
d_{\cR}, & r\in\fs^c. \end{cases}$
For any such $z$, the operation in Definition~\ref{def:equiv-cpics}(5C) is realised by $X=\phi(x)$, $Y=X^{|\cR|/2}y$, where $\phi(x)=\dfrac{\pi^h}{x-z}$.

\item \emph{Inverting the top cluster when $\cR$ has no orphan:}
Choose a centre $z_{\cR}\in K$. Then the operation in Definition~\ref{def:equiv-cpics}(5D) is realised by $X=\phi(x)$, $Y=X^{|\cR|/2}y$, where $\phi(x)=\dfrac{1}{x-z_{\cR}}$.

\end{enumerate}
\end{enumerate}
\end{proposition}

\begin{proof}
The assertions about the Galois actions and the integrality of the displayed values of $v_c'$ follow from the discussion preceding the proposition. It remains to check the clusters, depths and leading coefficient valuations.

(1),(2) Clear. 

(3) By Lemma~\ref{lem:rational-centre} we can choose a centre $z_{\cR}\in K$. As $d_{\cR}>0$, we have $v(z_{\cR}-r)>0$ for all $r\in\cR$, so taking $r_0=z_{\cR}+1$ (for instance) gives $v(r_0-r)=0$ for all $r\in\cR$.
Let $\phi(x)=r_0+\frac{1}{x-r_0}$. The change of variables $X=\phi(x), Y=(X-r_0)^{(|\cR|+1)/2}y$ gives a Weierstrass equation $Y^2 = g(X) = c_f (X-r_0)\prod_{a\in\cR}(r_0-a)(X-\phi(a))$.
Thus the new roots are $\{r_0\}\cup\phi(\cR)$. For distinct $a,b\in\cR$, $v(\phi(a)-\phi(b)) = v(a-b)-v(a-r_0)-v(b-r_0) = v(a-b)$, while $v(\phi(a)-r_0)=0$ for all $a\in\cR$. Hence $\phi$ preserves all depths among the old roots, and adds a new root $r_0$ outside the cluster $\phi(\cR)$, with top depth $0$. This is the cluster picture prescribed by Definition~\ref{def:equiv-cpics}(3). Finally, $v(c_g)=v(c_f)+\sum_{a\in\cR}v(r_0-a)=v(c_f)$.

(4) Since $P(\{r\})=\cR$ and $\cR\setminus\{r\}$ is a cluster, we have $v(a-r)=d_{\cR}$ for all $a\in\cR\setminus\{r\}$. Hence for $a,b\in\cR\setminus\{r\}$,
$v(\phi(a)-\phi(b))=2d_{\cR}+v(a-b)-v(a-r)-v(b-r)=v(a-b)$.
Thus the cluster picture on the remaining roots is unchanged, while $r$ is sent
to infinity. By Lemma~\ref{lem:mobius-change},
$v_c' =  v_c+2d_{\cR}+\sum_{a\in\cR\setminus\{r\}}v(a-r) = v_c+(|\cR|+1)d_{\cR}$. We now prove (5).
	\begin{enumerate}[(A)]
	\item By Lemma \ref{lem:rational-centre}, we can choose $z_{\fs} \in K$. We prove the existence of such a $\lambda$. Take $\lambda = \pi^j u$ with $u \in K$ such that $v(u) = 0$. For $r \in \fs$,  $v(r - z_{\fs}) \geq d_{\fs} \geq j$, so $v(r-z_{\fs} + \lambda) = j$, except possibly when $v(r-z_{\fs}) = j$ and cancellation occurs. Similarly for $r \in \fs^c$, $v(r - z_{\fs}) = d_{\cR} \leq j$, so $v(r - z_{\fs} + \lambda) = d_{\cR}$, except possibly if cancellation occurs. Since only finitely many residue classes of $u$ can cause cancellation, we choose $u$ avoiding all of them. The valuation of the leading coefficient for the model $C'$ follows from Lemma \ref{lem:mobius-change}, and the depths of the cluster picture follow from the fact that $v(\phi(r) - \phi(r')) = j + v(r - r') - v(r - z_{\fs} + \lambda) - v(r' - z_{\fs} + \lambda)$ for distinct $r, r' \in \cR$. 
    
	\item  By Lemma \ref{lem:rational-centre}, we can choose $z_{\fs} \in K$. The centre $z_{\fs}$ satisfies $v(r-z_{\fs})= d_{\fs}$ for $r \in \fs$ and $v(r - z_{\fs})= d_{\cR}$ for $r \in \fs^c$. Indeed, suppose that $v(r-z_{\fs})>d_{\fs}$ for some $r\in\fs$. Then the set $ \ft=\{a\in\fs: v(a-z_{\fs})>d_{\fs}\}$ is the unique child of $\fs$ containing $r$. Since $z_{\fs}\in K$, $\ft$ is also $G$-stable and hence an orphan, a contradiction (if a cluster has a $G$-stable child, it has exactly one, see \cite[Rmk.~A.8]{bisatt}).  The rest follows from Lemma \ref{lem:mobius-change} and the fact that $v(\phi(r) - \phi(r')) = v(r - r') - v(r - z_{\fs}) - v(r' -z_{\fs})$ for distinct $r, r' \in \cR$.
 
	\item As $\fs'$ is $G$-stable, we may choose a centre $z_{\fs'}\in K$ by Lemma \ref{lem:rational-centre}. If $h<d_{\fs'}$, let $z=z_{\fs'}+\pi^h u$, where $u \in K$ with $v(u) = 0$ is chosen to avoid the finitely many residue classes which cause cancellation, as in (A). Then $v(r-z)=h=m$ for $r\in\fs'$, while $v(r-z)=d_{\fs}$ for $r\in\fs\setminus\fs'$ and $v(r-z)=d_{\cR}$ for $r\in\fs^c$.

    Now suppose $h\geq d_{\fs'}$ (then $d_{\fs'} < \infty$ and $\fs'$ is proper). As $\fs'$ has no orphan, the same argument as in (B) shows that $v(r-z_{\fs'})=d_{\fs'}=m$ for all $r\in\fs'$. Taking $z=z_{\fs'}$, we also have $v(r-z)=d_{\fs}$ for $r\in\fs\setminus\fs'$ and $v(r-z)=d_{\cR}$ for $r\in\fs^c$.
    
    The rest follows from Lemma \ref{lem:mobius-change} and the fact that $v(\phi(r) - \phi(r')) = h + v(r - r') - v(r - z) - v(r' -z)$ for distinct $r, r' \in \cR$.

    \item  By Lemma \ref{lem:rational-centre}, we can choose $z_{\cR} \in K$. Since $\cR$ has no orphan, the same argument as in (B) shows that $v(r-z_{\cR})=d_{\cR}$ for every $r\in\cR$. The rest follows from Lemma \ref{lem:mobius-change} and the fact that $v(\phi(r)-\phi(r')) = v(r-r')-2d_{\cR}$ for distinct $r, r' \in \cR$. 
\end{enumerate}
\end{proof}

We now prove the remaining main results of this section.

\begin{proof}[Proof of \Cref{thm:CP2}(1)]
That \Cref{def:equiv-cpics} defines an equivalence relation follows from \Cref{lem:equiv-is-equivalence}. 

Let $(\Sigma',\cR',d',v_c')$ be an admissible genus $g$ cluster picture in the equivalence class of $(\Sigma_f,\cR_f,d_f,v_c)$. By definition, there is a finite sequence of elementary moves in Definition~\ref{def:equiv-cpics} taking $(\Sigma_f,\cR_f,d_f,v_c)$ to $(\Sigma',\cR',d',v_c')$, up to isomorphism of admissible cluster pictures. Each elementary move is realised by a $K$-rational change of Weierstrass equation by Proposition~\ref{prop:cpic-ABC}. Applying these changes successively gives a Weierstrass equation of $C$ over $K$ whose associated abstract cluster picture is isomorphic to $(\Sigma',\cR',d',v_c')$.
\end{proof}

\begin{proof}[Proof of \Cref{thm:CP3}]
\begin{enumerate}[(1)]
    \item Bisatt's classification lists all admissible tame cluster pictures of size $6$; see \cite[App.~C]{bisatt}. A finite comparison of the conditions in \cite[App.~C]{bisatt} with the conditions displayed in Tables~\ref{table:potentiallygood}--\ref{table:potentiallymultiplicativex2} shows that the degree $6$ pictures in the tables are exactly the admissible degree $6$ genus $2$ cluster pictures with tame reduction. Thus every admissible degree $6$ picture appears in the tables, and every displayed picture is admissible.
    If one starts with a quintic model, then after shifting all depths by an integer if necessary, Definition~\ref{def:equiv-cpics}(3) replaces its cluster picture by the associated equivalent degree $6$ picture.

    \item We first note a finite column-by-column verification. Fix a column of the tables and specialise its parameters to allowed values. The cluster pictures displayed in this specialised column are connected by the moves of Definition~\ref{def:equiv-cpics}; this is recorded schematically in Table~\ref{ss:changeofmodel} (\Cref{s:tables} explains how to read this table). The arrows and loops indicate the relevant moves of Definition~\ref{def:equiv-cpics}(5), possibly composed with integral shifts of all depths as in Definition~\ref{def:equiv-cpics}(2).
    Substituting the specialised parameters into the displayed depth data verifies the hypotheses of the indicated moves. Hence the pictures in a fixed specialised column lie in a single equivalence class.

    Each specialised column contains the cluster picture of a Namikawa--Ueno example from \Cref{prop:NUclusterpictures}. By \Cref{prop:NUclusterpictures}, this example has the MRM reduction type heading the column. Its potential stable type is the stable type of the table containing the column, by \Cref{thm:CP1}(2). Since every other picture in the column is equivalent to the Namikawa--Ueno picture, Theorem~\ref{thm:CP2}\textup{(1)} shows that the latter is realised by a Weierstrass equation of a curve $K$-isomorphic to $C$. Thus every curve with a cluster picture displayed in the column has the potential stable type and MRM reduction type indicated by the position of that entry.

    \item We have already shown that the pictures in a fixed specialised column lie in one equivalence class. Conversely, let $(\Sigma,\cR,d,v_c)$ be a degree $6$ admissible genus $2$ cluster picture equivalent to one of the pictures in that column. By (1), after identifying pictures which differ by an integral shift of all depths, it appears somewhere in the tables. By \Cref{thm:CP2}(1), equivalent admissible cluster pictures are realised by Weierstrass equations of $K$-isomorphic curves, so they have the same MRM reduction type and potential stable type. By (2), the table and column in which $(\Sigma,\cR,d,v_c)$ appears must therefore be the chosen specialised column. Hence the pictures displayed in the column are precisely the degree $6$ cluster pictures in a single equivalence class, up to integral shifts of all depths.
\end{enumerate}
\end{proof}

\begin{proof}[Proof of Theorem~\ref{thm:CP2}(2)]
It remains to prove (2). Let $\mathcal E$ be an equivalence class of admissible genus $2$ cluster pictures. Choose a representative $(\Sigma,\cR,d,v_c)$ of $\mathcal E$, realised by a tame genus $2$ curve $C\colon y^2=f(x)$. We send $\mathcal E$ to the pair consisting of the MRM reduction type and the potential stable type of $C/K$.

This is well-defined. If we replace $(\Sigma,\cR,d,v_c)$ by another representative of the same equivalence class, then by
Theorem~\ref{thm:CP2}(1) this new representative is realised by a Weierstrass equation of a curve $K$-isomorphic to $C$. Hence the MRM reduction type and potential stable type are unchanged. If we replace $C$ by another curve realising the same abstract cluster picture, then the MRM reduction type and potential stable type are unchanged by Theorem~\ref{thm:CP1}. The construction is independent of the choice of $K$, since admissibility is independent of $K$ by Remark~\ref{rem:K-indep}.

By Theorem~\ref{thm:CP3}(1), every admissible genus $2$ cluster picture is represented in the tables after the stated degree $6$ replacement and depth-shift convention. By Theorem~\ref{thm:CP3}(3), after specialising the parameters in a column, the displayed cluster pictures in that column are precisely the degree $6$ representatives in one equivalence class. By Theorem~\ref{thm:CP3}(2), the table containing the column gives the potential stable type, and the column itself gives the MRM reduction type. Hence equivalence classes of admissible genus $2$ cluster pictures are in one-to-one correspondence with pairs consisting of an MRM reduction type and a potential stable type.
\end{proof}

\begin{corollary}\label{cor:isom-implies-equivalent}
Suppose $K$ has residue characteristic $p \ne 2$. If two tame genus $2$ curves are $K$-isomorphic, then the cluster pictures associated to any Weierstrass equations of $C$ and $C'$ are equivalent.
\end{corollary}

\begin{proof}
Isomorphic curves have the same MRM reduction type and potential stable type. The result follows from Theorem \ref{thm:CP2}.
\end{proof}

\begin{remark}
Equivalently, for tame genus $2$ curves, the moves in Definition~\ref{def:equiv-cpics} generate all changes of cluster pictures that can occur between Weierstrass
equations of $K$-isomorphic curves.
\end{remark}

\begin{remark} 
  In \cite{SarahThesis}, Nowell introduced the notion of a \emph{metric open quotient BY tree}. These trees are in one-to-one correspondence with metric cluster pictures of hyperelliptic curves with tame reduction and top cluster depth $d_{\cR} \geq 0$ (\cite[Thm.\ 5.1.2]{SarahThesis}). Nowell defined an equivalence relation on these trees and showed that two curves are $K$-isomorphic if and only if their corresponding trees are equivalent. Thus, an alternative approach to proving Theorem \ref{thm:CP2} would be to compute the open quotient BY tree for each cluster picture and show that they are equivalent precisely when specified in our tables.
  
  In practice, it appears that this approach requires as much computation as the method presented here.  
\end{remark}

\section{N{\'e}ron component groups}\label{s:componentgroups}
Let $A/K$ be an abelian variety and $\mathcal{A}/R$ the N{\'e}ron model of $A$. The \textit{N{\'e}ron component group} of $A$ is
$$
\Phi_A = \mathcal{A}_k/\mathcal{A}_k^0,
$$
where $\mathcal{A}_k$ is the special fibre of $\mathcal{A}$ and $\mathcal{A}_k^0$ is its connected component.

A number of methods have been developed for controlling N\'eron component groups \cite{boschliu, edixhoven, lorenzini}. In the case where $A$ is the Jacobian of a curve $C$, Raynaud's theorem \cite[Thm. 1.1]{boschliu} shows that the component group can be recovered from the special fibre of a regular model of the curve. In particular, the N{\'e}ron component groups can be computed by manipulating the intersection pairing matrix, they only depend on the MRM reduction type of the curve, and hence it is a natural invariant attached it. In particular, for our classification problem, we have the freedom to choose our DVR and our curves.

In the present section we prove \Cref{thm:introcomponentgroups} and  use it to prove the classification of N\'eron component groups of Jacobians of genus 2 curves postulated in \Cref{thm:main}(2). A similar approach (without proofs) was used by Liu \cite{liu-models} to compute component groups associated to genus $2$ reduction types. This section recover's Liu's results and provides a (proved) framework for computing component groups more generally.

\subsection{Monodromy and Galois actions} 

Let $C \rightarrow \mathcal{D}$ be a smooth and proper family of complex curves of genus $g$ over a punctured disc $\mathcal{D}$. For any $t \in \mathcal{D}$, the smooth fibre $C_{t}$ is a compact Riemann surface. There is an automorphism of singular homology 
$$T : H_{1}(C_{t}, \Z) \rightarrow H_{1}(C_{t}, \Z) $$
induced by looping around $t=0$. Moreover, dualising we obtain an automophism
$$T^{\wedge} : H^{1}_{\text{sing}}(C_{t}, \Z) \rightarrow H^{1}_{\text{sing}}(C_{t}, \Z) $$
called the \textit{monodromy} of $C_{t}$. 

Algebraically, the generic fibre of $C$ can be viewed as a curve over $K = \C((t))$. The absolute Galois group $G_K = \Gal(\bar{K}/K)$ naturally acts on $H_{\text{{\'et}}}^1(C_{\Bar{t}}, \Z_{\ell})$ for every prime $\ell$. Artin's comparison theorem gives an isomorphism of $\Z_{\ell}$-lattices, see \cite[Ch. 3, Thm. 4.1]{etalcoh}, 
$$H_{\text{{\'et}}}^{1}(C_{\bar{t}}, \Z_{\ell}) \cong H^{1}_{\text{sing}}(C_{t}, \Z) \otimes \Z_{\ell}$$
and we show that the actions of monodromy and tame inertia coincide under this isomorphism. As we could not find any reference to this in literature, we give a proof for completeness.

\begin{theorem} \label{monodromyinertia}
    Let $C \rightarrow \mathcal{D}$ be a smooth and proper family of complex curves over a punctured disc $\mathcal{D}$ and $t \in \mathcal{D}$. For any prime $\ell$,  the action of monodromy on $H_{\text{sing}}^1(C_{t}, \Z) \otimes \Z_{\ell} $  conicides with the action of some  generator of tame inertia on $H_{\text{{\'et}}}^{1}(C_{\bar{t}}, \Z_{\ell})$ under the canonical isomorphism, up to conjugation. 
\end{theorem}

\begin{proof}
  The singular cohomology groups $H^{1}_{\text{sing}}(C_t, \Z_\ell)$ form a local system of $\Z_\ell$-modules, which define a representation 
  $$\rho_1 : \pi_1(\mathcal{D}, t) \rightarrow  \Aut(H^1_{\text{sing}}(C_t, \Z_\ell))$$
  by \cite[Cor. 2.6.2]{szamuely}.
  The monodromy operator $T$ acting on $H^1_{\text{sing}}(C_t, \Z_\ell))$ is given by $T = \rho( [ \gamma])$, where $\gamma$ is a loop generating $ \pi_1(\mathcal{D}, t)$.

  Similarly, $H_{\text{{\'et}}}^{1}(C_{\bar{t}}, \Z_{\ell})$ corresponds to a locally constant {\'e}tale sheaf, and thus defines a representation 
  $$\rho_{2}:\pi_1^{\textup{{\'et}}}(\mathcal{D}, \bar{t}) \rightarrow \Aut ( H_{\text{{\'et}}}^{1}(C_{\bar{t}}, \Z_{\ell})).$$
 As $\mathcal{D}$ is connected $\pi_1^{\textup{{\'et}}}(\mathcal{D}, \bar{t}) \cong \widehat{\pi_{1}(\mathcal{D}, t)}$ by  \cite[Thm. 4.6.7]{szamuely}. Combining this  with the natural extension of $\rho_1$ to the pro-finite completion  $\widehat{\pi_{1}(\mathcal{D}, t)}$ and the comparison theorem, we obtain a commutative diagram 

 \[
\begin{tikzcd}
\widehat{\pi_{1}(\mathcal{D}, t)} \arrow{r}{\widehat{\rho_{1}}} \arrow[swap]{d}{\cong} & \Aut ( H_{\text{sing}}^{1}(C_{t}, \Z_{\ell})) \arrow{d}{\cong} \\
\pi_1^{\textup{{\'et}}}(\mathcal{D}, \bar{t}) \arrow{r}{\rho_{2}} & \Aut( H_{\text{{\'et}}}^{1}(C_{\bar{t}}, \Z_{\ell})). \end{tikzcd}
\]
Recall that   $\pi_1^{\textup{{\'et}}}(\mathcal{D}, \bar{t}) \cong \Gal(\Bar{K}/K)$ by \cite[Prop. 4.6.1 and Def. 4.6.3]{szamuely} and  $G_K\cong \hat{\Z} \cong I_K^{\text{tame}}.$
Let $[ \gamma ]$ be the generator of $ \widehat{\pi_{1}(\mathcal{D}, t)} $ which maps to monodromy under $\hat{\rho}_1$. The image of this under the vertical isomorphism on the left in the above diagram is a generator of $\pi_1^{\textup{{\'et}}}(\mathcal{D}, \bar{t})$, and hence a generator of the tame inertia subgroup under the previous identifications. As the diagram commutes the claim follows

\end{proof}

\begin{corollary}  \label{mondromyandtatemod}
Let $C \rightarrow \mathcal{D}$ be a smooth and proper family of complex curves over a punctured disc $\mathcal{D}$. For any $t_0 \in \mathcal{D}$ and prime $\ell$, the action of monodromy on $H_1(C_{t_0}, \Z)$ and the action of a generator of tame inertia on the Tate module of the Jacobian of $C_{t_0}$ coincide, up to conjugation. 
\end{corollary}

\begin{proof}
This follows from the fact that both the action of monodromy on $H_1(C_{t_0}, \Z)$ and tame inertia on the Tate module of the Jacobian are obtained by dualising the actions on cohomology from \Cref{monodromyinertia}.
\end{proof}

\subsection{Computing N{\'e}ron component groups}
To compute component groups of genus $2$ curves, we use the matrices of the monodromy action classified by Namikawa and Ueno, alongside the following result. 

\begin{lemma}[{\cite[Lemma 2.1]{lorenzini}}] \label{genoftame}
Let $\ell\neq p$ be a prime. Let $C/K$ be a curve and $M_{\tau}$ the matrix representing the action of the generator of the tame inertia group acting on $T_\ell(\text{Jac}(C))$. 
Then the $\ell$-primary part of $\Phi_C$ is isomorphic to the torsion subgroup of the co-invariants of $M_{\tau}$. Explicitly,  
$$
\Phi_C [\ell^{\infty} ] \cong \Z_{\ell}/e_1 \Z_{\ell} \times \dots \Z_{\ell} /e_s \Z_{\ell},
$$
where $e_1, \ldots, e_s$ are the non-zero diagonal entries of the Smith normal form of $M_{\tau} -I_{2g}$.
\end{lemma}

\begin{theorem} \label{compgroupthm} 
Let $C$ be a curve of genus $g$ over $K$, and let $\Tilde{C}$ be a curve of genus $g$ over $\C[t]$. Suppose $C$ and $\Tilde{C}_{t=0}$ have the same MRM reduction type. Then the N{\'e}ron component group of $C$ is isomorphic to the torsion subgroup of the co-invariants of the monodromy of $\Tilde{C}$.
\end{theorem}
\begin{proof}
First we note that the N{\'e}ron component group is determined by the intersection pairing matrix (by Raynaud's theorem \cite[Thm. 1.1]{boschliu}) and hence it depends only on the MRM reduction type. Thus, by assumption, the N{\'e}ron component groups of $C$ and $\Tilde{C}_{t=0}$ coincide. 

To compute the component group of $\Tilde{C}_{t=0}$, we view this as a curve over $\C((t))$. As the residue characteristic is $0$, it follows from \Cref{genoftame} that we can recover the component group from the action of the generator of tame inertia on $T_{\ell}(\Jac(C)) $ for all primes $\ell$. Base changing $\Tilde{C}$ to a punctured disc, our claim follows by \Cref{mondromyandtatemod}.  

\end{proof}

\begin{theorem}\label{thm:g2_comp_gp}
Let $C/K$ be a genus $2$ curve. The N{\'e}ron component group $\Phi_C$ is determined by the MRM reduction type of $C$ and it is given in Tables \ref{table:potentiallygood} -\ref{table:potentiallymultiplicativex2}. 
\end{theorem}

\begin{proof} By \Cref{compgroupthm}, computing the component group reduces to determining the Smith normal form of $1-M$, where $M$ is the monodromy matrix classified by Namikawa-Ueno\footnote{The monodromy matrix for type $\text{II}^*-\text{II}_n^*$ stated in \cite{namikawaueno} has a typo, the $-n$ entry should instead be $-(n+1)$.}. Note that this recovers the entire component group since the residue field has characteristic $0$.  When computing the component group of a reduction type family, if the entries of the monodromy matrix have no parameters then computing the Smith normal form is a straightforward computation. This is the case for all families in Tables \ref{table:potentiallygood} and \ref{table:potentiallygoodx2Table2}. 

For families in Tables \ref{table:potentiallygoodxpotentiallymultiplicative} and \ref{table:potentiallymultiplicativex2}, the monodromy matrices are equivalent to matrices consisting of two $2 \times 2$ blocks, each representing the monodromy of the elliptic curves appearing in the reduction type label. Therefore, the component groups of such families are the products of the component groups of the two elliptic curves appearing in the Namikawa-Ueno label. 

The component groups of the semistable families $\text{I}_{l-0-0}$, $\text{I}_{l-m-0}$ and $\text{I}_{l-m-n}$ were computed using the homology of the dual graph and the local invariants given in \cite[Table 8]{M2D2}.

The special fibres of the MRM reduction types $\text{I}_{l-0-0}^*$, $\text{I}_{l-m-0}^*$, $\text{I}_{l-m-n}^*$, $\text{III}^*$-$\text{II}_l^*$, $\text{IV}^*$-$ \text{II}_l^* $, $\text{II}_{(l+1) -0}^*$, $\text{IV}$-$ \text{II}_{l+1}^*$, $\text{III}$-$ \text{II}_{l+1}^*$, $ \text{III}_l^*$  all contain a chain of $\mathbb{P}^1$s of length dictated by the family parameter. To reduce the amount of computations, we observe that by \cite[Theorem 4.1.14]{padma}, the length of this chain (modulo multiplicities) does not affect the component group. This eliminates parameters from the monodromy matrix and the component group is straightforward to compute from these as before. 

The computation of the relevant Smith normal forms for all remaining reduction type families are presented in the following table. 
 
\end{proof}

{
    \centering
    \begingroup
    \setlength\extrarowheight{2pt}

    \begin{longtable}{|>{\centering\arraybackslash}m{2cm}|>{\centering\arraybackslash}m{3cm}|>{\centering\arraybackslash}m{3cm}| >{\centering\arraybackslash}m{7cm}|}
    \caption*{Smith normal forms of monodromy matrices.}
    \label{tab:exceptions} \\
    \hline 
        Reduction Type  & Monodromy matrix, $M$ & Smith normal form, $N$ &  $A,B $ such that  $A(M -I_2)B = N$ \\[1em]
         \hline 

$\text{III}-\text{II}_l$ & $ \matr 0010 011{l+1} {-1}00{-1} 0001$ & $\matr 1000 0100 00{2l+1}0 0000$ & 
$\matr {-1}100 0100 {-1}210 0001$, $\matr 10{-(l+1)}0 0001 01{-(l+1)}0 0010$ \\[1em]
\hline 
$ \text{IV}-\text{II}_l$ & $
\matr 0010 01{-1}{l+1} {-1}0{-1}1 0001$ & $\matr 1000 0100 00{3l+2}0 0000$ 
& $\matr 02{-1}0 12{-1}0 13{-1}0 0001$, $ \matr 10{-(2l+1)}0 0001 01{-(2l+1)}0 0010$\\[1em]
\hline 

$\text{II}_{(l+1)-0}$ & $\matr {-1}000 {-1}10{l + 1} 00{-1}{-1} 0001$  & $ \matr 1000 0100 00{4l + 4}0 0000$& $\matr {-1}100 00{-1}0 {-1}2{2(l+1)}0 0001$, $\matr 10{-(l+1)}0 0001 01{-(l+1)}0 0010$ \\[1em] 
\hline 

$ \text{IV}^*-\text{II}_{l+1}$& $\matr {-1}0{-1}{-1} {-1}10{l + 1} 1 000 0001$ & $\matr 1000 0100 00{3l+4}0 0000$ & $\matr {-1}210 {-1}200 {-1}310 0001$, $\matr 10{-(2l + 3)}0 0001 01{-(2l + 3)}0 0010 $ \\[1em]
\hline 
$\text{III}^*-\text{II}_{l+1}$ &
$ \matr 00{-1}1 110{l+1} 1000 0001$
& 
$\matr 1000 0100 00{2l+3}0 0000$ & $ \matr 0100 01{-1}0 12{-1}0 0001$, 
$\matr 10{-(l+1)}0 0001 01{-(l+1)}0 0010 $ \\[1em] 
\hline

$\text{III}_l\text{(p182)}$ $l \text{ even}$ & $\matr 0{-1}10 10l{-1} 000{-1} 0010 $ &  $ \matr 1000 0100 0020 0002 $ & $\matr 1000 0100 {-1}{-1}10 {-1}{-1}01$, $\matr {-1}0 {\frac{l}{2} -1} {-\frac{l}{2}} 00{-\frac{l}{2}}{\frac{l}{2} +1} 00{-1}1 {-1}{-1}{-1}{-1}$ \\[1em]
\hline 
$\text{III}_l\text{(p182)}$ $l \text{ odd}$ & $\matr 0{-1}10 10l{-1} 000{-1} 0010$ & $ \matr 1000 0100 0010 0004$ & $\matr 1000 0100 0010 {-2}{-2}{-3}1$ ,  $\matr {-\frac{l+1}{2}}{-\frac{l-1}{2}}{-\frac{l+1}{2}}{-l} {\frac{l+1}{2}}{\frac{l+1}{2}}{\frac{l+3}{2}} {l+2} 1112 {-1}{-1}{-2}{-2}$ \\[1em] 
\hline 
$ \text{II}_{l-m} \text{(p182)}$ $lm \text{ even}$ & $ \matr {-1}0{-m}{-1} 011l 00{-1}0 0001  $ &  $\matr 1000 0100 00{4l}0 0000 $ & $ \matr 0100 {-1}{-m}00 {-2l}{-(2lm+2)}{-1}0 0001 $, $ \matr 0{\frac{lm}{2}}{-(lm-1)}0 0001 1{-l}{2l}0 01{-2}0 $  \\[1em]
\hline 

$\text{II}_{l-m} \text{(p182)}$ $lm \text{ odd} $ &$ \matr {-1}0{-m}{-1} 011l 00{-1}0 0001  $ & $ \matr 1000 0200 00{2l}0 0000 $ 
& $\matr 0100 {-1}{-m}00 0210 0001$,  $\matr 01{\frac{lm -1}{2}}0 0001 10{-l}0 0010$ \\[1em] 
\hline 

$\text{II}_{l-m}\text{(p183)} $ & $\matr {-1}0{-m}0 11{l}{m} 00{-1}1 0001 $ & $ \matr 1000 0100 00{4 l+m }0 0000$  & $\matr 0010 01{-l}0 12{-2 l}0 0001$,  $\matr 0 1{-m-2l}0 0001 0010 1020$  \\[1em] \hline

$\text{II}_{l-m}^*$ &
 $
\matr 10{-m}0  {-1}{-1}m{-l} 001{-1} 000{-1} 
 $
& $\matr 1000 0100 00m0 0000$ & $\matr 00{-1}0 0{-1}{l}0 {-1}000 00{-2}1$, 
$\matr 01{m}{-2} 0001 0010 1000 $\\[1em]
\hline 

$2\text{I}_l-t$ \ \ \ $l\text{ even}$ & $\matr 010{l} 1000 0001 0010 $ & $ \matr 1000 0100 00{l}0 0000$ &$\matr 1000 0010 {l -1}{-1}{l}0 0011$, $ \matr {\frac{l -2}{2}}{\frac{l}{2}}{- \frac{l}{2}}1 {-\frac{l}{2}}{- \frac{l}{2}}{\frac{l}{2}}1 10{-1}0 11{-1}0 $\\[1em]
\hline 

$2\text{I}_l-t$ \ \ \ $l\text{ odd}$ & $\matr 010{l} 1000 0001 0010 $ & $\matr 1000 0100 00{l}0 0000$ &$\matr 1000 0010 {l-1}{-1}l0 0011$ , $\matr {\frac{l -1}{2}}{\frac{l-1}{2}}{-\frac{l-1}{2}}1 {-\frac{l-1}{2}}{- \frac{l +1 }{2}}{\frac{l + 1}{2}}1 10{-1}0 11{-1}0 $ \\[1em]
\hline 

$ 2\text{I}_l^*-t$ \ \ \ $l\text{ even}$ & $\matr 0{-1}0{-l} 1000 000{-1} 0010$ & $ \matr 1000 0100 0020 0002$ & $\matr 1000 0010 1{-1}20 00{-1}1 $,  $ \matr 0{\frac{l +2}{2}}{-1}{\frac{l}{2}} {-1} {\frac{l -2}{2}}1{\frac{l}{2}} 0001 0{-1}0{-1}$ \\[1em]
\hline 

$2\text{I}_l^*-t$ \ \ \ $l\text{ odd}$ & $\matr 0{-1}0{-l} 1000 000{-1} 0010$ & $\matr 1000 0100 0010 0004 $ & $\matr  1000 0100 0010 {-2}{-2}{-3}1 $, $ \matr {\frac{l -1}{2}}{\frac{l +1 }{2}}{l} {l} {\frac{l -1}{2}}{\frac{l -1}{2}}{\ell}{\ell} 1112 {-1}{-1}{-2}{-2}$\\[1em]
\hline 
    \end{longtable}
\endgroup}

\section{Conductors, discriminants and holomorphic volume forms}\label{s:discriminant}

In this section we address the conductor, minimal discriminant and the Birch--Swinnerton-Dyer ``fudge factor'' $v(\omega_{\text{min}}/\omega^0)$ that appear in \Cref{thm:main}(3,5,7). We make use of the fact that these terms, and the number of components on the special fibre of the MRM model, are connected by a generalisation of Ogg's formula (\Cref{oggformula}). We also use a classification of Ueno \cite{ueno} of discriminants of N{\'eron} minimal Weierstrass models (see \Cref{def:neronminimal}).

\subsection{Conductor exponent}
Let $C$ be a curve over $K$. 
Recall that the \textit{conductor exponent} $n_C$ is divided into two parts, the \textit{tame} and \textit{wild} conductors (see \cite{serreconductor} for precise definitions), 
$$
n_{C} = n_{C, \text{tame}} + n_{C, \text{wild}}.
$$ 

When the curve has tame reduction, the wild part of the conductor is zero. 
The tame conductor exponent is straightforward to compute from the special fibre of a regular normal crossings model of $C$ as follows. 
\begin{lemma} \label{tameconductorformula}
 Let $C/K$ be a curve 
 and let $\mathcal{C}/R$ be a regular normal crossings model of $C$. The tame conductor exponent $n_C$ is determined  by the special fibre $\mathcal{C}_s$ of $\mathcal{C}$.  More precisely,  
 $$n_{C, \textup{tame}} = 2g -2a -t $$
 where $g$ is the genus of the curve and 
 \begin{itemize}
     \item $a$ is the \textit{abelian rank}, equal to the sum of geometric genera of components of $\mathcal{C}_s$;
     \item $t$ is the \textit{toric rank}, equal to the number of loops in the dual graph of  $\mathcal{C}_s$. 
 \end{itemize}
\end{lemma}
\begin{proof}
By definition of tame conductor exponent  $n_{C, \text{tame}} = 2g - \dim V^{I_K}$, where $I_K$ is the inertia subgroup and $V= T_\ell(\text{Jac}(C))\otimes_{\Z_\ell}\Q_\ell$ the $\ell$-adic representation associated to $C/K$. The fact that $\dim V^{I_K} = 2a +t$ and the relation between the special fibre and abelian and toric ranks are both proved in \cite[Ch. 9,10]{neronmodels}. 
\end{proof}

\begin{corollary}\label{cor:g2_conductor}
Let $C/K$ be a genus $2$ curve and assume $C$ has tame reduction. The conductor exponent of the Jacobian of $C$ is determined by the MRM reduction type of $C$ and it is stated in Tables \ref{table:potentiallygood} -\ref{table:potentiallymultiplicativex2}. 

\end{corollary}
\begin{proof}
As the curve has tame reduction, the wild conductor exponent is zero.  By \Cref{tameconductorformula} the tame conductor exponent depends only on the MRNC reduction type, and hence, by \Cref{cor:genus2MRNCtoMRM}, only on the MRM reduction type. We determine the conductor exponent associated to each MRM reduction type in Tables \ref{table:potentiallygood}-\ref{table:potentiallymultiplicativex2} by applying the explicit formula of \Cref{tameconductorformula} to the special fibre of the corresponding MRNC reduction type.
\end{proof}

\subsection{Minimal Weierstrass models}
Let $C/K$ be a curve of genus $g$. For the rest of this section we assume that the residue characteristic of $K$ is odd.

Any holomorphic differential $\omega$ on $C$ defines an invariant $1$-form on $J = \Jac(C)$, and a basis $\omega_1, \ldots, \omega_g$ of $H^0(C, \Omega_C^1)$ defines   
$$
\omega_1 \wedge \ldots \wedge \omega_g, 
$$
a global holomorphic $g$-form on $J$, since $H^0(J, \Omega_{J}^g) \cong \wedge^g_{i=0} H^0(J, \Omega^1_{J})$.

\begin{defn}\label{def:neronexteriorform}
The \textit{N{\'e}ron exterior form} of $J$ is a generator of $H^0( \mathcal{J}, \Omega^g_{\mathcal{J}})$, denoted $\omega^0$. It is well-defined up to scaling by a unit. 
\end{defn}

For the rest of this section, $C/K$ will denote a curve of genus $2$.

\begin{defn}\label{def:hol1form}
Let $y^2 =f(x)$ be a Weierstrass model of $C/K$. Define the \textit{holomorphic volume form}
$$ \omega(f) = \frac{dx}{y} \wedge \frac{xdx}{y} \in H^0(J, \Omega_{J}^2)$$
and the \textit{discriminant}  
$$\Delta(f) = \disc (f(x)) \in K  $$
where $\disc(f(x))$ is the discriminant of the polynomial $f(x) \in K[x]$.
\end{defn}

\begin{lemma}[{\cite[\S 1.3]{Liu}}] \label{relativevaluation}
Let $C/K$ be a curve of genus $2$ and $y^2 = f(x)$ a Weierstrass model for $C$.  Then  $\Delta(f) \left( \frac{dx}{y} \wedge \frac{xdx}{y} \right)^{\otimes 10}$ remains invariant under M\"obius transformations.
\end{lemma}

\begin{defn}\label{def:mindisc}
 Let $C/K$ be a curve of genus $2$. A Weierstrass model $y^2 =f(x)$ of $C$ is \textit{minimal} if $f(x) \in R[x]$ and $v_K(\Delta(f))$ is minimal amongst all such models. We write $v(\Delta_{\textit{min}})$ and $ \omega_{\text{min}}$ for the discriminant and volume form associated to a minimal Weierstrass model of $C$. 

\end{defn}

\begin{remark}
The $2$-form $ \omega_{\text{min}}$ is independent of the choice of minimal Weierstrass model (up to multiplication by a unit) by \Cref{relativevaluation}.   
\end{remark}

\begin{defn}\label{def:neronminimal}
 Let $C/K$ be a curve of genus $2$. A Weierstrass model $y^2 =f(x)$ of $C$ is called \textit{N{\'e}ron minimal} if $\omega(f)$ is the N{\'e}ron exterior form of $\Jac(C)$. (This model is not necessarily integral.)
 We write $ v(\Delta_{u})$ for $v(\Delta(f))$ where $y^2 =f(x)$  is any N{\'e}ron minimal Weierstrass model for $C$. This model is not necessarily integral. 
\end{defn}

\begin{remark}
 The definition of $ v(\Delta_{u})$ is independent of the choice of minimal N{\'e}ron model by \cite[p.56]{Liu}.
\end{remark}

\begin{remark} \label{remarkondiscsinequality}
By \cite[Prop. 2]{Liu} $v(\Delta_{u}) \le v(\Delta_{\min})$. 
\end{remark}

\subsection{The generalised Ogg's formula} 
 
For an elliptic curve $E/K$, Ogg's formula relates the conductor exponent $n_E$ of $E$ to the valuation of the  minimal discriminant via 
$$n_E =v(\Delta_{\text{min}}) + 1 -m   $$
where $m$ is the number of irreducible components of the special fibre of the  N{\'e}ron model of $E$. 
In the case of genus $2$ curves, a similar formula due to Ueno \cite{ueno}, Saito \cite{saitodiscriminants} and Liu \cite{Liu} relates the conductor exponent, the number of irreducible components in the special fibre of the minimal regular model, the minimal discriminant and~$v(\omega_{\text{min}}/\omega^0)$.

\begin{lemma} \label{oggformula2}
 Let $C/K$ be a genus $2$ curve and assume $p \ne 2$. Then 
$$ 
v(\Delta_{\min}) = v(\Delta_{u}) - 10 v(\omega_{\text{min}}/\omega^0).
$$
\end{lemma} 

\begin{proof}
By \Cref{relativevaluation}, it follows that $\Delta_{\min} \omega_{\min}^{\otimes 10} = \Delta_u \left( \omega^0 \right)^{\otimes 10}$, by comparing a minimal Weierstrass model to a N{\'e}ron minimal Weierstrass model our claim follows. 
\end{proof}

\begin{theorem}\label{oggformula}
Let $C/K$ be a genus $2$ curve and assume $p \ne 2$. Then 
$$ 
n_C = v(\Delta_{\min}) + 1 -m + 11 v(\omega_{\min}/\omega^0)
$$
where $m$ is the number of components in the special fibre of the \textit{minimal regular model} of $C$ and  $n_C$ is the conductor exponent of the Jacobian of $C$.
\end{theorem}

\begin{proof}
Combining \cite[Thm.~2 and Prop.~1]{Liu} we obtain
$$ 
n_C = v(\Delta_{\min}) + 1 -m -11( c(C) )
$$
where $c(C)$ is the constant appearing in Liu's formula.
By \cite[Thm.~2]{Liu}
$v(\Delta_{\min})\! -\! v(\Delta_u) \!=\! 10c(C)$,
and combined with \Cref{oggformula2} gives $c(C)\! =\! - v(\omega_{\min}/\omega^0)$, which completes our proof.  
\end{proof}

\subsection{Minimal discriminants} 
The discriminant $v(\Delta_{u})$ was classified for all genus $2$ curves by Ueno \cite{ueno}. In this section we determine $v(\Delta_{\min})$. 
We also describe the cluster picture of a Weierstrass model realising the minimal discriminant.


Cluster pictures encode both the discriminant and whether the underlying Weierstrass equation can be made integral by a shift of the variable. (See \S\ref{subsec:examples} for relevant definitions and notation and \Cref{fieldofdefnremark} for a discussion of the field of the definition.)  

\begin{lemma}[{\cite[Thm. 16.2]{M2D2}}] \label{clusterpicdisc}
Let $C/K$ be a hyperelliptic curve of genus $g$, $p \ne 2$ and suppose $ y^2 = f(x)$ is a  Weierstrass model of $C$.  Then
$$
 v(\Delta(f)) = v(c_f)(4g+2) + |\mathcal{R}|(|\mathcal{R}|-1)d_\cR + \displaystyle \sum_{\mathfrak{s}\neq\mathcal{R}  \ \text{proper}} |\mathfrak{s}|(|\mathfrak{s}|-1)\delta_{\mathfrak{s}},
$$
where $\delta_{\mathfrak{s}}$ is the relative depth of $\mathfrak{s}$.
\end{lemma}

\begin{theorem}[{\cite[Def.~13.2 and Thm.~13.2]{M2D2}}]\label{thm:M2D2_13}
 Let $y^2 = f(x)$ be a Weierstrass model of a  tame hyperelliptic curve $C/K$ and assume $p \ne 2$.   
 Then $f(x-z) \in R[x]$ for some $z \in K$ if and only if one of the following conditions~hold:
 \begin{itemize}
     \item $v(c_f) \ge 0$ and $d_{\mathcal{R}} \ge 0$, or 
     \item there exists a $G_K$-stable proper cluster $\mathfrak{s}$ with $d_{\mathfrak{s}} \le 0$ and 
     $$ 
     v(c_f) + (\vert \mathfrak{s} \vert - \vert \mathfrak{t} \vert) d_{\mathfrak{s}} + \sum_{r \notin \mathfrak{s}} d_{r \wedge \mathfrak{s}} \ge 0 
     $$
     for some $\mathfrak{t}$ which is either empty or a $G_K$-stable child $\mathfrak{t} < \mathfrak{s}$ with either $\vert \mathfrak{t} \vert =1 $ or $ d_{\mathfrak{t}} \ge 0$.
\end{itemize}
In particular, in this case $v(\Delta_{ \min}) \le v(\Delta(f))$. 
\end{theorem}

The above results allow us to deduce the following. 

\begin{theorem} \label{bsdfudgeanddiscMRMfns}
 Let $C/K$ be a genus $2$ curve, $p \ne 2$ and suppose $C$ has tame reduction. Then $v(\Delta_{\min})$ and $v(\omega_{\text{min}}/\omega^0)$ are determined by the MRM reduction type and the potential stable type of $C$.    
\end{theorem}

\begin{proof}
The discriminant of a Weierstrass model can be determined from the associated cluster picture  by \Cref{clusterpicdisc}. As the MRM reduction type and potential stable type determine all possible cluster pictures (\Cref{thm:CP2}(2)) and one can determine which cluster pictures are defined by integral Weierstrass models (\Cref{thm:M2D2_13}, and \Cref{rem:K-indep} to pin down the $G_K$-action), our claim for $v(\Delta_{\min})$ follows. 
For $ v(\omega_{\min}/\omega^0)$, the result follows from \Cref{oggformula} and \Cref{tameconductorformula}.   
\end{proof}

To determine $v(\Delta_{\min})$ explicitly, we use the following three criteria for minimal Weierstrass models.

\begin{theorem}[{\cite[Thm.~17.1]{M2D2}}]  \label{semistabledisc}
    Let $C/K$ be a hyperelliptic curve, $p \ne 2$ and $y^2 = f(x)$ a  Weierstrass model for $C$. Suppose that $f(x) \in R[x]$, $d_{\mathcal{R}} = 0$, $v(c_f) = 0$ and there are no clusters $\mathfrak{s} \neq \mathcal{R}$ of size $|\mathfrak{s}| > g + 1$. Then $y^2 = f(x)$ is a  minimal Weierstrass model of $C$. 
\end{theorem}

\begin{lemma} \label{minimaldisclemma}
 Let $C/K$ be a genus $2$ curve, $p \ne 2$ and assume $C$ has tame reduction. Suppose $y^2 = f(x)$ is an integral Weierstrass equation for $C$ such that 
 $$ 
 v(\Delta(f)) < v(\Delta_{u}) + 110 \quad \text{ and }\quad  v(\Delta(f)) - m - n_C + 1 \equiv 0 \mod 11
 $$
 where $m$ is the number of components in the special fibre of a minimal regular model of $C$ and $n_C$ is the conductor exponent of the Jacobian. Then $ v(\Delta(f)) = v(\Delta_{min})$.   
\end{lemma}
\begin{proof}
  Let $y^2 = g(x)$ be a minimal Weierstrass model of $C$. Then $v(\Delta(g)) = v(\Delta_{\min})$ by definition and $v(\Delta(f)) = v(\Delta(g)) + 10m$ for some integer $n \ge 10$, since changes of model modify the valuation of the discriminant by a factor of $10$. By the hypothesis and \Cref{oggformula}, $110 \vert v(\Delta(f)) - v(\Delta(g))$. As $ v(\Delta_{u}) \le v(\Delta(g))$ the claim follows. 
\end{proof}

\begin{lemma} \label{2Kdisc}
Let $C/K$ be a genus $2$, $p \ne 2$
and suppose $C$ has tame reduction. If the MRM reduction type of $C$ is 
one of 
\textup{2I}${}_0$-$t$,
\textup{2I}${}^*_0$-$t$,
\textup{2I}${}_\ell$-$t$,
\textup{2I}${}^*_\ell$-$t$,
\textup{2IV}${}$-$t$,
\textup{2IV}${}^*$-$t$,
\textup{2III}${}$-$t$,
\textup{2III}${}^*$-$t$,
\textup{2II}${}$-$t$ or
\textup{2II}${}^*$-$t$,
then $v(\Delta_{\text{min}})$ is as claimed in Table \ref{table:potentiallygoodx2Table2} or \ref{table:potentiallymultiplicativex2}.
\end{lemma}

\begin{proof}
If $C / K$ has one of these MRM reduction types, then the cluster picture of any Weierstrass model for $C$ is given by 
$$\LBpmpmBb$$
where $k \in \Z$ and $l >0$ only for types $\textup{2I}_\ell$-$t$ and
$\textup{2I}{}^*_\ell$-$t$ (as visible on the third page of Table~\ref{table:potentiallygoodx2Table2} or Table~\ref{table:potentiallymultiplicativex2} and justified by Theorem \ref{thm:CP3}). Suppose $C \colon y^2  = f(x)$ has this cluster picture and leading coefficient $c_f$. Note that $G_K$ must swap the two clusters of size 3. By Theorem \ref{thm:M2D2_13}, 
$f(x)$ is integral if and only if either $v(c_f) \ge 0$ and $k \leq 0$, or $k \geq 1$ and $v(c_f) \geq 6k - 3$. Thus an integral equation has 
$$
v(\Delta(f)) = 10 v(c_f) + l +  12(t + \delta) + 30(-k + \frac{1}{2}) \geq  10\cdot\max(0, 6k-3) + l + 12(t + \delta) + 30(-k + \frac{1}{2}).
$$
This attains its minimum at $k = 0$ and $v(c_f) = 0$, giving $v(\Delta_{\min}) = 12(t + \delta) + l + 15$.
\end{proof}

We are now in a position to determine minimal Weierstrass models and discriminants in all cases.

\begin{theorem}\label{thm:g2_discriminant}
Let $C/K$ be a genus $2$ curve with tame reduction and assume $p \ne 2$. Then $v(\Delta_{\min})$ is given in Tables \ref{table:potentiallygood}-\ref{table:potentiallymultiplicativex2}. Moreover, entries marked with $\star$ (or $\diamond$ or $\diamond \diamond$) indicate a cluster picture realised by a minimal Weierstrass model of $C$; for these, the depth of the maximal cluster is taken to be the minimal possible value in $[0, 1)$ allowed by the parameters, and $v(c_f)$ the minimum of 0 or 1 as allowed by the indicated congruence.
\end{theorem}

\begin{proof}

The valuation of the discriminant $\Delta(f)$ is computed from the cluster picture indicated by $\star, \diamond$ or $\diamond\diamond$ using  \Cref{clusterpicdisc}. (Its value is displaying in the $v(\Delta_{\min})$ row in the tables.)

The entries in Tables \ref{table:potentiallygoodx2Table2}, \ref{table:potentiallygoodxpotentiallymultiplicative} and \ref{table:potentiallymultiplicativex2} that are indicated with a $\star$ are all seen to correspond to minimal Weierstrass equations by either \Cref{semistabledisc} or \Cref{2Kdisc}; in particular $\Delta(f)=\Delta_{\min}$.

For all other reduction types, 
the indicated cluster picture is realised by an integral Weierstrass model by \Cref{thm:M2D2_13}, as it has $v_c\ge 0$ and $d_\cR\ge 0$. Here $\Delta(f)$ and $\Delta_u$ satisfy the hypotheses of \Cref{minimaldisclemma}; the value of $\Delta_u$ is taken from the table in \cite[Section 5]{ueno}, where it is denoted by $\delta_x$. Hence these are also minimal Weierstrass equations, and $\Delta(f)=\Delta_{\min}$.

\end{proof}

\begin{corollary} \label{bsdfactorcalc} 
Let $C/K$ be a genus $2$ curve and $p \ne 2$. Then $v(\omega_{\text{min}}/\omega^0)$ is determined by the MRM reduction type and the potential stable type of $C$, and is explicitly given 
in Tables \ref{table:potentiallygood}-\ref{table:potentiallymultiplicativex2}.   
\end{corollary}
\begin{proof} 
The first claim follows from \cite{Liu} Remarque 6 and the fact that $c(C) = - v(\omega_{\min}/\omega_0)$, as in the proof of \Cref{oggformula}. In particular, to determine the term explicitly, we may assume that the reduction is tame. The classification now follows from \Cref{oggformula} and the previously computed invariants. 
\end{proof}

\section{Proof of \Cref{thm:main}} \label{mainthmproofsection}
We now complete the proof of our main result, Theorem \ref{thm:main}, on how the reduction type of the minimal regular model of a genus $2$ curve determines the data in Tables \ref{table:potentiallygood} to \ref{table:potentiallymultiplicativex2}. 

\begin{proof}[Proof of Theorem \ref{thm:main}]
The completeness of our classification follows from that of Namikawa and Ueno \cite{namikawaueno}, noting the changes in labels and parameters highlighted throughout the paper, see \Cref{rmk:MRM-NU}.

\begin{itemize}
    \item[(1)] In Corollary \ref{cor:genus2MRNCtoMRM} we proved the $1:1$ correspondence between MRM and MRNC reduction types; 
    \item[(2)] In Theorem \ref{thm:g2_comp_gp} we computed the N{\'e}ron component group for each reduction type family.
    \item[(3)] In  Corollary \ref{cor:g2_conductor} we determined the tame conductor exponent from the 
    MRM reduction type.
    \item[(4)] This follows from explicit verification by applying Saito's criterion (\cite[Thm~3.11]{saito}, \cite[Thm.~10.4.47]{Liu_book}) to each MRNC reduction type. 
    \item[(6)] Theorem \ref{thm:CP1}(2).
  \item[(5), (7)] In \Cref{bsdfudgeanddiscMRMfns} we proved that both $v(\omega_{\min}/ \omega^0)$ and $ v(\Delta_{\min})$ are invariants of the MRM reduction type, and computed them in  \Cref{thm:g2_discriminant} and \Cref{bsdfactorcalc}.
    \item[(8)] \Cref{thm:CP2} and \Cref{thm:CP3}. That we have exhausted all possible cluster pictures for tame genus $2$ curves with degree $6$ models is \Cref{thm:CP3}(1).
\end{itemize}
 
\end{proof}

\addtocontents{toc}{\protect\setcounter{tocdepth}{2}}

\section{Classification: Reduction types and associated invariants}\label{s:tables}

In this section and in Tables \ref{table:potentiallygood}--\ref{table:potentiallymultiplicativex2}, $C$ is a genus $2$ curve over a complete discrete valuation field $K$ with an algebraically closed residue field $k$. The tables enumerate the possible MRM reduction types of $C$, and list the associated invariants of $C$ in terms of this type. Tables \ref{table:potentiallygood}--\ref{table:potentiallymultiplicativex2} display the following data:

\noindent \underline{MRM Reduction type}: the Namikawa--Ueno \cite{namikawaueno} label for the special fibre of the minimal regular model. 
\begin{itemize}
  \item Labels II${}_{t-m}^*$ and 2I${}_m^*$-$t$ have the same MRM reduction type (including $m\!=\!0$). The distinction is used when $p\neq 2$ to indicate different potential stable types. (We claim no distinction between them for $p\!=\!2$.)
  \item \cite{namikawaueno} uses the label II$_{n-p}$ for two different reduction types; we refer to the page number in \cite{namikawaueno} to distinguish them. Similarly for III$_n$. 
  \item We follow Liu's convention \cite{Liu-errata} for type $2\text{I}_0$-$t$, so the number of components is $t\!+\!3$.
  \item We set $t\!=\!-1$ for the cases labelled as $-\alpha$ in \cite{namikawaueno}, e.g. we use I$_0^*$-III$^*$-$(-1)$ for their I$_0^*$-III$^*$-$\alpha$.

  \item We artificially duplicate reduction type I${}^*_l$-I${}_m$-$t$ as I$_l$-I${}^*_m$-$t$ in Table \ref{table:potentialtwonodes}, and similarly in Table \ref{table:potentiallymultiplicativex2}. This is for convenience of determining the MRM reduction type from the cluster picture.
  \item We draw the reader's attention to the errata to the Namikawa--Ueno data listed in \cite{Liu-errata}. 
\end{itemize}

\noindent \underline{Stable type (table title)}: The table name specifies the potential stable type for curve with the MRM types appearing in the table, provided $p\neq 2$ and reduction is tame. 
We use the notation (I)-(VII) from \cite{liu-stabletypes}; see also Table \ref{ss:changeofmodel}.

\vspace{0.5em}

\noindent \underline{MRNC Reduction type}: Label from \cite[Table G2]{timreduction} for the MRNC reduction type; it matches the labelling convention used in T. Dokchitser's online reduction type library \cite{timredlib}. 
(For types in Tables \ref{table:potentiallygoodx2Table2}, \ref{table:potentiallygoodxpotentiallymultiplicative} we only list the MRM label, as the two are virtually identical, e.g. $\text{I}_{\text{g}1}\!\underset{\raisebox{2pt}{$\scriptstyle \smash{t}$}}{-}\!\text{III}$ and I${}_0$-III-$t$.)

\vspace{0.5em}

\noindent \underline{$\#$ Components (MRM)}: the number of components in the special fibre of the minimal regular model. 

\vspace{0.5em}

\noindent \underline{Néron component group}: the component group of the Néron model of Jac$(C)$ over $K$. The shorthand notation $X_l, Y_l$ and $Z_{l,m}$ is defined at the top of the relevant page.

\vspace{0.5em}

\noindent \underline{Tame conductor exponent}: the tame part of the conductor exponent of the Jacobian of $C$.

\vspace{0.5em}

\noindent \underline{Wild reduction iff $p=$} : the residue characteristics for which the reduction type is wild. 

\vspace{0.5em}

\noindent \underline{$v(\Delta_{\min})$}:  valuation of the discriminant of a minimal Weierstrass model (\Cref{def:mindisc}). For each reduction type family, we indicate with $\star, \diamond$ or $\diamond\diamond$ a cluster picture whose underlying Weierstrass model is minimal. The minimal Weierstrass model has $v_c = 0$ or $1$ depending on indicated congruence (0 if either is allowed), and the depth of the maximal cluster the (minimal) value in $[0,1)$ allowed by the parameter values.

\vspace{0.5em}

\noindent \underline{$-v(\omega_{\min}/\omega^0)$}: $\omega_{\text{min}} = \frac{dx}{y} \wedge x\frac{dx}{y}$, for a minimal Weierstrass equation, and $\omega^0$ the Néron exterior form; here $p\!\neq\! 2$. (The quantity $v(\omega_{\min}/\omega^0)$ is negative by \Cref{remarkondiscsinequality} and \Cref{oggformula2}.)
The formula is usually also valid for $p\!=\!2$ with suitable definitions; see \cite{Liu} Remarque 6.

\vspace{0.5em}

\noindent \underline{Cluster pictures}:
These are drawn as explained in \S\ref{subsec:examples}. In our convention the cluster picture data includes $v_c$, the valuation of the leading coefficient of $f(x)$, where $C:y^2\!=\!f(x)$ is the underlying equation.
\begin{itemize}
    \item The depth of the cluster comprising all the roots of $f(x)$ is displayed modulo 1.
    \item Parameters $a,b,j,s\in\Z$ are arbitrary, constrained only by all relative depths being $>0$.
    \item All congruences are taken modulo $2$, e.g. $v_c\equiv k$ denotes $v_c\equiv k \mod{2}$. 
    \item The cluster pictures for quintic polynomials are not displayed separately: these are obtained by taking the cluster pictures that have a cluster of size 5 and removing the root outside that cluster, retaining the depths and $v_c$ (see Definition \ref{def:equiv-cpics}(4) and Proposition \ref{prop:cpic-ABC}(4)).
    \item The value in the cell in the row of a given cluster picture and column of an MRM reduction type is the value of the boxed parameter displayed in the cell to the right of the cluster picture; typically, this is the value of a relative depth that makes a curve with the given cluster picture attain the given MRM type. An unordered tuple is separated by a ``,'', an ordered tuple by a ``;''.
    \item The row for a given cluster picture is sometimes split into two sub-rows, according to the value of the underlined parameter given in the cell to the right of the cluster picture.
\end{itemize}

Table~\ref{ss:changeofmodel} is a guide to the cluster pictures which can occur for each potential stable type when $p\ne 2$ and also explains how the cluster pictures displayed in Tables~\ref{table:potentiallygood}--~\ref{table:potentiallymultiplicativex2} are related by changes of Weierstrass equation. It records only the underlying cluster shapes of size $6$, forgetting the depths and valuation of leading coefficient. When the cluster pictures arise from curves with tame reduction, the arrows and loops, labelled A, B, C, and D, record possible changes induced by the changes of Weierstrass equation in Proposition~\ref{prop:cpic-ABC}(5). Loops indicate changes which preserve the underlying cluster shape but alter the omitted data. 

If several labels appear on one arrow or loop, for example ``$A/C$'', this means that the same change of underlying cluster shape can arise from either of the indicated cases of Proposition~\ref{prop:cpic-ABC}(5), depending on the depth data. For a given cluster picture (with depths and $v_c$), the hypotheses of Proposition~\ref{prop:cpic-ABC} determine which of these changes can actually be performed. Moreover, the table also ignores integral shifts of all depths, and so the corresponding transformations in the tables below are understood up to composing with the shift in Proposition~\ref{prop:cpic-ABC}(2).

To infer the MRM type from the cluster picture requires the curve to be tame. One can use:

\begin{lemma}[Tame reduction criterion]
\label{lem:cluster_pic_tame}
Suppose $p\neq 2$. Let $C:y^2=f(x)$, with $d=\deg(f) =5\textup{ or }6$.
\begin{enumerate}
\item $C/K$ is tame if and only if the splitting field of $f(x)$ is tamely ramified.
\item In particular, if no element of order $p$ in $S_d$ preserves the cluster picture of $C$, then $C/K$ is tame.
\end{enumerate}
\end{lemma}

\begin{proof}
(1)  \cite[Thm.~10.3(4)]{M2D2}. (2) Clear. 
\end{proof}

\newpage

\thispagestyle{empty}

\textwidth 25.5cm
\oddsidemargin -2.3cm
\evensidemargin -2.3cm
\textheight 27cm
\topmargin -2.2cm

\newcolumntype{?}{!{\vrule width 1.5pt}}
\newcommand{\htline}{\noalign{\hrule height 1.5pt}}

\robustify{\potgoodpicture}

\begin{landscape}

\subsection{Table 1 (Stable Type I\phantom{VI} if tame and $p \neq 2$ 
\smash{\potgoodpicture})}

\label{table:potentiallygood}
\vphantom{blah}

\noindent {\small{Entries with a $\star$ indicate a minimal Weierstrass model (setting $s=0$ where relevant).}} 
$$
\hspace{-1cm}
\begin{array}{?c|cc?c|c|c|c|c|c|c|c|c|c|c|c|c|c|c|c|c|c?} 
\htline      
      \multicolumn{3}{?c?}{\text{MRM Reduction type}} %
  & \text{I}_{0-0-0} & \text{I}^*_{0-0-0} &\text{II} & \text{III}  & \text{IV}  &  \text{V} &  \text{V}^{*} & \text{VI}  &  \text{VII} & \text{VII}^{*}  & \text{VIII-}1  & \text{VIII-}2 & \text{VIII-}3  &  \text{VIII-}4 &  \text{IX-}1 &  \text{IX-}2 &  \text{IX-}3 &  \text{IX-}4   \\ \htline  
\multicolumn{3}{?c?}{\text{MRNC Reduction type}} & \textup{I}_{\textup{g}2} & 2^{1,1,1,1,1,1} & \textup{D}_{\textup{g}1} & 3^{1,1,2,2} & 6^{2,4,3,3} & 6^{1,1,4}& 6^{5,5,2}& 4^{1,3,2,2}& 8^{1,3,4}&8^{5,7,4}& 10^{1,4,5} & 10^{7,8,5}& 10^{3,2,5}& 10^{9,6,5}& 5^{1,2,2} & 5^{1,1,3}& 5^{2,4,4}& 5^{3,3,4}\\ 
\hline 
\multicolumn{3}{?c?}{{\text{\# Components (MRM) }} }
& 1 &7&3&7&6& 2 & 12  & 7  & 2 & 12 & 1 & 9 & 4 & 13 & 5 & 3 & 11 & 9 \\  
\hline
\multicolumn{3}{?c?}{\text{N\'eron component group}} 
    & (0) & \bigl(\frac{\Z}{2\Z}\bigr)^4 &  (0) 
    & \frac{\Z}{3\Z}\!\times\!\frac{\Z}{3\Z}
    & (0) & \frac{\Z}{3\Z} & \frac{\Z}{3\Z} 
    & \frac{\Z}{2\Z}\!\times\!\frac{\Z}{2\Z}
    & \frac{\Z}{2\Z} & \frac{\Z}{2\Z} & (0) & (0) & (0) & (0) & \frac{\Z}{5\Z} & \frac{\Z}{5\Z} & \frac{\Z}{5\Z} & \frac{\Z}{5\Z} \\ 
    \hline 
\multicolumn{3}{?c?}{\text{Tame conductor exponent}} &0&4&2&4&4&4&4&4&4&4&4&4&4&4&4&4&4&4 \\ \hline 
\multicolumn{3}{?c?}{\text{Wild reduction iff $p=$}} &-&2&2&3&2,3&2,3&2,3&2&2&2&2,5&2,5&2,5&2,5&5&5&5&5 \\ \hline
\multicolumn{3}{?c?}{-v(\omega_{\text{min}} / \omega^0) \text{ \hspace{1em} if $p \not=2$}} 
& 0 & 0 &1 & 0 & 1 & 0 & 0 & 0 & 0 & 0 & 0 & 0 & 1 & 0 & 0 & 0 & 0 & 0  \\ 

\htline 
  \multicolumn{21}{?c?}{\textup{Entries below this line only valid for curves with tame reduction and }p\neq 2} 
  
  \\  
\htline 
\multicolumn{3}{?c?}{v(\Delta_{\text{min}})}&0&10&15&10&20&5&15&10&5&15&4&12&18&16&8&6&14&12 \\
\htline
 \multicolumn{21}{?c?}{\textup{All possible cluster pictures for reduction type \rlap{\qquad\qquad(displayed value is $\delta$)}}} \\ \htline
\smash{\raisebox{-6pt}{\LBpgA}} & \smash{\raisebox{-10pt}{\boxed{\delta}}} & \underline{v_c \equiv 0} & 0^{\sstar} &  & \frac{1}{2}^{\sstar}  & \frac{1}{3}^{\sstar} \ \text{or} \ \frac{2}{3} &   & \frac{1}{6}^{\sstar} & \frac{5}{6} & & & &  & &  &  & \frac{3}{5} & \frac{1}{5}^{\sstar}& \frac{4}{5}  & \frac{2}{5}^{\sstar} \\
          & & \underline{v_c \equiv 1}  & & 0^{\sstar}  &  {\frac{1}{2}} &  & {\frac{1}{3}^{\sstar}  \textup{ or }   \frac{2}{3}} &{\frac{5}{6}}  & \frac{1}{6}^{\sstar} &&&& \frac{4}{5}& \frac{2}{5}&\frac{3}{5}& \frac{1}{5}^{\sstar}  &&&& \\ \hline    
\smash{\raisebox{-10pt}{\LBpgB}} &  \smash{\raisebox{-10pt}{\boxed{\delta}}} & \underline{v_c \equiv s} & 0 &  & & & & & &\frac{1}{2}^{\sstar} & \frac{1}{4}^{\sstar} \textup{ or } \frac{3}{4}& & \frac{1}{5}^{\sstar} & \frac{3}{5}^{\sstar} &  & & \frac{2}{5}^{\sstar} & \frac{4}{5}& & \\
& & \underline{v_c \not\equiv s } &  & 0 & & & & & &\frac{1}{2}  & & \frac{1}{4}^{\sstar} \textup{ or } \frac{3}{4}&  &  &  \frac{2}{5}^{\sstar} & \frac{4}{5} &  & & \frac{1}{5}^{\sstar} & \frac{3}{5} \\ 
\htline 
\end{array}
$$
\end{landscape}

\robustify{\potonenodepicture}

\newpage 
\thispagestyle{empty}

\textwidth 25.5cm
\oddsidemargin -2.3cm
\evensidemargin -2.3cm
\textheight 27cm
\topmargin -2.5cm

\begin{landscape}

\subsection{Table 2 (Stable Type II\phantom{V} if tame and $p\neq 2$ \potonenodepicture)}

\label{table:potentialonenode}

\vphantom{blah}

\noindent 
{\small{$X_l$ denotes $\frac{\Z}{2\Z}\!\times\! \frac{\Z}{2\Z}$ if $l$ is even, and $\frac{\Z}{4\Z}$ if $l$ is odd.}}\\
\noindent 
{\small{Entries with $\dagger$: if $l\!=\!-1$, decrease $\#$Components by 1, increase $v(\Delta_{\min})$ by 10, set $-v(\omega_{\text{min}}/\omega^0)\!=\!1$.}} \\
\noindent
{\small{Entries with $\star$ indicate a minimal Weierstrass model (setting $j=0$ and $s=1$ where relevant), superceded by $\diamond$ if $l\!=\!-1$ and $\diamond\diamond$ if $l\!=\!0$.}} 
$$
\hspace{-1cm}
\begin{array}{?c|cc?c|c|c|c|c|c|c|c|c|c|c|c?} \htline  
\multicolumn{3}{?c?}{\text{MRM Reduction type}} & \text{I}_{l-0-0} & \text{I}_{l-0-0}^* &\text{III-II}_l & \text{III}^*\text{-II}_l^*  & \text{IV-II}_l  &  \text{II}^*\text{-II}_l^* &  \text{II}_{(l+1)-0} & \text{II}_{(l+1)-0}^*  &  \text{IV}^*\text{-II}_{l+1} & \text{II-II}_{l+1}^*  & \text{III}^*\text{-II}_{l+1}  & \text{III-II}_{l+1}^*    
\\[-0.5em] 
\multicolumn{3}{?c?}{\textup{(family parameter)}} & l > 0 & l > 0 & l \ge 0& l \ge 0& l \ge 0& l \ge 0& l \ge 0& l \ge 0& l \ge -1 & l \ge -1 & l \ge -1& l \ge -1 
\\  \htline 
\multicolumn{3}{?c?}{\text{MRNC Reduction type}} & \text{I}_{l,\text{g}1} & \text{I}^{*}_{0, l\D} & \text{III}_{l} & \text{III}^{*}_{l\D} & \text{IV}_l & \text{II}^{*}_{l\D} & \text{I}_{0,l}^{*} & [2]\text{I}_{\text{g}1, (l+1)\D} &\text{IV}^{*}_{l} & \text{II}_{(l+1)\D} & \text{III}^{*}_{l } &  \text{III}_{(l+1)\D} \\
\hline
\multicolumn{3}{?c?}{{\text{\# Components (MRM) }}}
         & l & l + 7 & l+1 & l + 10 & l + 2 & l + 11 & l + 4 & l+4 & l+6^{\dagger} & l+5 & l+7^{\dagger} & l+6\\ \hline
\multicolumn{3}{?c?}{\text{N\'eron component group}}
         & \frac{\Z}{l\Z} & \frac{\Z}{2\Z} \times \frac{\Z}{2\Z} \times  X_l &  \frac{\Z}{(2l+1)\Z} &\frac{\Z}{8\Z} &\frac{\Z}{(3l+2) \Z}& X_l & \frac{\Z}{(4l+4)\Z} & (0) & \frac{\Z}{(3l+4) \Z} & X_l & \frac{\Z}{(2l+3) \Z} & \frac{\Z}{8 \Z} \\ \hline
         \multicolumn{3}{?c?}{\text{Tame conductor exponent}} &1&4&3&4&3&4&3&2&3&4&3&4 \\ \hline 
     \multicolumn{3}{?c?}{\text{Wild reduction iff} \ p=} & - & 2 & 2 & 2 & 3 & 2,3 & 2,3 & 2 & 3 & 2,3 & 2 & 2 \\ \hline
\multicolumn{3}{?c?}{-v(\omega_{\text{min}} / \omega^0) \text{ \hspace{1em} if $p \not=2$}} & 0 & 0 & 0 & 0 & 0 & 0 & 0 & 1 & 0^{\dagger} & 0 & 0^{\dagger} & 0 \\    
         \htline 
\multicolumn{15}{?c?}{\textup{Entries below this line only valid for curves with tame reduction and }p\neq 2}  \\  \htline

\multicolumn{3}{?c?}{v(\Delta_{\text{min}})}& l & l+10 & l+3& l+13 & l+4 & l+14 & l+6 & l+16 & l+8^{\dagger} & l+8 & l+9^{\dagger}& l+9 \\   
\htline 
\multicolumn{15}{?c?}{\textup{All possible cluster pictures for reduction type \rlap{\qquad\qquad(displayed value is $\delta$)}}}  \\  \htline
\smash{\raisebox{-10pt}{\LBonA}} &
\smash{\raisebox{-10pt}{\boxed{\delta}}} & \underline{v_c\equiv 0} & 0^{\sstar} & & & \frac{1}{4} & & & \frac{1}{2}&& &  & &\frac{3}{4}^{\sstar} \\
& & \underline{v_c\equiv 1} &  &0^{\sstar}  & \frac{1}{4} &  &&&& \frac{1}{2}& & & \frac{3}{4}^{\mathrlap{\diamond}} & \\ \hline  
\smash{\raisebox{-10pt}{\LBonB}} &
\smash{\raisebox{-10pt}{\boxed{\delta}}} & \underline{v_c\equiv l} & 0 & & \frac{1}{4} & &  \frac{1}{3}& & \frac{1}{2}&  & \frac{2}{3}^{\mathrlap{\diamond\diamond}} & & \frac{3}{4}^{\mathrlap{\diamond\diamond}} &   \\
& & \underline{v_c\not\equiv l} & & 0 & & \frac{1}{4}&  &\frac{1}{3} & & \frac{1}{2} & & \frac{2}{3} & &  \frac{3}{4}  \\ \hline
\smash{\raisebox{-10pt}{\LBonC}} &
\smash{\raisebox{-10pt}{\boxed{\delta}}} & \underline{v_c\equiv 0}  & 0 & & \frac{1}{4}^{\sstar} & &  \frac{1}{3}^{\sstar}& & \frac{1}{2}^{\sstar}& & \frac{2}{3}^{\sstar} & & \frac{3}{4}^\sstar &  \\
 & & \underline{v_c\equiv 1} & & 0 & &\frac{1}{4}^{\sstar} & & \frac{1}{3}^{\sstar}& & \frac{1}{2}^{\sstar} & &\frac{2}{3} & &\frac{3}{4} \\ \hline
\smash{\raisebox{-10pt}{\LBonD}} & \smash{\raisebox{-10pt}{\boxed{\delta}}} & \underline{v_c\equiv s}  & 0 & & & & & \frac{1}{3} & & & \frac{2}{3}^{\mathrlap{\diamond}} & & & \\
& & \underline{v_c\not\equiv s} & & 0 & & & \frac{1}{3} & & & & & \frac{2}{3}^{\sstar} & & \\ \hline   
\smash{\raisebox{-3pt}{\LBonE}} & \smash{\raisebox{-10pt}{\boxed{\delta}}} &
\underline{v_c \equiv s }& 0 & & \frac{1}{4} & &  \frac{1}{3}& & \frac{1}{2}& & \frac{2}{3} & & \frac{3}{4} &  \\
l \textup{ even} & & \underline{v_c \not\equiv s } & & 0 & & \frac{1}{4}& & \frac{1}{3} & & \frac{1}{2}& & \frac{2}{3} & & \frac{3}{4} \\ \htline 
\end{array}
$$
\end{landscape}

\robustify{\pottwonodespicture}

\newpage
\thispagestyle{empty}

\textwidth 25.5cm
\oddsidemargin -2.3cm
\evensidemargin -2.3cm
\textheight 27cm
\topmargin -2.5cm

\begin{landscape}
\subsection{Table 3 (Stable Type III\phantom{I} if tame and  $p\neq 2$ \pottwonodespicture)}
\label{table:potentialtwonodes}
\hfill 
{\small{$Z_{l,m}$ denotes $\frac{\Z}{4m\Z}$ if $l m$ is even and $\frac{\Z}{2 \Z} \!\times\! \frac{\Z}{2m \Z}$ if $lm$ is odd.}}

\noindent{\small{Entries with a $\star$ indicate a minimal Weierstrass model (setting $j=0$ and $s=0$ where relevant). \hfill $X_l$ denotes $\frac{\Z}{2\Z}\!\times\! \frac{\Z}{2\Z}$ if $l$ is even, and $\frac{\Z}{4\Z}$ if $l$ is odd.}}
$$
\hspace{-1cm}
\begin{array}{? c | cc  ? c| c| c| c| c  c?}\htline
\multicolumn{3}{?c?}{\text{MRM Reduction type }} & \mathrm{I}_{l-m-0} & \mathrm{I}^*_{l-m-0} & 2\mathrm{I}_l\text{-}0 & \mathrm{III}_{l} (p182) & \phantom{ha}\mathrm{II}_{m-l}(p182)\phantom{ha} & \phantom{ha}\mathrm{II}_{l-m}(p182)\phantom{ha} 
\\[-0.8em] 
\multicolumn{3}{?c?}{\textup{(family parameters)}} & l,m>0& l,m>0 & l>0 & l>0 & \multicolumn{2}{c?}{l,m>0} 
\\[-0.2em]  \htline
\multicolumn{3}{?c?}{\text{MRNC Reduction type}} & \mathrm{I}_{l,m}& \mathrm{I}_{l, m\D}^* = \mathrm{I}_{m, l\D}^*  & \D_{\{2-2 \}l} & 4_{l\D}^{1,3} &  \mathrm{I}_{l, m-1}^* & \mathrm{I}_{m,l-1}^* \\ \hline
\multicolumn{3}{?c?}{\# \text{ Components (MRM)}} & l + m -1  & l + m + 7   & l + 2 &  l + 7 & \multicolumn{2}{c?}{l + m + 3} \\[-0.2em] \hline
\multicolumn{3}{?c?}{\text{N\'eron component group}} & \frac{\mathbb{Z}}{l \mathbb{Z}} \times \frac{\mathbb{Z}}{m \mathbb{Z}} & X_l \times X_m & \frac{\mathbb{Z}}{l \mathbb{Z}} & X_l & Z_{l,m}& Z_{m,l}\\ \hline 
\multicolumn{3}{?c?}{\text{Tame conductor exponent}}  &2  & 4 &3 &4 & \multicolumn{2}{c?}{3} \\[-0.2em] \hline  
\multicolumn{3}{?c?}{\text{Wild reduction iff} \ p =}  &-&2&2&2& \multicolumn{2}{c?}{2} \\[-0.1em]
 \hline
\multicolumn{3}{?c?}{-v(\omega_{\text{min}} /   \omega^0 ) \text{ \hspace{1em} if $p \not=2$}} & 0  & 0 & 1 & 0  & \multicolumn{2}{c?}{0} \\[-0.2em] 
\htline
\multicolumn{9}{?c?}{\textup{Entries below this line only valid for curves with tame reduction and }p\neq 2}  \\[-0.1em]
\htline 
\multicolumn{3}{?c?}{v(\Delta_{\text{min}})}  & l + m  & l + m + 10 & l + 15 & l + 10 & \multicolumn{2}{c?}{l + m + 5}   \\ \htline 
\multicolumn{9}{?c?}{\textup{All possible cluster pictures for reduction type\rlap{\quad\qquad(displayed value is $v_c$ mod 2)}}}  \\ [-0.1em]
\htline 
\smash{\raisebox{-1pt}{\LBtwonAa}} & \boxed{v_c \text{ mod } 2}& & 0^{\sstar} & 1^{\sstar} & & & &  \\[0.5em] \hline 
\LBtwonAb & \boxed{v_c \text{ mod } 2}& & & & 0^{\sstar}\>\text{ or } 1 & & & \\ \hline 
\smash{\raisebox{-5pt}{\LBtwonBa}} & \smash{\raisebox{-8pt}{\boxed{v_c \text{ mod } 2}}}& \underline{m \equiv 0} & 0 & 1  & & & &  \\[-1em] 
& & \underline{m \equiv 1}& & & & & 0 & 1\\ \hline
\smash{\raisebox{-5pt}{\LBtwonBb}} & \smash{\raisebox{-8pt}{\boxed{v_c \text{ mod } 2}}}& \underline{m \equiv 0}&  &  & & & 1 & 0 \\[-1em] 
& & \underline{m \equiv 1}&  1 & 0& & & & \\ \hline
\smash{\raisebox{-1pt}{\LBtwonCa}} & \boxed{v_c \text{ mod } 2}& & 0 & 1 & & & & \\[0.8em] \hline
\smash{\raisebox{-1pt}{\LBtwonCb}} & \boxed{v_c \text{ mod } 2}& & & & & & 0^{\sstar} & 1 \\[0.8em] \hline 
\smash{\raisebox{-1pt}{\LBtwonDa}} & \boxed{v_c \text{ mod } 2}& & s & s+1  & & & & \\[0.8em] \hline 
\smash{\raisebox{-1pt}{\LBtwonDb}} & \boxed{v_c \text{ mod } 2}& & &  & & 0^{\sstar}\>\textup{ or }1 & & \\[0.8em] \hline 
\smash{\raisebox{-5pt}{\LBtwonE}} & \smash{\raisebox{-8pt}{\boxed{v_c \text{ mod } 2}}} & \underline{m \equiv 0}& s & s+1 & & & &  \\[-1em] 
& & \underline{m \equiv 1}& & & & & s & s+1 \\ \htline
\end{array}
$$
\end{landscape}


\robustify{\potthreenodespicture}

\newpage
\thispagestyle{empty}

\textwidth 25.5cm
\oddsidemargin -2.3cm
\evensidemargin -2.3cm
\textheight 27cm
\topmargin -2.5cm

\begin{landscape}
\subsection{Table 4 (Stable Type IV\phantom{I} if tame and $p\neq 2$ \smash{\raisebox{-3pt}{\potthreenodespicture}})}
\label{table:potentialthreenodes}

In this table the family parameters are $l,m,n>0$. \hfill {\small{$X_l$ denotes $\frac{\Z}{2\Z}\!\times\! \frac{\Z}{2\Z}$ if $l$ is even, and $\frac{\Z}{4\Z}$ if $l$ is odd.}}

\noindent{{\small{Entries with a $\star$ indicate a minimal Weierstrass model (setting $j=0$ where relevant). \hfill $Y_l$ denotes $\frac{\Z}{3\Z}\times \frac{\Z}{3\Z}$ if $3|l$, and $\frac{\Z}{9\Z}$ if $3\nmid l$.}}}
$$
\hspace{-1cm}
\begin{array}{?c | c ? c| c| c| c| c| c?}
\htline 
 \multicolumn{2}{?c?}{\text{MRM Reduction type }} & \text{I}_{l-m-n} & \text{I}^{*}_{l-m-n} & \text{II}_{l-m} (p183) & \text{II}^{*}_{l-m} & \text{III}_l (p184) & \text{III}_{l}^{*}  \\
     \htline 
\multicolumn{2}{?c?}{ \text{MRNC Reduction type }} & 
    \text{I} \underset{\raisebox{2pt}{$\scriptstyle \smash{l}$}}{-}\underset{\raisebox{2pt}{$\scriptstyle \smash{m}$}}{-}\underset{\raisebox{2pt}{$\scriptstyle \smash{n}$}}{-}  \text{I} & [2]\text{I}_{l\D, m\D, n\D} &
    \D\underset{\raisebox{2pt}{$\scriptstyle \smash{l\text{-}1}$}}{-}\overset{{\raisebox{2pt}{$\scriptstyle \smash{2\text{-}2}$}}}{\underset{\raisebox{2pt}{$\scriptstyle \smash{m}$}}{\relbar}}\D
    & [2]\text{I}_{m, l\D}
   & \text{T}\overset{{\raisebox{2pt}{$\scriptstyle \smash{3\text{-}3}$}}}{\underset{\raisebox{2pt}{$\scriptstyle \smash{l}$}}{\relbar}} \text{T} & [2]\text{T}_{\{ 6 \}l\D} \\    
     \hline 
\multicolumn{2}{?c?}{\text{\# Components (MRM) }} & l + m + n - 1  & l + m + n + 7 &  l + m + 3 & l +m + 2 & l+7 & l + 6   \\ 
     \hline
\multicolumn{2}{?c?}{\text{N\'eron component group}} & \frac{\mathbb{Z}}{\frac{lm + mn + nl}{\text{gcd}(l,m,n)}\mathbb{Z}} \times \frac{\mathbb{Z}}{\text{gcd}(l,m,n)\mathbb{Z}} & X_{\frac{lm+mn+nl}{\text{gcd}(l,m,n)}} \times X_{\text{gcd}(l,m,n)} &\frac{\mathbb{Z}}{(4l + m)\mathbb{Z} } & \frac{\mathbb{Z}}{m\mathbb{Z}} & Y_l & (0) \\
     \hline
     \multicolumn{2}{?c?}{\text{Tame conductor exponent}} & 2 & 4 & 3 &3 &4 &4 \\ \hline 
       \multicolumn{2}{?c?}{\text{Wild reduction iff} \ p=} & - & 2  & 2  & 2  & 3 & 2,3  \\  \hline
 \multicolumn{2}{?c?}{-v(\omega_{\text{min}}/\omega^0) \text{ \hspace{1em} if $p \not=2$}} & 0   & 0  & 0   & 1  & 0   & 1     \\ 
\htline      
\multicolumn{8}{?c?}{\textup{Entries below this line only valid for curves with tame reduction and }p\neq 2}  \\
     \htline 
\multicolumn{2}{?c?}{v (\Delta_{\text{min}})} & l + m + n & l + m + n + 10 & l+m+5 & l+m+15 & l+10 & l+20 \\ 
\htline
\multicolumn{8}{?c?}{\textup{All possible cluster pictures for reduction type\rlap{\qquad\qquad(displayed value is $v_c$ mod 2)}}}  \\
     \htline 
\LBthreenAa & \boxed{v_c \textup{ mod }2}  &  0^{\sstar}  & 1^{\sstar} & & & & \\
\hline 
\LBthreenAb & \boxed{v_c \textup{ mod }2} & & &   0 &   1 & & \\ \hline
\>\>\>\>\>\> \LBthreenAc & \boxed{v_c \textup{ mod }2} &  & & & &  0^{\sstar} &   1^{\sstar} \\ \hline
\>\LBthreenBa  & \boxed{v_c \textup{ mod }2} & l & l+1  & &  & &\\ \hline
\>\>\>\>\>\>\LBthreenBb & \boxed{v_c \textup{ mod }2} &&   & l+1 &   l& & \\  \hline 
\LBthreenCa & \boxed{v_c \textup{ mod }2} & 
       0 &   1 & & & & \\ \hline 
\LBthreenCb & \boxed{v_c \textup{ mod }2} & 
     & & 0^{\sstar} &   1^{\sstar} & & \\ \hline 
\begin{tabular}[x]{@{}c@{}} \LBthreenDa \\ [-0.8em] $l$ \textup{ even} \end{tabular}  
      &\boxed{v_c \textup{ mod }2}  &   s & s+1 & & & & \\[-0.4em] \hline
\begin{tabular}[x]{@{}c@{}} \LBthreenDb \\ [-0.8em] $l$ \textup{ odd} \end{tabular} & 
    \boxed{v_c \textup{ mod }2}  & & &  s & s + 1 & & \\ 
    \htline
\end{array}
$$
\end{landscape}


\robustify{\potgoodtimestwopicture}

\newpage 
\thispagestyle{empty}
\textwidth 25.5cm
\oddsidemargin -2.3cm
\evensidemargin -2.3cm
\textheight 27cm
\topmargin -1.65cm

\begin{landscape}
\subsection{Table 5 (Stable Type V\phantom{II} if tame and $p\neq 2$ \smash{\potgoodtimestwopicture})}
\label{table:potentiallygoodx2Table2}
(This table spans three pages.)

\noindent{\small{Entries with $\dagger$: if $t\!=\!-1$, decrease $\#$Components by 1, increase $v(\Delta_{\min})$ by 10, set $v(\omega_{\text{min}}/\omega^0)\!=\!0$.}} 

\noindent{\small{Entries with a $\star$ indicate a minimal Weierstrass model (superceded by $\diamond$ when $t=-1$).}}
$$
\hspace{-0.8cm}
\begin{array}{?c | c c? c|c|c|c|c|c|c|c|c|c|c|c|c|c|c|c|c| c?} \htline
\multicolumn{3}{?c?}{\text{MRM/MRNC Reduction type }}
 & 
\!\!\text{I}_0\text{-I}_0\text{-}t\!\! & 
\!\!\text{I}_0\text{-I}_0^{*}\text{-}t\!\! & 
\!\!\text{I}_{0}^{*}\text{-I}_{0}^{*}\text{-}t\!\! & 
\!\!\text{I}_{0}\text{-III}\text{-}t\!\! & 
\!\!\text{I}_{0}^{*}\text{-III}^{*}\!\!\text{-}t\!\! & 
\!\!\text{I}_{0}^{*}\text{-III}\text{-}t\!\! & 
\!\!\text{I}_{0}\text{-III}^{*}\!\!\text{-}t\!\! &  
\!\!\text{I}_{0}\text{-II}\text{-}t\!\! & 
\!\!\text{I}_{0}^{*}\text{-}\text{IV}^{*}\!\!\text{-}t\!\! & 
\!\!\text{I}_{0}^{*}\text{-}\text{II}\text{-}t\!\! & 
\!\!\text{I}_{0}\text{-}\text{IV}^{*}\!\!\text{-}t\!\! & 
\!\!\text{I}_{0}\text{-}\text{IV}\text{-}t\!\! & 
\!\!\text{I}_{0}^{*}\text{-}\text{II}^{*}\!\!\text{-}t\!\! & 
\!\!\text{I}_{0}^{*}\text{-}\text{IV}\text{-}t\!\! & 
\!\!\text{I}_{0}\text{-}\text{II}^{*}\!\!\text{-}t\!\! & 
\!\!\text{III-}\text{III-}t\!\! & 
\!\!\text{III}^{*}\!\text{-}\text{III}^{*}\!\!\text{-}t\!\! & 
\!\!\text{III-}\text{III}^{*}\!\!\text{-}t \!\! 
\\[-0.5em]
\multicolumn{3}{?c?}{\text{(family parameter)}} & t > 0 & t \ge 0& t \ge 0& t \ge 0& \!t \ge -1\! & t \ge 0 & t \ge 0 & t \ge 0& \!t\ge -1\! & t \ge 0 & t \ge 0 & t \ge 0& \!t \ge -1\! &t \ge 0 & t \ge 0 & t \ge 0 & \!t \ge -1\! & t \ge 0\\ \htline
 \multicolumn{3}{?c?}{\text{\# Components (MRM)}} & t\!+\!1 & t\!+\!5 & t\!+\!9 & t\!+\!2 & \!t\!+\!12^{\dagger}\! & t\!+\!6 & t\!+\!8 & t\!+\!1 & \!t\!+\!11^{\dagger}\! & t\!+\!5 & t\!+\!7 & t\!+\!3 & \!t\!+\!13^{\dagger}\! & t\!+\!7 & t\!+\!9 & t\!+\!3 & \!t\!+\!15^{\dagger}\! & t\!+\!9 \\
\hline
 \multicolumn{3}{?c?}{\text{N\'eron component group}} &
(0)&
\!\!\frac{\Z}{2\Z} \!\times\! \frac{\Z}{2\Z}\!\!&
\!\left( \frac{\Z}{2\Z} \right)^{\!4}\! &
\!\frac{\Z}{2\Z}\! &
\!\left( \frac{\Z}{2\Z} \right)^{\!3}\! &
\!\left( \frac{\Z}{2\Z} \right)^{\!3}\! &
\!\frac{\Z}{2\Z}\!&
(0)&
\!\!\frac{\Z}{2\Z} \!\times\! \frac{\Z}{6\Z}\!\!&
\!\!\frac{\Z}{2\Z} \!\times\! \frac{\Z}{2\Z}\!\!&
\frac{\Z}{3\Z}&
\frac{\Z}{3\Z}&
\!\!\frac{\Z}{2\Z} \!\times\! \frac{\Z}{2\Z}\!\!&
\!\!\frac{\Z}{2\Z} \!\times\! \frac{\Z}{6\Z}\!\!&
(0)&
\!\!\frac{\Z}{2\Z} \!\times\! \frac{\Z}{2\Z}\!\!&
\!\!\frac{\Z}{2\Z} \!\times\! \frac{\Z}{2\Z}\!\!& 
\!\!\frac{\Z}{2\Z} \!\times\! \frac{\Z}{2\Z}\!\! \\
\hline
   \multicolumn{3}{?c?}{\text{Tame conductor exponent}} &0&2&4&2&4&4&2&2&4&4&2&2&4&4&2&4&4&4 \\ \hline
   \multicolumn{3}{?c?}{\text{Wild reduction iff} \ p=} &-&2&2&2&2&2&2&2,3&2,3&2,3&3&3&2,3&2,3&2,3&2&2&2\\ \hline  
\multicolumn{3}{?c?}{-v(\omega_{\text{min}}/\omega^0) \text{ \hspace{1em} if $p \not=2$}} & t  & t  & t  & t & t^{\dagger}  & t  & t & t & t^{\dagger} & t & t  & t  & t^{\dagger} & t & t  & t  &t^{\dagger}   &  t\\
\htline
\multicolumn{21}{?c?}{\textup{Entries below this line only valid for curves with tame reduction and }p\neq 2}  \\
\htline
\multicolumn{3}{?c?}{v (\Delta_{\text{min}})}  &
\!\!\scriptstyle{12t}\!\!&
\!\!\scriptstyle{12t\!+\!6}\!\!&
\!\!\scriptstyle{12t\!+\!12}\!\!&
\!\!\scriptstyle{12t\!+\!3}\!\!&
\!\!\scriptstyle{12t\!+\!15}^{\dagger}\!\!&
\!\!\scriptstyle{12t\!+\!9}\!\!&
\!\!\scriptstyle{12t\!+\!9}\!\!&
\!\!\scriptstyle{12t\!+\!2}\!\!&
\!\!\scriptstyle{12t\!+\!14}^{\dagger}\!\!&
\!\!\scriptstyle{12t\!+\!8}\!\!&
\!\!\scriptstyle{12t\!+\!8}\!\!&
\!\!\scriptstyle{12t\!+\!4}\!\!&
\!\!\scriptstyle{12t\!+\!16}^{\dagger}\!\!&
\!\!\scriptstyle{12t\!+\!10}\!\!&
\!\!\scriptstyle{12t\!+\!10}\!\!&
\!\!\scriptstyle{12t\!+\!6}\!\!&
\!\!\scriptstyle{12t\!+\!18}^{\dagger}\!\!&
\!\!\scriptstyle{12t\!+\!12}\!\!\\
\htline 
\multicolumn{21}{?c?}{\textup{All possible cluster pictures for reduction type \rlap{\qquad\qquad(displayed values are $\delta, \epsilon$ or $\gamma,\eta$, as indicated)}}}  \\
\htline
\!\!\!\!\!\!\! \smash{\raisebox{-4pt}{\LBeeB}} \!\!\!\!\! & \!\!\smash{\raisebox{-10pt}{\boxed{\delta, \epsilon}}} & \!\!\!\!\! \underline{v_c\equiv 0} \!\!\! & 0,0^\sstar & 0,1^\sstar & 1,1^{\sstar} & 0, \frac{1}{2}^{\sstar} & 1, \frac{3}{2}^{\sstar} & 1, \frac{1}{2}^{\sstar} & 0, \frac{3}{2}^{\sstar} & 0, \frac{1}{3}^{\sstar} & 1, \frac{4}{3}^{\sstar} & 1,\frac{1}{3}^{\sstar} & 0,\frac{4}{3}^{\sstar} & 0, \frac{2}{3}^{\sstar} & 1, \frac{5}{3}^{\sstar} & 1,\frac{2}{3}^{\sstar} & 0,\frac{5}{3}^{\sstar} & \frac{1}{2}, \frac{1}{2}^{\sstar} & \frac{3}{2}, \frac{3}{2}^\sstar & \frac{1}{2}, \frac{3}{2}^{\sstar}  \\
t=a+b  & & \!\!\!\!\! \underline{v_c\equiv 1} \!\!\! &1,\text{-}1& 0,1 &0,2 &\text{-}1,\frac{3}{2}& 2,\frac{1}{2} & 0,\frac{3}{2} & 1,\frac{1}{2} &\text{-}1,\frac{4}{3} & 2,\frac{1}{3} & 0,\frac{4}{3}& 1,\frac{1}{3} & \text{-}1,\frac{5}{3} & 2,\frac{2}{3} & 0,\frac{5}{3}& 1,\frac{2}{3} & \text{-}\frac{1}{2},\frac{3}{2} & \frac{1}{2},\frac{5}{2}^{\aast} & \frac{1}{2},\frac{3}{2} \\
\hline 

\!\!\!\!\!\!\! \smash{\raisebox{-10pt}{\LBeeCb}} \!\!\!\!\! & \!\!\smash{\raisebox{-10pt}{\boxed{\gamma; \eta}}} &  \!\!\!\!\! \underline{v_c\equiv 0} \!\!\! &0;0& 1;0 & & \frac{1}{2};0 & &  & \frac{3}{2};0 & \frac{1}{3};0 & \frac{7}{3};\frac{2}{3}&  & \!\!\frac{4}{3}; 0\, \text{or}\, \frac{2}{3}\!\!\! &  \!\!\frac{2}{3};0\, \text{or}\, \frac{1}{3}\!\!\! & & \frac{5}{3};\frac{1}{3}& \frac{5}{3};0 & &&  \\
& & \!\!\!\!\! \underline{v_c\equiv 1} \!\!\! && 1;0 & 2;0 &  & \frac{5}{2};0^{\aast}& \frac{3}{2};0 & & \frac{1}{3};\frac{2}{3} & \frac{7}{3};0^{\aast} & \!\!\frac{4}{3};0 \,\text{or}\, \frac{2}{3}\!\!\!& & & \!\!\frac{8}{3};0^{\diamond} \,\text{or}\, \frac{1}{3}\!\!\!& \frac{5}{3};0 & \frac{5}{3},\frac{1}{3}& &&  \\
\hline 

\!\!\!\!\!\!\! \smash{\raisebox{-10pt}{\LBeeDb}} \!\!\!\!\!  & \!\!\smash{\raisebox{-10pt}{\boxed{\gamma; \eta}}}& \!\!\!\!\! \underline{v_c\equiv 0} \!\!\! & 0;0 & \!\!\!1;0\,\text{or}\,1\!\!\! & 2;1 & \!\!\frac{1}{2};0\,\text{or}\,\frac{1}{2}\!\!\! & \!\!\frac{5}{2};1\,\text{or}\,\frac{3}{2}\!\!\! & \!\!\frac{3}{2};1\,\text{or}\,\frac{1}{2}\!\!\! & \!\!\frac{3}{2};0\,\text{or}\,\frac{3}{2}\!\!\! & \frac{1}{3};0 & \frac{7}{3};1 & \frac{4}{3};1 & \frac{4}{3};0 & \frac{2}{3};0 & \frac{8}{3};1 & \frac{5}{3};1 & \frac{5}{3};0 & 1;\frac{1}{2} & 3;\frac{3}{2} & 2;\frac{3}{2}\\
& &  \!\!\!\!\! \underline{v_c\equiv 1} \!\!\! & 0;1&\!\!\! 1;0\,\text{or}\,1\!\!\!& 2;0&\!\!\frac{1}{2};1\,\text{or}\,\frac{3}{2}\!\!\!&\!\!\frac{5}{2};0\,\text{or}\,\frac{1}{2}\!\!\!&\!\!\frac{3}{2};0\,\text{or}\,\frac{3}{2}\!\!\!&\!\!\frac{3}{2};1\,\text{or}\,\frac{1}{2}\!\!\!&\frac{1}{3};1&\frac{7}{3};0&\frac{4}{3};0&\frac{4}{3};1&\frac{2}{3};1 &\frac{8}{3};0&\frac{5}{3};0&\frac{5}{3};1 & 1;\frac{3}{2} &3;\frac{1}{2} & 2;\frac{3}{2}   \\
\htline
\end{array}
$$
\end{landscape}


\newpage 
\thispagestyle{empty}

\textwidth 25.5cm
\oddsidemargin -2.3cm
\evensidemargin -2.3cm
\textheight 27cm
\topmargin -1.65cm

\begin{landscape}
\label{table:potentiallygoodx2Table3}
\noindent (Table 5 continued.)

\noindent{{\small{Entries with $\dagger$: if $t\!=\!-1$, decrease $\#$Components by 1, increase $v(\Delta_{\min})$ by 10, set $v(\omega_{\text{min}}/\omega^0)\!=\!0$.}}} 

\noindent{{\small{Entries with a $\star$ indicate a minimal Weierstrass model (setting $s=0$ where relevant), superceded by $\diamond$ when $t=-1$.}}}
$$
\hspace{-1cm}
\begin{array}{?c| c c ? c|c|c|c|c|c|c|c|c|c|c|c|c|c|c|c|c| c?} \htline
\multicolumn{3}{?c?}{\text{MRM/MRNC Reduction type}} & 
\!\!\text{II-III-}t\!\! & 
\!\!\text{IV}^{*}\!\!\text{-}\text{III}^{*}\!\!\text{-}t\!\! & 
\!\!\text{IV-}\text{III-}t\!\! & 
\!\!\text{II}^{*}\!\!\text{-}\text{III}^{*}\!\!\text{-}t\!\! & 
\!\!\text{IV-}\text{III}^{*}\!\text{-}t\!\!& 
\!\!\text{II}^{*}\!\!\text{-}\text{III-}t\!\! & 
\!\!\text{IV}^{*}\!\!\text{-}\text{III-}t\!\!& 
\!\!\text{II-}\text{III}^{*}\!\!\text{-}t\!\!& 
\!\!\text{II-II-}t\!\! & 
\!\!\text{IV}^{*}\!\!\text{-}\text{IV}^{*}\!\!\text{-}t\!\! & 
\!\!\text{II-IV-}t\!\! & 
\!\!\text{II}^{*}\!\!\text{-}\text{IV}^{*}\!\!\text{-}t\!\! &
\!\!\text{II-}\text{IV}^{*}\!\!\text{-}t\!\! & 
\!\!\text{IV-IV-}t\!\! & 
\!\!\text{II}^{*}\!\!\text{-}\text{II}^{*}\!\!\text{-}t\!\! & 
\!\!\text{IV-}\text{IV}^{*}\!\!\text{-}t\!\! & 
\!\!\text{II-}\text{II}^{*}\!\!\text{-}t\!\! & 
\!\!\text{II}^{*}\!\!\text{-}\text{IV-}t

\\[-0.5em] 
 \multicolumn{3}{?c?}{\text{(family parameter)}} & t\ge 0 & \! t \!\ge\! -1 \! &t \ge 0 &\! t \!\ge\! -1\! & \!t \!\ge\! -1 \!& \!t \!\ge\! -1\!& t \ge 0& t \ge 0 & t \ge 0& \! t \!\ge\! -1 \!& t \ge 0& \!t \!\ge\! -1\! & t \ge 0 & t \ge 0 & \!t \!\ge\! -1\!& t \ge 0& t \ge 0& \! t \!\ge\! -1 \! \\ \htline
\multicolumn{3}{?c?}{\text{\# Components (MRM)}} & t\!+\!2 & \!t\!+\!14^{\dagger}\! & t\!+\!4 & \!t\!+\!16 ^{\dagger}\!& \!t\!+\!10^{\dagger}\! & t\!+\!10^{\dagger} & t\!+\!8 & t\!+\!8 & t\!+\!1 & t\!+\!13^{\dagger} & t\!+\!3 &\! t\!+\!15^{\dagger}\! &\! t\!+\!7 \!& t\!+\!5 & \!t\!+\!17^{\dagger}\! & t\!+\!9 & t\!+\!9 &\! t\!+\!11^{\dagger}\! \\
\hline 
\multicolumn{3}{?c?}{\text{N\'eron component group}}  & \frac{\Z}{2\Z}& \frac{\Z}{6\Z} & \frac{\Z}{6\Z} & \frac{\Z}{2\Z} & \frac{\Z}{6\Z} &\frac{\Z}{2\Z} & \frac{\Z}{6\Z} & \frac{\Z}{2\Z} & (0) & \!\!\frac{\Z}{3\Z} \!\times\! \frac{\Z}{3\Z}\!\! & \frac{\Z}{3\Z} & \frac{\Z}{3\Z} &\frac{\Z}{3\Z} & \!\!\frac{\Z}{3\Z} \!\times\! \frac{\Z}{3\Z}\!\! & (0)  &  \!\!\frac{\Z}{3\Z} \!\times\! \frac{\Z}{3\Z} \!\!& (0) & \frac{\Z}{3\Z}\\
\hline 
\multicolumn{3}{?c?}{\text{Tame conductor exponent}} &4&4&4&4&4&4&4&4&4&4&4&4&4&4&4&4&4&4 \\ \hline 
\multicolumn{3}{?c?}{\text{Wild reduction iff} \ p=} &2,3&2,3&2,3&2,3&2,3&2,3&2,3&2,3&2,3&3&2,3&2,3&2,3&3&2,3&3&2,3&2,3\\ \hline 
\multicolumn{3}{?c?}{-v(\omega_{\text{min}} / \omega^0) \text{ \hspace{1em} if $p \not=2$}} & t & t^{\dagger} & t & t^{\dagger} & t^{\dagger} & t^{\dagger} & t & t & t & t^{\dagger} & t & t^{\dagger} & t & t & t^{\dagger} & t & t& t^{\dagger} \\ 
\htline 
\multicolumn{21}{?c?}{\textup{Entries below this line only valid for curves with tame reduction and }p\neq 2}  \\
\htline
\multicolumn{3}{?c?}{v (\Delta_{\text{min}})}
&\! \scriptstyle{12t\!+\!5} \! 
&\!  \scriptstyle{12t\!+\!17}^{\dagger}\!
&\! \scriptstyle{12t\!+\!7} \!
&\! \scriptstyle{12t\!+\!19}^{\dagger} \!
&\! \scriptstyle{12t\!+\!13}^{\dagger} \! 
&\! \scriptstyle{12t\!+\!13}^{\dagger} \! 
&\! \scriptstyle{12t\!+\!11}\!
&\! \scriptstyle{12t\!+\!11}\!
&\! \scriptstyle{12t\!+\!4} \!
&\! \scriptstyle{12t\!+\!16}^{\dagger} \! 
&\! \scriptstyle{12t\!+\!6} \!
&\! \scriptstyle{12t\!+\!18}^{\dagger} \!
&\! \scriptstyle{12t\!+\!10}\!
&\! \scriptstyle{12t\!+\!8} \!
&\! \scriptstyle{12t\!+\!20}^{\dagger}\!
&\! \scriptstyle{12t\!+\!12}\!
&\! \scriptstyle{12t\!+\!12}\!
&\! \scriptstyle{12t\!+\!14}^{\dagger}\! \\
\htline 
\multicolumn{21}{?c?}{\textup{All possible cluster pictures for reduction type \rlap{\qquad\qquad(displayed values are $\delta, \epsilon$ or $\gamma,\eta$, as indicated)}}}  \\
\htline
\!\!\!\!\!\!\! \smash{\raisebox{-4pt}{\LBeeB}}\!\!\!\!\! & \!\!\smash{\raisebox{-10pt}{\boxed{\delta, \epsilon}}} & \!\!\! \underline{v_c \equiv 0} \!\!& \frac{1}{2}, \frac{1}{3}^{\sstar} & \frac{3}{2}, \frac{4}{3}^{\sstar} & \frac{1}{2}, \frac{2}{3}^{\sstar} & \frac{3}{2}, \frac{5}{3}^{\sstar} & \frac{3}{2}, \frac{2}{3}^{\sstar} & \frac{1}{2}, \frac{5}{3}^{\sstar} & \frac{1}{2}, \frac{4}{3}^{\sstar} & \frac{3}{2}, \frac{1}{3}^{\sstar} & \frac{1}{3}, \frac{1}{3}^{\sstar} & \frac{4}{3}, \frac{4}{3}^{\sstar} & \frac{1}{3}, \frac{2}{3}^{\sstar} & \frac{4}{3}, \frac{5}{3}^{\sstar} & \frac{1}{3}, \frac{4}{3}^{\sstar} & \frac{2}{3}, \frac{2}{3}^{\sstar} & \frac{5}{3}, \frac{5}{3}^{\sstar} & \frac{2}{3}, \frac{4}{3}^{\sstar} & \frac{5}{3}, \frac{1}{3}^{\sstar} & \frac{2}{3}, \frac{5}{3}^{\sstar}  \\
{t = a+b \>\>\>\>\>} & & \!\!\! \underline{v_c\equiv 1} \!\!& 
\text{-}\frac{1}{2}, \frac{4}{3} & \frac{5}{2}, \frac{1}{3}^{\aast}& \text{-}\frac{1}{2}, \frac{5}{3} & \frac{5}{2}, \frac{2}{3}^{\aast}& \frac{1}{2}, \frac{5}{3} & \frac{3}{2}, \frac{2}{3} & \frac{3}{2}, \frac{1}{3} & \frac{1}{2}, \frac{4}{3} & \frac{4}{3}, \text{-}\frac{2}{3} & \frac{1}{3}, \frac{7}{3}^{\aast}& \frac{4}{3}, \text{-}\frac{1}{3}& \frac{7}{3}, \frac{2}{3}^{\aast}& \frac{4}{3}, \frac{1}{3}& \frac{5}{3}, \text{-}\frac{1}{3}&  \frac{8}{3}, \frac{2}{3}^{\aast}& \frac{5}{3}, \frac{1}{3} & \frac{2}{3}, \frac{4}{3} & \frac{5}{3}, \frac{2}{3} \\ \hline
 \!\!\!\!\!\!\! \smash{\raisebox{-10pt}{\LBeeCb}}\!\!\! & \!\!\smash{\raisebox{-10pt}{\boxed{\gamma; \eta}}} & \!\!\! \underline{v_c \equiv 0}\!\! & & \frac{17}{6}; \frac{2}{3} & \frac{7}{6}; \frac{1}{3}&  & \frac{13}{6}; \frac{1}{3}^{\aast} & & \frac{11}{6}; \frac{2}{3} &  & & \frac{8}{3}; \frac{2}{3} & 1; \frac{1}{3} & 3; \frac{2}{3} & \frac{5}{3}; \frac{2}{3}& \frac{4}{3}; \frac{1}{3} & & \!\!2; \frac{1}{3}\,\text{or}\,\frac{2}{3}\!\!\! & & \frac{7}{3}; \frac{1}{3}^{\aast} \\
& & \!\!\! \underline{v_c \equiv 1}\!\!& \frac{5}{6}; \frac{2}{3} &&& \frac{19}{6}; \frac{1}{3} && \frac{13}{6}; \frac{1}{3} && \frac{11}{6}; \frac{2}{3}& \frac{2}{3}; \frac{2}{3} && 1;\frac{2}{3} & 3;\frac{1}{3} & \frac{5}{3};\frac{2}{3} && \frac{10}{3}; \frac{1}{3} & & \!\!2; \frac{1}{3}\,\text{or}\,\frac{2}{3}\!\!\! & \frac{7}{3}; \frac{1}{3} \\ \hline
\!\!\!\!\!\! \smash{\raisebox{-10pt}{\LBeeDb}} \!\!\!\! & \!\! \smash{\raisebox{-10pt}{\boxed{\gamma; \eta}}}& \!\!\! \underline{v_c \equiv 0} \!\!& \frac{5}{6}; \frac{1}{2} & \frac{17}{6}; \frac{3}{2} & \frac{7}{6}; \frac{1}{2} & \frac{19}{6}; \frac{3}{2} & \frac{13}{6}; \frac{3}{2} & \frac{13}{6}; \frac{1}{2}^{\aast} &  \frac{11}{6};  \frac{1}{2} & \frac{11}{6}; \frac{3}{2}  &&&&&&&&&&\\
& & \!\!\! \underline{v_c \equiv 1} \!\!& \frac{5}{6};\frac{3}{2} & \frac{17}{6};\frac{1}{2} & \frac{7}{6};\frac{3}{2} & \frac{19}{6};\frac{1}{2} & \frac{13}{6};\frac{1}{2} & \frac{13}{6};\frac{3}{2} & \frac{11}{6};\frac{3}{2} & \frac{11}{6};\frac{1}{2} &&&&&&&&&&  \\
\htline 
\end{array}
$$
\end{landscape}


\newpage
\thispagestyle{empty}

\textwidth 25.5cm
\oddsidemargin -2.3cm
\evensidemargin -2.3cm
\textheight 27cm
\topmargin -2.5cm

\begin{landscape}
\label{table:potentiallygoodx2Table1}
\noindent (Table 5 continued.)

\vphantom{blah}

\noindent
$$
\hspace{-1cm}
\begin{array}{?c|c?c|c|c|c|c|c|c|c?} \htline
\multicolumn{2}{?c?}{\text{MRM Reduction type }} & 2\text{I}_{0}\text{-}t & 2\text{I}_0^{*}\text{-}t& 2\text{IV}\text{-}t  & 2\text{IV}^{*}\text{-}t & 2\text{III}\text{-}t & 2\text{III}^{*}\text{-}t & 2\text{II}\text{-}t & 2\text{II}^{*}\text{-}t  \\[-0.5em]
\multicolumn{2}{?c?}{\textup{(family parameter)}}
& t>0 & t\geq 0 & t\geq 0 & t\geq 0 & t\geq 0 & t\geq 0 & t\geq 0 & t\geq 0 \\
\htline 
\multicolumn{2}{?c?}{\text{MRNC Reduction type}}  & [2]\text{I}_{\text{g}1,{t\D}}& [2]\mathrm{I}^*_{0,t\D} & [2]\mathrm{IV}_{t\D}&  [2]\mathrm{IV}^*_{t\D}& [2]\mathrm{III}_{t\D} & [2]\mathrm{III}^*_{t\D} & [2]\mathrm{II}_{t\D} & [2]\mathrm{II}^*_{t\D} \\
\hline 
  \multicolumn{2}{?c?}{\text{\# Components (MRM) }} & t+3 & t+7 & t+5 & t+9 & t+4 & t+10 & t+3 & t+11 \\  \hline 
\multicolumn{2}{?c?}{\text{N\'eron component group}} & (0) & \frac{\Z}{2\Z} \times \frac{\Z}{2\Z} & \frac{\Z}{3\Z}& \frac{\Z}{3\Z}& \frac{\Z}{2\Z}& \frac{\Z}{2\Z}& (0)& (0) \\
\hline 
\multicolumn{2}{?c?}{\text{Tame conductor exponent}} & 2 & 4 & 4 & 4 & 4 & 4 & 4 & 4 \\ \hline 
\multicolumn{2}{?c?}{\text{Wild reduction iff} \ p=} & 2 & 2 & 2,3 & 2,3 &2  &2  & 2,3 & 2,3 \\ \hline
   \multicolumn{2}{?c?}{-v(\omega_{\text{min}} / \omega^0) \text{ \hspace{1em} if $p \not=2$}} & t+1 & t+1 & t+1 & t+1 & t+1 & t+1 & t+1 & t+1 \\ 
\htline 
\multicolumn{10}{?c?}{\textup{Entries below this line only valid for curves with tame reduction and }p\neq 2}  \\
\htline 
\multicolumn{2}{?c?}{v (\Delta_{\text{min}})} & 12t+ 15 & 12t+ 21 & 12t+ 19 & 12t+ 23 & 12t+ 18 & 12t+ 24 & 12t+ 17 & 12t+ 25 \\ 
   \htline 
   \multicolumn{10}{?c?}{\textup{All possible cluster pictures for reduction type \rlap{\qquad\qquad(displayed value is $\delta$)}}}  \\
\htline 
 \LBeeA &  \>\> \boxed{\delta} \>\>\>\>\> v_c \equiv 0 \text{ or } 1 & 0 &  \frac{1}{2} &  \frac{1}{3} &  \frac{2}{3} & \frac{1}{4} & \frac{3}{4} & \frac{1}{6} & \frac{5}{6} \\
\htline 
\end{array}
$$
\end{landscape}


\robustify{\potgoodtimespotmultpicture}

\newpage
\thispagestyle{empty}

\textwidth 25.5cm
\oddsidemargin -2.3cm
\evensidemargin -2.3cm
\textheight 27cm
\topmargin -1.75cm

\begin{landscape}
\subsection{Table 6 (Stable Type VI\phantom{I} if tame and $p\neq 2$ \smash{\potgoodtimespotmultpicture}\hspace{-3pt})}
\label{table:potentiallygoodxpotentiallymultiplicative} 
\hfill {\small{$X_l$ denotes $\frac{\Z}{2\Z}\!\times\! \frac{\Z}{2\Z}$ if $l$ is even, and $\frac{\Z}{4\Z}$ if $l$ is odd.}}
\\
\noindent 
{\small{Entries with a $\star$ indicate a minimal Weierstrass model (superceded by $\diamond$ when $t=-1$). \hfill Entries with $\dagger$: if $t\!=\!-1$, decrease $\#$Components by 1, increase $v(\Delta_{\min})$ by 10, set $v(\omega/\omega^0)\!=\!0$.}}
$$
\hspace{-1cm}
\begin{array}{?c|cc?c|c|c|c|c|c|c|c|c|c|c|c|c|c|c|c?} \htline
    \multicolumn{3}{?c?}{ \text{MRM/MRNC Reduction type}} &
 \!\text{I}_{0}\text{-I}_{l}\text{-}t \! & 
     \!\text{I}_{0}^*\text{-I}^{*}_{l}\text{-}t \! & 
     \!\text{I}_{0}\text{-I}^{*}_{l}\text{-}t \! & 
     \!\text{I}_{0}^*\text{-I}_{l}\text{-}t \! & 
     \!\text{III}\text{-I}_{l}\text{-}t \! &  
     \!\text{III}^{*}\text{-I}^{*}_{l}\text{-}t \! & 
     \!\text{III-I}^{*}_{l}\text{-}t \! & 
     \!\text{III}^{*}\text{-I}_{l}\text{-}t \! & 
     \!\text{II-I}_{l}\text{-}t \! & 
     \!\text{IV}^{*}\text{-I}^{*}_{l}\text{-}t \! & 
     \!\text{II-I}^{*}_{l}\text{-}t \! & 
     \!\text{IV}^{*}\text{-I}_{l}\text{-}t \! & 
     \!\text{IV-I}_{l}\text{-}t \! & 
     \!\text{II}^{*}\text{-I}^{*}_{l}\text{-}t \! &  
     \!\text{IV-I}^{*}_{l}\text{-}t \! & 
     \!\text{II}^{*}\text{-I}_{l}\text{-}t \!
    \\[-0.9em] 
\multicolumn{3}{?c?}{\text{(family parameter)}} & t > 0 & t \ge 0 & t \ge 0 & t \ge 0 & t \ge 0 & t \ge -1 & t \ge 0 & t \ge 0 & t \ge 0 & t \ge -1 & t \ge 0 & t \ge 0& t \ge 0 & t \ge -1 & t \ge 0& t \ge 0\\ \htline
 \multicolumn{3}{?c?}{\textup{\# Components (MRM)}} & \!l\!+\!t\!  & \!l\!+\!t\!+\!9\!   &  \!l\!+\!t\!+\!5\! & \!l\!+\!t\!+\!4\!  & \!l\!+\!t\!+\!1\! & \!l\!+\! t\!+\!12^\dagger\!\! & \!l\!+\!t\!+\!6\!  & \!l\!+\!t\!+\!7\!  & \!l\!+\!t\! & \!l\!+\!t\!+\!11^\dagger\!\! & \!l\!+\!t\!+\!5\!  & \!l\!+\!t\!+\!6\!  & \!l\!+\!t\!+\!2\!  & \!l\!+\!t\!+\!13^\dagger\!\! & \!l\!+\!t\!+\!7\!  & \!l\!+\!t\!+8\! \\
 \hline
\multicolumn{3}{?c?}{\text{N\'eron component group}}
    & \!\!\frac{\Z}{l \Z}\!\!  
    & \!\!\left(\frac{\Z}{2 \Z} \right)^{\!2}\!\!\times\! X_{l}\!\! 
    & X_{l}  
    &\!\!\left(\frac{\Z}{2 \Z}\right)^{\!2}\!\!\times\! \frac{\Z}{l \Z}\!\!
    & \!\!\frac{\Z}{2 \Z} \!\times\! \frac{\Z}{l \Z}\!\! 
    & \!\!\frac{\Z}{2 \Z} \!\times\! X_{l}\!\! 
    & \!\!\frac{\Z}{2 \Z} \!\times\! X_{l}\!\! 
    & \!\!\frac{\Z}{2 \Z} \!\times\! \frac{\Z}{l \Z}\!\! 
    & \!\!\frac{\Z}{l \Z} \!\! 
    & \!\!\frac{\Z}{3 \Z} \!\times\! X_{l} \!\!
    & X_{l} 
    & \!\!\frac{\Z}{3 \Z} \!\times\! \frac{\Z}{l \Z} \!\!
    & \!\!\frac{\Z}{3 \Z} \!\times\! \frac{\Z}{l \Z} \!\!
    & X_{l} 
    &\!\!\frac{\Z}{3 \Z} \!\times\! X_{l} \!\! 
    &\!\!\frac{\Z}{l \Z} \!\!\\ \hline
    \multicolumn{3}{?c?}{\text{Tame conductor exponent}} &1&4&2&3&3&4&4&3&3&4&4&3&3&4&4&3 \\ \hline 
       \multicolumn{3}{?c?}{\text{Wild reduction iff} \ p=} & - & 2& 2& 2& 2& 2& 2& 2& 2,3& 2,3& 2,3&3&3&2,3&2,3&2,3 \\ [-0.1em] \hline
\multicolumn{3}{?c?}{-v(\omega_{\text{min}} / \omega^0) \text{ \hspace{1em} if $p \not=2$}} & t & t  & t  & t  & t & t^\dagger  &  t  & t  & t  & t^\dagger  & t  & t  & t  & t^\dagger & t   & t   \\[-0.1em]
       \htline 
\multicolumn{19}{?c?}{\textup{Entries below this line only valid for curves with tame reduction and }p\neq 2}  \\[-0.1em] \htline
\multicolumn{3}{?c?}{v(\Delta_{\text{min}})} & 
\!\!\!{\scriptstyle{12t+l}}\!\!\! & 
\!\!\!{\scriptstyle{12t+l+12}}\!\!\!  &
\!\!\!{\scriptstyle{12t+l+6}}\!\!\!& 
\!\!\!{\scriptstyle{12t+l+6}}\!\!\!& 
\!\!\!{\scriptstyle{12t+l+3}}\!\!\!&
\!\!\!{\scriptstyle{12t+l+15}^\dagger}\!\!\!\!\!&
\!\!\!{\scriptstyle{12t+l+9}}\!\!\!& 
\!\!\!{\scriptstyle{12t+l+9}}\!\!\!& 
\!\!\!{\scriptstyle{12t+l+2}}\!\!\!&
\!\!\!{\scriptstyle{12t+l+14}^\dagger}\!\!\!\!\!&
\!\!\!{\scriptstyle{12t+l+8}}\!\!\!& 
\!\!\!{\scriptstyle{12t+l+8}}\!\!\!& 
\!\!\!{\scriptstyle{12t+l+4}}\!\!\!&
\!\!\!{\scriptstyle{12t+l+16}^\dagger}\!\!\!\!\!&
\!\!\!{\scriptstyle{12t+l+10}}\!\!\!&
\!\!\!{\scriptstyle{12t+l+10}}\!\!\! \\ 
\htline 
\multicolumn{19}{?c?}{\textup{ All possible cluster pictures for reduction type \rlap{\qquad\qquad(displayed values are $\delta, \epsilon$ or $\gamma$ or $\gamma,\eta$, as indicated)}}}  \\[-0.2em] \htline
\!\!\!\!\!\!\raisebox{-3pt}{\LBemA}\!\!\!
&\!\!\smash{\raisebox{-5pt}{\boxed{\delta;\epsilon}}}\!\!\! &  \!\!\underline{v_c\equiv 0}\!\!  & {0,0}^{\sstar} & {1,1}^{\sstar} & {1,0}^{\sstar} & {0,1}^{\sstar} & {0,\frac{1}{2}}^{\sstar} & {1,\frac{3}{2}}^{\sstar} & {1,\frac{1}{2}}^{\sstar} & {0,\frac{3}{2}}^{\sstar} & {0,\frac{1}{3}}^{\sstar} & {1,\frac{4}{3}}^{\sstar} & {1,\frac{1}{3}}^{\sstar} & {0,\frac{4}{3}}^{\sstar} & {0,\frac{2}{3}}^{\sstar} & {1,\frac{5}{3}}^{\sstar} & {1,\frac{2}{3}}^{\sstar} & {0,\frac{5}{3}}^{\sstar}
     \\[-1.1em]
\smash{\raisebox{-5pt}{$t=a+b$}} & & \!\! \underline{v_c\equiv 1} \!\!& \text{-}1,1 & 0,2 & 0,1 & 1,0 &1, \text{-}\frac{1}{2} & 2, \frac{1}{2} & 0, \frac{3}{2} & 1, \frac{1}{2} & 1, \text{-}\frac{2}{3} & 2, \frac{1}{3} & 0, \frac{4}{3} & 1, \frac{1}{3} & 1, \text{-}\frac{1}{3} & 2, \frac{2}{3} & 0, \frac{5}{3} & 1, \frac{2}{3} \\ 
\hline  

\!\!\!\!\!\! \smash{\raisebox{-5pt}{\LBemB}}\!\!\! &\smash{\raisebox{-5pt}{\boxed{\gamma}}} & \!\!\underline{v_c\equiv 0}\!\! &0 &  &  & 1& \frac{1}{2}& & & \frac{3}{2} & \frac{1}{3} & & & \frac{4}{3}& \frac{2}{3}& & & \frac{5}{3} \\[-1em]
&&\underline{v_c\equiv 1} & & 2 & 1 & & & \frac{5}{2}^{\aast} & \frac{3}{2} &  &  & \frac{7}{3}^{\aast} & \frac{4}{3} & & & \frac{8}{3}^{\aast} & \frac{5}{3} & \\ \hline

\!\!\!\!\!\!\smash{\raisebox{-5pt}{\LBemC}}\!\!\! &\!\!\smash{\raisebox{-5pt}{\boxed{\gamma; \eta}}}\!\!\!& \!\!\underline{v_c\equiv 0}\!\! &0,0& & 1, 0&& & & & & & \frac{7}{3}, \frac{2}{3}&& \frac{4}{3}, \frac{2}{3}& \frac{2}{3}, \frac{1}{3}& &\frac{5}{3}, \frac{1}{3}&  \\[-1em]
&&\!\!\underline{v_c\equiv 1}\!\!& & 2,0 & & 1,0 & & & & & \frac{1}{3},\frac{2}{3} & & \frac{4}{3}, \frac{2}{3} & & & \frac{8}{3}, \frac{1}{3} & & \frac{5}{3}, \frac{1}{3} \\ \hline   

\!\!\!\!\!\!\smash{\raisebox{-5pt}{\LBemD}}\!\!\! & \smash{\raisebox{-5pt}{\boxed{\gamma}}} &\!\!\underline{v_c\equiv l}\!\! & 0 &  &  & 1 & \frac{1}{2} & & & \frac{3}{2} & \frac{1}{3} & & & \frac{4}{3} & \frac{2}{3} & & & \frac{5}{3} \\[-1em]
& & \!\! \underline{v_c\not\equiv l} \!\!& & 2 & 1 & & & \frac{5}{2} & \frac{3}{2} & & & \frac{7}{3} & \frac{4}{3} & & & \frac{8}{3} & \frac{5}{3} & \\ \hline   
\!\!\!\!\!\!\smash{\raisebox{-5pt}{\LBemE}}\!\!\!
&\smash{\raisebox{-5pt}{\boxed{\gamma}}}&\!\!\underline{v_c\equiv 0}\!\! & 0 &  &  & 1 & \frac{1}{2} &  &  & \frac{3}{2} & \frac{1}{3} &  &  & \frac{4}{3} & \frac{2}{3} & & & \frac{5}{3} \\[-1em]
 &&\underline{v_c\equiv 1} \!\! &  & 2 & 1 &  &  & \frac{5}{2} & \frac{3}{2} &  &  & \frac{7}{3} & \frac{4}{3} & &  & \frac{8}{3} & \frac{5}{3} &  \\ \hline
\!\!\!\!\!\!\smash{\raisebox{-5pt}{\LBemFb}}\!\!\! &\!\!\smash{\raisebox{-5pt}{\boxed{\gamma; \eta}}}\!\!\!&\!\!\underline{v_c\equiv 0}\!\!
& 0,0  & 2,1 & 1,0 & 1,1 & \frac{1}{2},\frac{1}{2} & \frac{5}{2},\frac{3}{2} & \frac{3}{2},\frac{1}{2} & \frac{3}{2},\frac{3}{2}& & & & & & & &  \\[-0.7em]
&&\!\! \underline{v_c\equiv 1}\!\! & 0,1  & 2,0 & 1,1 & 1,0 & \frac{1}{2},\frac{3}{2} & \frac{5}{2},\frac{1}{2} & \frac{3}{2},\frac{3}{2} & \frac{3}{2},\frac{1}{2}& & & & & & & &
\\ \hline 
\!\!\!\!\!\!\smash{\raisebox{-5pt}{\LBemG}}\!\!\! & \smash{\raisebox{-5pt}{\boxed{\gamma}}} &  \!\!\underline{v_c\equiv s} \!\! &0 &  &  & 1& \frac{1}{2}& & & \frac{3}{2} & \frac{1}{3} & & & \frac{4}{3}& \frac{2}{3}& & & \frac{5}{3} \\[-1em]
&& \!\! \underline{v_c\not\equiv s}\!\! & & 2 & 1 & & & \frac{5}{2}& \frac{3}{2} &  &  & \frac{7}{3} & \frac{4}{3} & & & \frac{8}{3} & \frac{5}{3} & \\ \hline

\!\!\!\!\!\raisebox{-3pt}{\LBemH}\!\!\! & \smash{\raisebox{-5pt}{\boxed{\gamma}}} &  \!\!\underline{v_c\equiv s} \!\! &0 &  &  & 1& \frac{1}{2}& & & \frac{3}{2} & \frac{1}{3} & & & \frac{4}{3}& \frac{2}{3}& & & \frac{5}{3}
\\[-1.8em]
\smash{\raisebox{-8pt}{$l$ {\textup{ even}}}} & & \!\! \underline{v_c\not\equiv s}\!\! & & 2 & 1 & & & \frac{5}{2}& \frac{3}{2} &  &  & \frac{7}{3} & \frac{4}{3} & & & \frac{8}{3} & \frac{5}{3} & \\[0.2em] \htline  
\end{array}
$$
\end{landscape}

\robustify{\potmulttimestwopicture}

\newpage
\thispagestyle{empty}

\textwidth 25.5cm
\oddsidemargin -2.3cm
\evensidemargin -2.3cm
\textheight 27cm
\topmargin -2.5cm

\begin{landscape}
\subsection{Table 7 (Stable Type VII if tame and $p\neq 2$ \smash{\potmulttimestwopicture}\hspace{-3pt})}
\label{table:potentiallymultiplicativex2}
 \vphantom{blah}
  \noindent {\small{$X_l$ denotes $\frac{\Z}{2\Z}\!\times\! \frac{\Z}{2\Z}$ if $l$ is even, and $\frac{\Z}{4\Z}$ if $l$ is odd. \hfill Entries with $\star$ indicate a minimal Weierstrass model.}}
$$
\hspace{-1cm}
\begin{array}{? c |c c ? c| c| c c| c| c?}\htline
\multicolumn{3}{?c?}{\text{MRM Reduction type}} & \text{I}_l\text{-}\text{I}_m\text{-}t & \text{I}_l^*\text{-}\text{I}_m^{*}\text{-}t  & \phantom{ha} \text{I}_l^*\text{-}\text{I}_m\text{-}t \phantom{ha}& \phantom{ha}\text{I}_l\text{-}\text{I}_m^*\!-\!t \phantom{ha}& 2\text{I}_l\text{-}t & 2\text{I}_l^*\text{-}t \\[-0.9em]
\multicolumn{3}{?c?}{\textup{(family parameter)}}& t  >0 & t \ge 0& t \ge 0& t \ge 0& t> 0 & t \geq 0 \\ 
\htline
\multicolumn{3}{?c?}{\text{MRNC Reduction type}}  & 
    \text{I}_l\!\underset{\raisebox{2pt}{$\scriptstyle \smash{t}$}}{-}\!\text{I}_m\! & 
    \text{I}_l^*\!\underset{\raisebox{2pt}{$\scriptstyle \smash{t}$}}{-}\!\text{I}_m^{*}\!  & 
    \text{I}_l^*\!\underset{\raisebox{2pt}{$\scriptstyle \smash{t}$}}{-}\!\text{I}_m\! & 
    \text{I}_l\!\underset{\raisebox{2pt}{$\scriptstyle \smash{t}$}}{-}\!\text{I}_m^*\! & 
    [2]\text{I}_{l,t\D} & 
    [2]\text{I}_{l,t\D}^*  \\
\hline
\multicolumn{3}{?c?}{\textup{\# Components (MRM)}} & l\!+\!m\!+\!t\!-\!1   & l\!+\!m\!+\!t\!+\!9 & \multicolumn{2}{|c|}{l\!+\!m\!+\!t\!+\!4} &l\!+\!t\!+\!2 &  l\!+\!t\!+\!7 \\
\hline
\multicolumn{3}{?c?}{\text{N\'eron component group}} & \frac{\Z}{l\Z}\times \frac{\Z}{m\Z} & X_l\times X_m & X_l\times \frac{\Z}{m\Z} & X_m \times \frac{\Z}{l\Z} &  \frac{\Z}{l\Z} & X_l \\
\hline
\multicolumn{3}{?c?}{\text{Tame conductor exponent}} &2&4&\multicolumn{2}{|c|}{3}&3&4 \\
\hline 
\multicolumn{3}{?c?}{\text{Wild reduction iff } p =} &-&2&\multicolumn{2}{|c|}{2}&2&2 \\[-0.2em]
\hline
\multicolumn{3}{?c?}{-v(\omega_{\min}  / \omega^0) \text{ \hspace{1em} if $p \not=2$}} &t &t & \multicolumn{2}{|c|}{t} &t + 1 & t + 1\\
\htline
\multicolumn{9}{?c?}{\textup{Entries below this line only valid for curves with tame reduction and }p\neq 2}\\ \htline
\multicolumn{3}{?c?}{v(\Delta_{\textup{min}})}&\>\>12t\!+\!l\!+\!m\>\>\>&\>\>12t\!+\!l\!+\!m\!+\!12\>\>\>&   \multicolumn{2}{|c|}{12t\!+\!l\!+\!m\!+\!6}&{\>\>12t\!+\!l\!+15\>\>\>}&{\>\>12t\!+\!l\!+\!21\>\>\>} \\
\htline
\multicolumn{9}{?c?}{\textup{All possible cluster pictures for reduction type\rlap{\qquad(displayed values are $\delta,\epsilon$ or $\delta$ or $\gamma$, as indicated)}}}\\[-0.1em] \htline
\phantom{ha}
\smash{\raisebox{-4pt}{\LBpmpmA}} 
\phantom{ha}
& 
\phantom{ha}
\smash{\raisebox{-10pt}{\boxed{\delta, \epsilon}}} 
& 
\phantom{a}\>\>
\underline{v_c \equiv 0} 
\phantom{ha}
& \>\>0,0^{\sstar} & 1,1^{\sstar} & 1,0^{\sstar} & 0,1^{\sstar} & & \\
t\!=\!a\!+\!b & & \underline{v_c \equiv 1} & {\text{-}1,1} & {0,2} & {0,1} & {1,0} &  & \\
\hline
\smash{\raisebox{-10pt}{\LBpmpmB}} & \smash{\raisebox{-10pt}{\boxed{\delta}}} & \underline{v_c \equiv 0} &&&&& 0^{\sstar} & \frac{1}{2}^{\sstar}\\[-0.8em]
& & \underline{v_c \equiv 1}&&&&& 0 & {\frac{1}{2}}\\
\hline
\smash{\raisebox{-10pt}{\LBpmpmC}} & \smash{\raisebox{-10pt}{\boxed{\gamma}}} & \underline{v_c \equiv 0}  & 0&& 1&& &  \\[-0.8em] 
& & \underline{v_c \equiv 1} &&2&&1&&  \\ 
\hline
\smash{\raisebox{-8pt}{\LBpmpmD}} & \smash{\raisebox{-10pt}{\boxed{\gamma}}} & \underline{v_c \equiv m} & {0} &  & {1} &  &&\\[-0.8em]
& & \underline{v_c \not\equiv m} & & {2} & & {1} && \\
\hline
\raisebox{-3pt}{\LBpmpmE} & \smash{\raisebox{-10pt}{\boxed{\gamma}}} & \underline{v_c \equiv s}& {0}  & & {1} & && \\[-1.2em]
\smash{\raisebox{-4pt}{$m$\textup{ even}}}
& & \underline{v_c\not\equiv s}&& {2}&&{1}&& \\
\hline
\smash{\raisebox{-8pt}{\LBpmpmF}} & \smash{\raisebox{-10pt}{\boxed{\gamma}}} & \underline{v_c \equiv s}  & 0&& 1&& &  \\[-0.8em] 
& & \underline{v_c \not\equiv s} &&2&&1&&  \\ 
\hline
\smash{\raisebox{-8pt}{\LBpmpmG}} & \smash{\raisebox{-10pt}{\boxed{\gamma}}} & \underline{v_c \equiv 0}  & 0&& 1&& &  \\[-0.8em] 
& & \underline{v_c \equiv 1} &&2&&1&&  \\ 
\htline
\end{array}
$$
\end{landscape}


\newpage 

\textwidth 25.5cm
\topmargin 0cm
\oddsidemargin -2.3cm
\evensidemargin -2.3cm
\topmargin -1.75cm

\newsavebox{\picOneA}
\newsavebox{\picOneB}
\sbox{\picOneA}{\pLBpgA}
\sbox{\picOneB}{\pLBpgB}

\newsavebox{\picTwoTR}
\newsavebox{\picTwoML}
\newsavebox{\picTwoMR}
\newsavebox{\picTwoBL}
\newsavebox{\picTwoBR}
\sbox{\picTwoTR}{\pLBonE} 
\sbox{\picTwoML}{\pLBonA} 
\sbox{\picTwoMR}{\pLBonB} 
\sbox{\picTwoBL}{\pLBonD} 
\sbox{\picTwoBR}{\pLBonC}

\newsavebox{\picThreeDaTL}
\newsavebox{\picThreeAa}
\newsavebox{\picThreeDaBL}
\newsavebox{\picThreeBaTR}
\newsavebox{\picThreeBaMR}
\newsavebox{\picThreeBbBR}
\sbox{\picThreeDaTL}{\pLBtwonDa} 
\sbox{\picThreeAa}{\pLBtwonAa}   
\sbox{\picThreeDaBL}{\pLBtwonCa} 
\sbox{\picThreeBaTR}{\pLBtwonE} 
\sbox{\picThreeBaMR}{\pLBtwonBa} 
\sbox{\picThreeBbBR}{\pLBtwonBb}

\newsavebox{\picFourTL}
\newsavebox{\picFourTR}
\newsavebox{\picFourBL}
\newsavebox{\picFourBR}
\sbox{\picFourTL}{\pLBthreenAa} 
\sbox{\picFourTR}{\pLBthreenCa} 
\sbox{\picFourBL}{\pLBthreenBa} 
\sbox{\picFourBR}{\pLBthreenDa}

\newsavebox{\picFiveBL}
\newsavebox{\picFiveTL}
\newsavebox{\picFiveTR}
\sbox{\picFiveBL}{\pLBeeDb} 
\sbox{\picFiveTL}{\pLBeeB}  
\sbox{\picFiveTR}{\pLBeeCb}

\newsavebox{\picSixCenter}
\newsavebox{\picSixLeft}
\newsavebox{\picSixTop}
\newsavebox{\picSixRight}
\newsavebox{\picSixBottom}
\newsavebox{\picSixTopLeft}
\newsavebox{\picSixBottomLeft}
\newsavebox{\picSixTopRight}
\sbox{\picSixCenter}{\pLBemB} 
\sbox{\picSixLeft}{\pLBemA}   
\sbox{\picSixTop}{\pLBemD}    
\sbox{\picSixRight}{\pLBemC}  
\sbox{\picSixBottom}{\pLBemE}
\sbox{\picSixTopLeft}{\pLBemFa}
\sbox{\picSixTopRight}{\pLBemG}
\sbox{\picSixBottomLeft}{\pLBemH}

\newsavebox{\picSevenL}
\newsavebox{\picSevenC}
\newsavebox{\picSevenR}
\newsavebox{\picSevenRR}
\newsavebox{\picSevenT}
\newsavebox{\picSevenD}

\sbox{\picSevenL}{\pLBpmpmA}
\sbox{\picSevenC}{\pLBpmpmC}
\sbox{\picSevenR}{\pLBpmpmD}
\sbox{\picSevenRR}{\pLBpmpmE}
\sbox{\picSevenT}{\pLBpmpmF}
\sbox{\picSevenD}{\pLBpmpmG}

\thispagestyle{empty}
\begin{landscape}  

\addtocontents{toc}{\protect\setcounter{tocdepth}{1}}

\subsection{All possible cluster pictures for potential stable type if $p\neq 2$}\label{ss:changeofmodel}
\begin{longtable}{|l  l | l |}
\hline
\multicolumn{2}{|l|}{ {Potential stable type}} & 
{All possible cluster pictures for potential stable type if $p\neq 2$, and model transformations between them when tame (see Definition \ref{def:equiv-cpics}(5))} \\
\hline
{\textbf{I}} & {(smooth curve of genus 2)} & 
\scalebox{0.95}{%
$\begin{tikzcd}[column sep=0.6cm, ampersand replacement=\&]
  \vcenter{\hbox{\usebox{\picOneA}}} 
    \arrow[in=170, out=190, loop, min distance=6mm, "{D}" ]
    \arrow[r, "A", shift left] 
  \& \vcenter{\hbox{\usebox{\picOneB}}} 
    \arrow[l, "B", shift left]
    \arrow[out=350, in=10, loop, min distance=6mm, "A" swap]
\end{tikzcd}$%
} \\
\hline

{\textbf{II}} & {(elliptic curve with a node)} & 
\scalebox{0.95}{%
$\begin{tikzcd}[column sep=0.6cm, row sep=0.6cm, ampersand replacement=\&]
\vcenter{\hbox{\usebox{\picTwoML}}} 
    \arrow[r, "A",  shift left] 
    \arrow[d, "A", shift left]
    \&  \vcenter{\hbox{\usebox{\picTwoBR}}} 
    \arrow[l, "B", shift left] 
    \arrow[r, "B" shift left, yshift=2pt]
    \arrow[dl, "C", shift right=-0.5ex]
    \arrow[in=250, out=290, loop, min distance=4mm, "A"]
  \& \vcenter{\hbox{\usebox{\picTwoMR}}}
    \arrow[l, "A", shift left] 
    \arrow[d, "A", shift left] \\
    \vcenter{\hbox{\usebox{\picTwoBL}}} 
   \arrow[u, "A", shift left] 
    \arrow[ur, "C", shift right=-0.5ex]
    \arrow[in=173, out=187, loop, min distance=6mm, "A" ]
    \& 
  \& \vcenter{\hbox{\usebox{\picTwoTR}}} 
    \arrow[u, "A", shift left] 
    \arrow[in=350, out=10, loop, min distance=6mm, "A" ]
\end{tikzcd}$%
} \\
\hline

{\textbf{III}}  & {(genus 0 curve with two nodes)} & 
\scalebox{0.95}{%
$\begin{tikzcd}[column sep=0.6cm, row sep=0.6cm, ampersand replacement=\&]
    \vcenter{\hbox{\usebox{\picThreeDaTL}}} 
    \arrow[r, "A", shift left] 
    \arrow[in=173, out=187, loop, min distance=6mm, "A / C" ]
      \& \vcenter{\hbox{\usebox{\picThreeAa}}} 
    \arrow[l, "A", shift left] 
    \arrow[r, "A / B", shift left] 
    \arrow[in=70, out=110, loop, min distance=4mm, "D"]
      \& \vcenter{\hbox{\usebox{\picThreeDaBL}}} 
    \arrow[l, "A / B", shift left]
        \arrow[in=70, out=110, loop, min distance=4mm, "A"]
        \arrow[r, "A / B", shift left] 
    \& \vcenter{\hbox{\usebox{\picThreeBaMR}}}
    \arrow[l, "A", shift left]
    \arrow[r, "A", shift left]
     \& \vcenter{\hbox{\usebox{\picThreeBaTR}}}
    \arrow[l, "A", shift left] 
    \arrow[in=-10, out=10, loop, min distance=6mm, "A" ]
\end{tikzcd}
$%
} \\
\hline

{\textbf{IV}} & {\smash{\parbox[t]{5cm}{(two genus 0 curves intersecting at three points)}}} & 
\scalebox{0.95}{%
$\begin{tikzcd}[column sep=0.6cm, ampersand replacement=\&]
  \vcenter{\hbox{\usebox{\picFourTR}}} 
    \arrow[r, "A", shift left=1ex]     
    \arrow[in=173, out=187, loop, min distance=6mm, "A"]
  \& \vcenter{\hbox{\usebox{\picFourTL}}} 
    \arrow[l, "A", shift left=1ex] 
    \arrow[r, "B",  shift left=1ex] 
    \arrow[in=70, out=110, loop, min distance=4mm, "D"]
  \& \vcenter{\hbox{\usebox{\picFourBL}}} 
    \arrow[l, "B", shift left=1ex] 
    \arrow[r, "A", shift left=1ex] 
  \& \vcenter{\hbox{\usebox{\picFourBR} }}
    \arrow[l, "A", shift left=1ex]
        \arrow[in=-10, out=10, loop, min distance=6mm, "A" ]
\end{tikzcd}$%
} \\
\hline

{\textbf{V}} & {\smash{\parbox[t]{5cm}{(two elliptic curves intersecting at one point)}}} & 
\scalebox{0.95}{%
$\begin{tikzcd}[column sep=0.6cm, ampersand replacement=\&]
  \vcenter{\hbox{\usebox{\picFiveTR}}}
    \arrow[in=170, out=190, loop, min distance=6mm, "B" ]
    \arrow[r, "A", shift left=1ex,  yshift=-2pt]
    \arrow[d, "A", shift left=1ex]
  \& \vcenter{\hbox{\usebox{\picFiveTL}}} 
    \arrow[in=-10, out=10, loop, min distance=6mm, "A / D"]
    \arrow[dl, "C", shift right= -0.5ex]
    \arrow[l, "B", shift left= 1ex, yshift=2pt] \\
    \vcenter{\hbox{\usebox{\picFiveBL}}} 
    \arrow[ur, "C", shift right= -0.5ex]
    \arrow[u, "A ", shift left=1ex]
    \arrow[in=173, out=187, loop, min distance=6mm, "A / C"]
\end{tikzcd}$%
} \\
\hline


{\textbf{VI}}
&
{\smash{\parbox[t]{5cm}{%
(an elliptic curve and a genus 0 curve with a node intersecting at one point)%
}}}
&
\scalebox{0.95}{%
$\begin{tikzcd}[
  column sep=0.6cm,
  row sep=0.5cm,
  ampersand replacement=\&
]
  \vcenter{\hbox{\usebox{\picSixTopLeft}}}
    \arrow[d, "A", shift left=1ex]
    \arrow[dr, "C", shift right=-0.5ex]
    \arrow[in=173, out=187, loop, min distance=6mm, "A"]
  \&
  \vcenter{\hbox{\usebox{\picSixTopRight}}}
    \arrow[d, "A", shift left=1ex]
    \arrow[in=-10, out=10, loop, min distance=6mm, "A"]
  \&
  \vcenter{\hbox{\usebox{\picSixTop}}}
    \arrow[dl, "B", shift right=-0.5ex]
    \arrow[r, "A", shift left=1ex]
  \&
  \vcenter{\hbox{\usebox{\picSixBottomLeft}}}
    \arrow[l, "A", shift left=1ex]
    \arrow[in=-10, out=10, loop, min distance=6mm, "A"]
  \\

  \vcenter{\hbox{\usebox{\picSixRight}}}
    \arrow[u, "A", shift left=1ex]
    \arrow[r, "A", shift left=1ex]
  \&
  \vcenter{\hbox{\usebox{\picSixCenter}}}
    \arrow[l, "B", shift left=1ex]
    \arrow[u, "A", shift left=1ex]
    \arrow[ul, "C", shift right=-0.5ex]
    \arrow[ur, "B", shift right=-0.5ex]
    \arrow[r, "A", shift left=1ex]
    \arrow[d, "A", shift left=1ex]
  \&
  \vcenter{\hbox{\usebox{\picSixLeft}}}
    \arrow[l, "A", shift left=1ex]
    \arrow[in=-10, out=10, loop, min distance=6mm, "A"]
  \&
  \\

  \&
  \vcenter{\hbox{\usebox{\picSixBottom}}}
    \arrow[u, "A", shift left=1ex]
    \arrow[in=-10, out=10, loop, min distance=6mm, "A"]
  \&
  \&
\end{tikzcd}$%
}
\\
\hline

{\textbf{VII}}  & {\smash{\parbox[t]{5cm}{(two genus 0 curves with a node (each) intersecting at one point)}}} & 
\scalebox{0.95}{
$ \begin{tikzcd}[column sep = 0.6cm, row sep= 0.5cm, ampersand replacement=\&]
\vcenter{\hbox{\usebox{\picSevenD}}}
\arrow[dr, "A", shift left] 
\arrow[in=173, out=187, loop, min distance=6mm, "A"] \& \vcenter{\hbox{\usebox{\picSevenT}}}
\arrow[d, "A", shift left]
\arrow[in=-10, out=10, loop, min distance=6mm, "A" ]
\& \& \\
\vcenter{\hbox{\usebox{\picSevenL}}} 
\arrow[r, "A", shift left]
\arrow[in=173, out=187, loop, min distance=6mm, "A / D"] 
\& \vcenter{\hbox{\usebox{\picSevenC}}}
\arrow[u, "A", shift left]
\arrow[l, "A", shift left]
\arrow[r, "B", shift left]
\arrow[ul, "A", shift left]
\& \vcenter{\hbox{\usebox{\picSevenR}}} 
\arrow[l, "B", shift left]
\arrow[r, "A", shift left]
\& \vcenter{\hbox{\usebox{\picSevenRR}}} 
\arrow[l, "A", shift left]
\arrow[in=-10, out=10, loop, min distance=6mm, "A" ]\\
\end{tikzcd}
$} \\
\hline
\end{longtable}
\end{landscape}

\newpage 
\restoregeometry

\textwidth 17cm
\textheight 22cm
\oddsidemargin 0cm
\evensidemargin 0cm

\end{document}